\documentclass[11pt, oneside]{amsart}   	
\usepackage{amsmath,amsthm,amscd,amsfonts, amssymb, mathrsfs}
\usepackage{geometry}                		
\usepackage{graphicx}				
\usepackage{amssymb}
\newcommand{\sect}[1]{\section{#1}\setcounter{equation}{0}}

\font\mbn=msbm10 scaled \magstep1
\font\mbs=msbm7 scaled \magstep1
\font\mbss=msbm5 scaled \magstep1
\newfam\mbff
\textfont\mbff=\mbn
\scriptfont\mbff=\mbs
\scriptscriptfont\mbff=\mbss   

\newcommand{\RR}       { \mathbb{R}}

\newcommand{\N}       { \mathbb{N}}
\newcommand{\Z}        {\mathbb{Z}  }

\newtheorem{Th}{Theorem}[section]
\newtheorem{Lm}[Th]{Lemma}
\newtheorem{C}[Th]{Corollary}
\newtheorem{D}[Th]{Definition}
\newtheorem{St}[Th]{Stipulation}
\newtheorem{Stat}[Th]{Statement}
\newtheorem{Prop}[Th]{Proposition}
\newtheorem{R}[Th]{Remark}

\newtheorem{Notation}[Th]{Notation}

\newtheorem*{Con}{Conjecture}
\newtheorem*{Th A}{Theorem A}
\newtheorem*{Th B}{Theorem B}
\begin{document}

\title[Multivariate BV Functions of Jordan-Wiener Type]{Multivariate  Bounded Variation  Functions of Jordan-Wiener Type}
\author{Alexander Brudnyi}
\address{Department of Mathematics and Statistics\newline
\hspace*{1em} University of Calgary\newline
\hspace*{1em} Calgary, Alberta, Canada\newline
\hspace*{1em} T2N 1N4}
\email{abrudnyi@ucalgary.ca}
\author{Yuri Brudnyi}
\address{Department of Mathematics\newline
\hspace*{1em} Technion\newline
\hspace*{1em} Haifa, Israel\newline
\hspace*{1em} 32000}
\email{ybrudnyi@math.technion.ac.il}


\thanks{Research of the first author is supported in part by NSERC}

\begin{abstract}
We introduce and study spaces of multivariate functions of bounded variation generalizing the classical Jordan and Wiener spaces.  
Multivariate generalizations of the Jordan space were given by several prominent researchers but each of them preserved only some special properties of the space used further in few selected  applications. Unlike this the multivariate generalization of the Jordan space presented in this paper preserves all known and reveals some previously unknown properties of the space. These, in turn, are special cases of the basic properties of the introduced spaces proved in the paper. They, in particular, include results on discontinuity sets and pointwise differentiability of the bounded variation functions and their Luzin type and $C^\infty$ approximation. Moreover, the second  part of the paper presents results on Banach structure of function spaces of bounded variation, namely, atomic decomposition and constructive characterization of their predual spaces and then constructive characterization of  preduals of the last ones and following from here the so-called two-stars theorems relating second duals of separable subspaces of ``vanishing variation''  to the (nonseparable) spaces of bounded variation.
\end{abstract}

\date{}

\maketitle

\sect{Introduction}
\subsection{Main Concept (Motivation)} The paper is devoted to the study of a new family of multivariate function spaces generalizing the classical Jordan-Wiener spaces $\{V_p\}_{1\le p\le\infty}$. A generic space of the family is denoted by $\dot V_p^k(Q^d)$, where $k\in\N$, $1\le p<\infty$ and $Q^d$ is a $d$-dimensional closed cube; in what follows, without loss of generality we take 
\begin{equation}\label{eq-n1.1}
Q^d:=[0,1]^d.
\end{equation}
The space consists of bounded on $Q^d$ functions with bounded
$(k,p)$-{\em variation} denoted by $|f|_{V_p^k}$. To introduce and motivate this notion we use an equivalent form of the classical $p$-variation given for a function $f\in\ell_\infty [0,1]$ by
\begin{equation}\label{eq-n1.2}
|f|_{V_p}:=\sup_{\pi}\left(\sum_{I\in\pi}{\rm osc}(f;I)^p\right)^{\frac 1 p};
\end{equation}
here $\pi$ is a {\em packing}, i.e., a set of pairwise nonoverlapping\footnote{Two subsets of a topological space are nonoverlapping if their interiors are disjoint.} closed intervals $I\subset [0,1]$ and 
\begin{equation}\label{eq-n1.3}
{\rm osc}(f;I):=\sup_{x,y\in I}|f(x)-f(y)|
\end{equation}
is the {\em oscillation} of $f$ on $I$.

The reader can easily see that \eqref{eq-n1.2} is equivalent to the classical Jordan  definition with ${\rm osc}(f;I)$ replaced by {\em deviation} $\delta(f;I):=|f(b)-f(a)|$, where $I:=[a,b]$ and $\pi$ runs over {\em partitions}, i.e., {\em coverings} of $[0,1]$ by pairwise nonoverlapping intervals.

It seems to be natural to define multivariate spaces of Jordan-Wiener type by taking in \eqref{eq-n1.2} $f\in\ell_\infty(Q^d)$ and $\pi$ being a family of pairwise nonoverlapping  closed subcubes $Q\subset Q^d$. However, assuming boundedness of such variation we obtain the family $\{\ell_p(Q^d)\}_{1\le p<\infty}$ that has no relation to the required  function spaces of bounded variation. 

To overcome this obstacle we enlarge the order of difference in \eqref{eq-n1.3} using a function of a cube $Q\mapsto {\rm osc}_k(f;Q)$, $Q\subset Q^d$, given by
\begin{equation}\label{eq-n1.4}
{\rm osc}_k(f;Q):=\sup_{h>0}\{|\Delta_h^k f(x)|\, :\, x+jh\in Q;\ j=0,\dots, k\};
\end{equation}
here as usual
\begin{equation}\label{eq-n1.5}
\Delta_h^k f:=\sum_{j=0}^k (-1)^{k-j} {k \choose j} f(\cdot +jh).
\end{equation} 
The family of spaces $\dot V_p^k(Q^d)$ obtained in this way  is the main object of our study.
\begin{R}\label{rem1.1}
{\rm The definition of $\dot V_p^k(Q^d)$ is meaningful only in the case of $k\ge\frac{d}{p}$; otherwise, $\dot V_p^k(Q^d)=\ell_p (Q^d)$.
}
\end{R}
To illustrate the results of the present paper we briefly discuss their special cases for  the space $\dot V_1^d(Q^d)$ that may be regarded as a $d$-dimensional analog of Jordan space $V\, (=\dot V_1^1(Q^1)$ in our notation). We justify its role as a ``genuine''  analog of Jordan's space for $d\ge 2$ by comparing its properties presented below with those for $d=1$.

 In what follows, $\dot V_1^d(Q^d)$ is denoted for brevity by $\dot V(Q^d)$.\smallskip
 
 \noindent (a) For every $d\ge 1$, functions of $\dot V(Q^d)$ have at most countable sets of discontinuities, see Theorem \ref{te2.1.3}.\smallskip
 
 \noindent (b) For every $d\ge 1$, functions of  $\dot V(Q^d)$ are $d$-differentiable in the Peano sense\footnote{See the corresponding definition in Section~2.2.} almost everywhere (a.e.). 
 
 In particular, this implies for $d=1$ the classical Lebesgue differentiability theorem, see Theorem \ref{te2.2.8}\,(a).\smallskip
 
 \noindent (c) Given  $f\in \dot V(Q^d)$, $d\ge 1$, and $\varepsilon\in (0,1)$ there is a function $f_\varepsilon\in C^{d}(Q^d)$ such that
 \[
 {\rm mes}_d\{x\in Q^d\, :\, f(x)\ne f_\varepsilon (x)\}<\varepsilon,
 \]
 see Theorem \ref{teo2.2.5}.
 
 \noindent (d) Let $AC(Q^d)$, $d\ge 1$, be a closed subspace of  $\dot V(Q^d)$ consisting of functions $f\in\ell_\infty(Q^d)$ satisfying
 \[
 \lim_{|\pi|\rightarrow 0}\sum_{Q\in\pi}{\rm osc}_k(f;Q)=0;
 \]
 here $|\pi|:=\sum_{Q\in\pi}|Q|$ and $|S|$ is a $d$-measure of $S\subset\RR^d$.
 
 Then it is true that
 \[
 AC(Q^d)=\dot W_1^d(Q^d),
 \]
 where hereafter $\dot W_p^k(Q^d)$ stands for the homogeneous Sobolev space of order $k$ over $L_p(Q^d)$.
 
 Moreover, every $f\in AC(Q^d)$ is restored by some integral operator up to an additive term being a polynomial of degree $d-1$  via the distributional gradient $\nabla^d f:=(D^\alpha f)_{|\alpha|=d}$, see Corollary \ref{cor9.7}.
 
 For $d=1$, this gives the so-called Fundamental Theorem of Calculus by Lebesgue-Vitali.\smallskip
 
 \noindent (e) Let $N\dot V(Q^d)$, $d\ge 1$, be the normalized part of $\dot V(Q^d)$ so that
 \[
 \dot V(Q^d)=N\dot V(Q^d)\oplus\ell_1(Q^d),
 \]
 see Section~2.1 below.
 
 Then it is true that
 \[
 N\dot V(Q^d)=B\dot V^d(Q^d),
 \]
 where  $B\dot V^k(Q^d)$ is the Sobolev space of order $k$ whose distributional $k$-th derivatives are bounded Borel measures, see Theorem \ref{teo2.2.11}\,(c).

 For $d=1$, this gives the classical Hardy-Littlewood theorem \cite{HL-28} as $B\dot V^1(Q^1)$ coincides with Lipschitz space over $L_1(Q^1)$.\smallskip
 
 \noindent (f) If a sequence $\{f_n\}_{n\in\N}\subset  \dot V(Q^d)$, $d\ge 1$, is such that
 \[
 \sup_{n\in\N}|f_n|_{V(Q^d)}<\infty,
 \]
 then for some sequences $J\subset\N$ of integers and $\{p_n\}_{n\in J}$ of polynomials of degree $d-1$ on $\RR^d$ the sequence $\{f_n+p_n\}_{n\in J}$ pointwise converges to some $f\in \dot V(Q^d)$ and, moreover,
 \[
 |f|_{V(Q^d)}\le\varliminf_{n\in J}|f_n|_{V(Q^d)}.
 \]
 
For $d=1$, this gives the classical Helly selection theorem.
 \smallskip
 
 The variety of useful properties of the space $\dot V(Q^d)$ (more generally,  spaces $\dot V_p^k(Q^d)$) allows their applicability to various fields of modern analysis including multivariate harmonic analysis, $N$-term approximation, the real interpolation of smoothness spaces. The corresponding results will be presented in the forthcoming papers.
 \begin{R}\label{rem1.2}
 {\rm There are two more spaces of the family, namely, $\dot V_d^d(Q^d)$ and $\dot V_d^1(Q^d)$ which for $d=1$ coincide with the space $V[0,1]$. However, they  do not possess almost all of the above formulated properties for $\dot V(Q^d)$ (e.g., if $d\ge 2$, then $\dot V_d^d(Q^d)$ contains nowhere differentiable functions).
 }
 \end{R}
\subsection{Historical Remarks} An important role of the space $V$ in univariate function theory attracted several prominent researchers to the problem of its multivariate generalization. The introduced notions single out multivariate function spaces preserving only some {\em particular} properties of the Jordan space. These properties, in turn, were used to solve a specific problem of multivariate analysis. 

In this paper we briefly discuss only those works that have been influenced further development of the multivariate theory.\smallskip

\noindent (a) {\em Vitali variation} (associated with the names of Vitali \cite{V-05}, Lebesgue \cite{Le-10}, Vall\'{e}e Poussen \cite{VP-16}).

Let $[x,y]\subset\RR^d$, $x=(x_1,\dots, x_d)$, $y=(y_1,\dots , y_d)$, be a $d$-{\em interval}, i.e., the set
\[
[x,y]:=\{z=(z_1,\dots, z_d)\in\RR^d\, :\, x_i\le z_i\le y_i,\  i=1,\dots, d\}.
\]
For a function $f\in\ell_\infty(I)$, $I:=[x,y]$, we define the {\em deviation} 
\begin{equation}\label{eq-n1.6}
\delta(f;I)=\sum_{j_1=0}^1\cdots\sum_{j_d=0}^1 (-1)^{j_1+\cdots +j_d}f(y_1+j_1(x_1-y_1),\dots, y_d+j_d(x_d-y_d));
\end{equation}
in particular, for $d=2$
\[
\delta(f;I)=f(y_1,y_2)-f(y_1,x_2)-f(x_1,y_2)+f(x_1,x_2).
\]

The {\em Vitali variation} is given for $f\in\ell_\infty(Q^d)$ by
\begin{equation}\label{eq-n1.7}
{\rm var}_v f:=\sup_\pi \left\{\sum_{I\in\pi} |\delta(f;I)|\right\},
\end{equation}
where $\pi$ runs over all families of pairwise nonoverlapping $d$-intervals in $Q^d$.

The space of functions on $Q^d$ with bounded Vitali variations is denoted by $V_v(Q^d)$.
\begin{R}\label{rem1.3}
{\rm (a) In the original definition, see, e.g., \cite{AC-33}, $\pi$ runs over all {\em partitions} of $Q^d$ by such $d$-intervals. Since
\[
\delta(f;I)=\sum_{I'\in\pi}\delta(f;I'),
\]
where $\pi$ is a partition of $I$ in nonoverlapping $d$-intervals $I'$ by hyperplanes parallel to the coordinate ones, this definition is equivalent to \eqref{eq-n1.7}.\smallskip

\noindent (b) The deviations in \eqref{eq-n1.7} can be clearly replaced by {\em oscillations}; here
\[
{\rm osc}(f;I):=\sup_{I'\subset I}|\delta(f;I')|,
\]
where supremum is taken over all $d$-subintervals of $I$.
}
\end{R}
Unlike affinity of their definitions the structure of functions from $V_v(Q^d)$ is essentially poorer than that for functions from $\dot V(Q^d)$. In particular, the former contains nonmeasurable functions and nowhere differentiable functions.

Nevertheless, the affinity leads to the following continuous embedding
\begin{equation}\label{eq-n1.8}
(\dot V\cap C)(Q^d)\subset V_v(Q^d).
\end{equation}

It remains to present the only application of Vitali variation to prove the existence of the Riemann-Stiltjes type multiple integral $\int_{Q^d}f\, dg$ for $f\in C(Q^d)$ and $g\in V_v(Q^d)$. i.e., the existence of the limit
\[
\lim_{\Pi}\sum_{I\in\pi} f(x_I)\delta(g;I),
\]
where $\pi$ is a partition of $Q^d$ in $d$-intervals by hyperplanes parallel to the coordinate ones, $x_I\in I$ and $\Pi=\{\pi\}$ is a net of such partitions with meshes $\sup_{I\in\pi}|I|\rightarrow 0$.\medskip

\noindent (b) {\em Hardy-Krause variation}.

This notion denoted by ${\rm var}_h f$ is defined via Vitali variation as follows.

Let $\{e^i\}_{1\le i\le n}$ be the standard basis of $\RR^d$ and $\hat x\in Q^d$, $\omega\subset\{1,\dots, d\}$ be fixed. Then a {\em partial function} of $f\in\ell_\infty(Q^d)$  defined by $\hat x, \omega$ is given by
\[
f_{\hat x,\omega}: (x_i)_{i\in\omega}\mapsto f\left(\sum_{i\in\omega}x_ie^i+\sum_{i\not\in\omega}\hat x_i e^i\right);
\] 
this is clearly defined on the unit cube $Q^\omega\cong [0,1]^{|\omega|}$ of dimension $|\omega|$ in the subspace of $\RR^d$ generated by vectors $e^i$, $i\in\omega$.

Now Hardy-Krause variation is given for $f\in\ell_\infty(Q^d)$ by
\begin{equation}\label{eq-n1.9}
{\rm var}_h f:=\sum_{\omega}{\rm var}_v f_\omega ,
\end{equation}
where $\omega$ runs over all nonempty subsets of $\{1,\dots, d\}$.

The space of functions of bounded variation \eqref{eq-n1.9} is denoted by $V_h(Q^d)$.

The functions of this subspace of $V_v(Q^d)$ are of essentially better structure. In particular, every $f\in V_h(Q^d)$ has one-sided  limits in each variable and is differentiable a.e. The former is the matter of definition and the latter is proved in  \cite{BH-32} for $d=2$ (the proof can be  easily extended to all $d$).

The main application of Hardy-Krause variation is a generalization of
Dirichlet-Jordan convergence criterion for multiple Fourier series, see \cite{H-06} for $d=2$ and \cite{MT-50} for $d>2$. For some modern development of these results in multivariate harmonic analysis, approximation theory and numerical computations see, e.g., \cite{Te-015} and references therein.\medskip

\noindent (c) {\em Tonelli variation}.\smallskip

Let $f\in \ell_\infty(Q^d)$ and $f_{x^i}: x_i\mapsto f(x)$, $0\le x_i\le 1$, be a univariate partial function with $x^i:=\sum_{j\ne i} x_j e^j$ being fixed, $1\le i\le d$.

Tonelli variation for this $f\in\ell_\infty(Q^d)$ is given by
\begin{equation}\label{eq-n1.10}
{\rm var}_t\, f:=\sum_{i=1}^d\int_{Q^{d-1}}({\rm var}\, f_{x^i})\, dx^i;
\end{equation}
here the (Jordan) variations are assumed to be Lebesgue integrable.

This notion is introduced to characterize continuous functions on $Q^d$ whose graphs are rectifiable, see \cite{T-26} for $d=2$ and proof in \cite[Ch.\,5]{S-37} that is easily extended to $d>2$.

For integrable functions, Tonelli variation loses its geometric meaning, see \cite{AC-34}. The corresponding adaptation to this case was given by \cite{Ce-36}; the modern version exploiting theory of distributions is presented in the books \cite{Gi-84} and \cite{AFP-00}. The corresponding variation for  $f\in L_1(Q^d)$ is given by
\begin{equation}\label{eq-n1.11}
\underset{Q^d}{\rm var}\, f:=\sum_{i=1}^d \underset{Q^d}{\rm var}\,\partial_i f,
\end{equation}
where the distributional partial derivatives are assumed to be finite Borel measures. 

The corresponding space is denoted by $BV(Q^d)$.

This definition is given in essence by De Giorgi \cite{DG-54}; its equivalence to that of Tonelli-Cesari is proved in the papers  \cite{FKP-57}. In particular, for $f\in C(Q^d)$
\[
{\rm var}_t f\approx {\rm var} f
\] 
with constants of equivalence depending only on $d$. 

Nevertheless, there is some connection of this space with a space of the family $\{\dot V_p^k(Q^d)\}$. Actually, using the equivalent definition of $\underset{Q^d}{\rm var} f$ given for $f\in L_1(Q^d)$, $d\ge 2$, by
\[
\underset{Q^d}{\rm var}\, f=\sup_{\pi}\sum_{Q\in \pi}\left(\int_Q |f-f_Q|^{\frac{d}{d-1}}\,dx\right)^{\frac{d-1}{d}},
\]
where $f_Q:=\frac{1}{|Q|}\int_Q f\,dx$,
see \cite[Sec.\,1.3.4]{BB-18}, and then applying the H\"{o}lder inequality we have
\[
\underset{Q^d}{\rm var}\, f\le c(d)|f|_{V_d^1(Q^d)}.
\]

Finally, we present one more variation given by Kronrod \cite{Kr-50} for $d=2$ and by Vitushkin \cite{Vi-55} for all $d$.

It is given only for continuous functions on $Q^d$ by
\begin{equation}\label{eq-n1.12}
{\rm var}_{kv} f:=\int_{\RR} H_{d-1}\{x\in Q^d\, :\, f(x)=t\}\, dt,
\end{equation}
where $H_{d-1}$ is the $(d-1)$-Hausdorff measure.

For $d=1$, this coincides with Jordan variation of $f$ by the Banach indicatrix theorem and for $d>1$
\[
{\rm var}_{kv}f\approx {\rm var}_t f
\]
with constants of equivalence depending only on $d$, see \cite{Kr-50}.

In fact, the definitions in the cited papers include also the definition of variations related to the sections of function graphs by planes of dimensions $1\le s\le d-1$.  For the applications  of these characteristics to the geometry of sets, approximation theory and the 13th Hilbert problem, see \cite{Vi-004}.

\subsection{Basic Definitions}  
In the presentation of the main objects of our study and everywhere below, we use the following:

\begin{Notation}\label{not1.2.2}
{\rm  $\N$, $\Z$, $\RR$ denote the sets of natural, integer and real numbers, respectively.\smallskip

 $\RR^d$ is the $d$-dimensional Euclidean space of vectors $x=(x_1,\dots, x_d), y, z$, etc.\smallskip

 If $S\subset\RR^d$, then $\bar S$ denotes its closure and $\mathring S={\rm int}\,S$ its interior.\smallskip

 The $d$-dimensional Lebesgue measure of $S$ is denoted by $|S|$. The characteristic function (indicator) of $S$ is denoted by $\chi_S$.\smallskip

 As above, $Q^d$ denotes the unit cube $[0,1]^d$. $Q,Q'$, etc. denote {\em subcubes} of $\RR^d$ homothetic to $Q^d$ (named further {\em cubes} or {\em subcubes}); $Q_r(x)$ denotes a cube centered at $x$ of sidelength $2r$.\smallskip

$\pi,\pi',\pi_i$, etc. denote {\em packings}, i.e., families of nonoverlapping cubes $Q\subset\RR^d$.\smallskip

 $\Pi(S)$ (briefly, $\Pi$ for $S=Q^d$) denotes the set of packings in $S$ of {\em finite cardinality} contained in $S$.\smallskip

 ${\mathcal P}_k^d$ denotes the linear space of polynomials in $x\in\RR^d$ of  degree $k$, i.e., the linear hull of monomials $x^\alpha$, $|\alpha|\le k$, where
\[
x^\alpha:=\prod_{i=1}^d x_i^{\alpha_i},\quad |\alpha|:=\sum_{i=1}^d\alpha_i,\quad \alpha\in\Z_+^d.
\]

$\ell_p(S)$, $1\le p\le\infty$, $S\subset\RR^d$, denotes the Banach space of functions $f:S\rightarrow\RR$ defined by norm
\[
\|f;S\|_p:=\left\{\sum_{x\in S}|f(x)|^p\right\}^{\frac 1 p};
\]
in particular,
\[
\|f;S\|_\infty:=\sup_S |f|.
\]

All $f\in\ell_p(S)$ are clearly bounded on $S$ and their supports
${\rm supp}\, f:=\{x\in S\, :\, f(x)\ne 0\}$
are at most countable for $p<\infty$.\smallskip

 $\ell_p^{{\rm loc}}(\RR^d)$ consists of functions on $\RR^d$ whose traces to every compact subset $K$ belong to $\ell_p(K)$.\smallskip
 
 $C(S)$, $S\subset\RR^d$, denotes the Banach space of the uniformly continuous functions $f:S\rightarrow\RR$ equipped with norm 
$\|f\|_{C(S)}:=\sup_S |f|$.\smallskip

 $C^\infty(S)$ denotes the linear space of traces $f|_S$ of functions $f\in C^\infty(\RR^d)$.\smallskip
 
  Let $X,Y$ be linear (semi-)\,normed vector spaces. We write 
\begin{equation}\label{2.2.7}
X\hookrightarrow Y
\end{equation}
if there is a linear continuous injection of $X$ into $Y$, and replace $\hookrightarrow $ by $\subset$ if $X$ is a linear subspace of $Y$ and the natural embedding operator $X\rightarrow Y$ is continuous.

Further, we say that these spaces are {\em isomorphic} and write
\begin{equation}\label{2.2.8}
X\cong Y
\end{equation}
if $X\hookrightarrow Y$ and $Y\hookrightarrow X$ and composition $X\rightarrow X$ of these continuous injections is the identity map; moreover, we write
\[
X=Y
\] 
if, in addition, they coincide as linear spaces, hence, have equivalent (semi-)\,norms. 

Moreover, spaces $X$ and $Y$  are said to be {\em isometrically isomorphic} if the injections in \eqref{2.2.8} are of norm $1$.
We write in this case 
\begin{equation}\label{2.2.9}
X\equiv Y.
\end{equation}

Finally, we denote by $B(X)$ the closed unit ball of a (semi-) normed space $X$, i.e.,
\begin{equation}\label{2.2.10}
B(X):=\{x\in X\, :\, \|x\|_X\le 1\}.
\end{equation}
 }
\end{Notation}

Now we present the basic notions of the paper.

\subsubsection{Local polynomial approximation} 
In the sequel, we exploit an essential more suitable for applications definition of $(k,p)$-variation. The main point here is the replacement  of 
${\rm osc}_k$ in the initial definition by a function of cube denoted by $E_k(f;\cdot)$ and called {\em local polynomial approximation}, see Definition \ref{def1.2.3} below. The relation between these two set-functions is given for $I\subset\RR$ by the equivalence \cite{Wh-59}
\begin{equation}\label{1.2.5}
E_k(f;I)\approx {\rm osc}_k(f;I),
\end{equation}
where the constants of equivalence depend only on $k$.\\
The general result of this kind proved below, see Theorem \ref{teor2.1.1}, shows that these two definitions of $(k,p)$-variation are equivalent. However, the second definition is more suitable for applications since it relies on various techniques of local approximation theory and so is easier to apply to the study of the Banach structure of $V_p^k(Q^d)$ spaces and their elements. (For instance, the fact that $(k,p)$-variation of $f\in \ell_\infty(Q^d)$ is zero iff $f$ is a polynomial of degree $k-1$ trivially follows from the second definition but is highly nontrivial in the framework of the first one.)

\begin{D}\label{def1.2.3}
Local polynomial approximation of order $k\in\N$ is a function
\[
E_k: (f,S)\mapsto\RR_+,\quad f\in\ell_\infty^{{\rm loc}}(\RR^d),\ S\subset\RR^d,
\]
given by 
\begin{equation}\label{1.2.11}
E_k(f;S):=\inf_{m\in\mathcal P_{k-1}^d}\|f-m;S\|_\infty.
\end{equation}
\end{D}
This function possesses several important properties, see, in particular, \cite[\S\,2]{Br1-94}, that will be used in the forthcoming proofs and mostly are proved there.
\subsubsection{(k,p)-variation} Using the previous notion we introduce a set-function ${\rm var}_p^k(f;S)$, $k\in\N$, $1\le p<\infty$, called $(k,p)$-{\em variation}.
\begin{D}\label{def1.2.4}
$(k,p)$-variation of a locally bounded function $f\in\ell_\infty^{{\rm loc}}(\RR^d)$ on a bounded set $S\subset\RR^d$ is given by
\begin{equation}\label{1.2.12}
{\rm var}_p^k(f;S):=\sup_{\pi\in\Pi(S)}\left(\sum_{Q\in\pi}E_k(f;Q)^p\right)^{\frac 1 p};
\end{equation}
this equals $0$ if $\mathring S=\emptyset$.
\end{D}
It is the matter of definition to verify the next properties of the object introduced.
\begin{Prop}\label{prop1.7}
 {\rm (\underline{Subadditivity})} If $\{S_i\}_{i\in\N}$ is a sequence of disjoint subsets in $\RR^d$, then
\[
\left(\sum_{i=1}^\infty\, {\rm var}_p^k(f;S_i)^p\right)^{\frac 1 p}\le {\rm var}_p^k\left(f;\bigcup_{i=1}^\infty\, S_i\right).
\]
{\rm (\underline{Lower semicontinuity})} If a sequence $\{f_i\}_{i\in\N}\subset\ell_\infty(S)$ converges in this space to a function $f$, then
\[
{\rm var}_p^k(f;S)\le\varliminf_{i\rightarrow\infty} {\rm var}_p^k (f_i;S).
\]
{\rm (\underline{Monotonicity})} The function
\[
(k,p,S)\mapsto {\rm var}_p^k(f;S)
\]
is nondecreasing in $S$ and nonincreasing in $k,p$.
\end{Prop}
\subsubsection{$V_p^k$ spaces} Now we consider ${\rm var}_p^k$ as a function of $f$. Using again its definition we have the following:
\begin{Prop}\label{prop1.8}
Let $S\subset\RR^d$  be the closure of a domain (open connected set). Then the function
\[
f\mapsto {\rm var}_p^k(f;S)\in [0,\infty],\quad f\in\ell_\infty^{\rm{loc}}(S),
\]
satisfies properties of a seminorm and equals $0$ iff $f\in\mathcal P_{k-1}^d|_S$.
\end{Prop}
Now we present the main object of our study.
\begin{D}\label{def-n1.9}
$\dot V_p^k(Q^d)$ is a linear space of functions $f\in\ell_\infty(Q^d)$ whose $(k,p)$-variation
\begin{equation}\label{1.2.14}
|f|_{V_p^k}:={\rm var}_p^k(f;Q^d),\qquad f\in \dot V_p^k(Q^d).
\end{equation}
is finite.
\end{D}
By Proposition \ref{prop1.8} $(\dot V_p^k(S), |\cdot|_{V_p^k})$ is a  seminormed space with 
$\mathcal P_{k-1}^d|_{Q^d}$ being its null-space. The standard argument shows that $\dot V_p^k(S)$ is complete.

\begin{St}\label{stip1.2.6}
{\rm Since $Q^d:=[0,1]^d$ is fixed, throughout the paper we remove it from notations writing, e.g., $\dot V_p^k$, $C^\infty$, $\mathcal P_{k-1}^d$ instead of $\dot V_p^k(Q^d)$, $C^\infty(Q^d)$, $\mathcal P_{k-1}^d|_{Q^d}$. 

In some cases, it will be more appropriate to use a seminorm on $\dot V_p^k$ defined by replacing in \eqref{1.2.12} $E_k(f;Q)$ by ${\rm osc}_k(f;Q)$, $Q\in\pi$. The latter is given for $f\in\ell_\infty^{{\rm loc}}(\RR^d)$ and a bounded subset $S\subset\RR^d$ by
\begin{equation}\label{1.2.17a}
{\rm osc}_k(f;S):=\sup_{h\in\RR^d}\{|\Delta_h^k f(x)|\, :\, x+jh\in S,\ j=0,\dots, k\},
\end{equation}
cf. definition \eqref{eq-n1.4}.
}
\end{St}
Our work is divided into two parts where  the first one studies pointwise behaviour of $\dot V_p^k$ functions while the second is focused on the duality structure of a Banach space $V_p^k$ associated to $\dot V_p^k$.
Among different  equivalent presentations of this space we prefer the quotient of $\dot V_p^k$ by its null-space.
Hence, we define an associated to $\dot V_p^k$ Banach space by setting
\begin{equation}\label{1.2.17}
V_p^k:=\dot V_p^k/\mathcal P_{k-1}^d
\end{equation}
and denoting  by $\|\cdot\|_{V_p^k}$ its (factor)-norm.

If $\mathring f\in V_p^k$ is a factor-class $\{f\}+\mathcal P_{k-1}^d$, $f\in\dot V_p^k$, then clearly
\[
\|\mathring f\|_{V_p^k}=|f|_{V_p^k}.
\]
By this reason, we call elements of $V_p^k$ {\em functions} and denote them by $f,g$, etc. 

In both parts of the work, the essential role plays a numerical characteristic of $\dot V_p^k$ given by
\begin{equation}\label{1.2.18}
s(\dot V_p^k):=\frac d p
\end{equation}
and called {\em smoothness}.

For instance, a function $f\in\dot V_p^k$ satisfies the Lipschitz condition of order $s\, (:=s(\dot V_p^k))$ a.e. if $0<s<k$, has the $k$-th Peano  differential a.e. if $s=k$ and equals  a polynomial of degree $k-1$ outside of at most countable subset of $Q^d$ if $s>k$, see Theorems \ref{te2.1.3},
\ref{te2.2.8} below.

Moreover, in the second part we essentially use a closed separable subspace\footnote{As we will see, space $\dot V_p^k$ is nonseparable.} of $\dot V_p^k$ denoted by $\dot{\textsc{v}}_p^k$
(respectively, $\textsc{v}_p^k\subset V_p^k$) given  by
\begin{equation}\label{eq-n1.20}
\dot{\textsc{v}}_p^k:={\rm clos}(C^\infty\cap\dot V_p^k,\dot V_p^k).
\end{equation}

\subsection{Outline of the Paper} The paper is organized as follows.

Section 2 devoted to the formulations of the main results is composed of four subsections.

Subsection~2.1 contains results describing the structure of $\dot V_p^k$ function spaces.

Subsection~2.2 presents several results describing relations between $\dot V_p^k$ and the Lipschitz spaces over $\ell_\infty$ and $L_p$ of order $s:=\frac d p \in (0,k]$. These include Luzin's type approximation of $\dot V_p^k$ functions by Lipschitz and $C^k$ functions, the Lebesgue type theorem on their pointwise differentiability a.e. of order $s$ and the Hardy-Littlewood type theorem on two-sided embeddings connecting $\dot V_p^k$ with Lipschitz-Besov spaces over $L_p$. As a consequence, the latter gives a multivariate generalization of the Lebesgue-Vitali Fundamental Theorem of Calculus.

Subsection~2.3 introduces a predual to $V_p^k$ Banach space denoted by $U_p^k$ and describes its Banach structure and the relation $(U_p^k)^*\equiv V_p^k$.

Subsection~2.4 contains the so-called two-stars theorem relating the second dual of the space $\textsc{v}_p^k$ to the nonseparable space $V_p^k$. The formulations preceding to that result describe measure theoretic properties of some dense subspace of $U_p^k$ and following  from them and the F. Riesz representation theorem the relation $(\textsc{v}_p^k)^*\cong U_p^k$.

The remaining sections contain the proofs of the formulated results.
\sect{Formulation of the Main Results}
\subsection{Structure Properties of $\dot V_p^k$ Functions} Our first result shows that $\dot V_p^k$ functions initially assumed to be only bounded (not necessarily measurable) have much finer structure.
\begin{Th}\label{te2.1.3}
Let $f$ be a $\dot V_p^k$ function of smoothness $s$. Then the following is true.

(a) If $d=1$, then $f$ has one-sided limits at each $f\in [0,1]$.\smallskip

(b) If $d\ge 2$, then $f$ has a limit at each point $x\in Q^d$.\smallskip

In both cases, $f$ has at most countable set of discontinuity points.\smallskip

(c) If $s>k$, then the vector space $\dot V_p^k$ is  the direct sum of subspaces $\ell_p$ and $\mathcal P_{k-1}^d$.
\end{Th}

Assertion (c) implies that for $s>k$ the Banach spaces $V_p^k$ and $\ell_p$ are isomorphic. Since this case is not of interest, we assume in the sequel that $\dot V_p^k$ satisfies the condition
\begin{equation}\label{2.1.3}
s(\dot V_p^k):=\frac{d}{p}\le k.
\end{equation}

Now let $f\in\dot V_p^k$ and $\hat f$ be given by
\[
\hat f(x):=\lim_{y\rightarrow x} f(y),\quad x\in Q^d,
\]
for $d\ge 2$, and by
\[
\hat f(x):=\left\{
\begin{array}{ccc}
\displaystyle\frac 1 2 \left(f(x^-)+f(x^+)\right),&0<x<1,\\
\\
f(0^+),&x=0,\\
\\
f(1^-),&x=1,
\end{array}
\right.
\]
for $d=1$.

\begin{C}\label{cor2.1.3}
{\rm (a)} The linear map $P$ on $\dot V_p^k$ sending $f$ to $\hat f$ is a projection of norm $1$ on a closed subspace denoted by $N\dot V_p^k\subset\dot V_p^k$. 

 By the definition, $N\dot V_p^k=C\cap\dot V_p^k$ if $d\ge 2$ and consists of  functions $f\in\dot V_p^k$ such that $f(x)=\frac{1}{2}\bigl(f(x^-)+f(x^+)\bigr)$, $x\in (0,1)$, if $d=1$.\smallskip

\noindent {\rm (b)} The kernel of $P$ coincides as a linear space with $\ell_p$ and, moreover, 
\[
|f|_{V_p^k}\le \|f\|_p\le 2|f|_{V_p^k},\quad f\in {\rm ker}(P).
\]
In particular, the space $({\rm ker}(P),|\cdot |_{V_p^k})$ is Banach.\smallskip

\noindent {\rm (c)} $\dot V_p^k$ is isomorphic to the direct sum $N\dot V_p^k\oplus\ell_p$.
\end{C}
Hereafter, the direct sum of (semi-) normed spaces $X$, $Y$ is defined by a (semi-) norm given for $(x,y)\in X\times Y$ by 
\[
\|(x,y)\|:=\|x\|_X+\|y\|_Y.
\]

Now we present two results on $C^\infty$ approximation of $\dot V_p^k$ functions.  The first of them concerns approximation in the weak$^*$ topology of $\dot V_p^k$ induced by duality $\ell_1^*=\ell_\infty$.
In its formulation and proof we use integrability on $Q^d$ of functions $f\in\dot V_p^k$, see Theorem \ref{te2.1.3}.
\begin{Th}\label{te2.1.4}
For every function $f\in\dot V_p^k$ there is a  sequence of $C^\infty$ functions $\{f_n\}_{n\in\N}$ such that
\begin{equation}\label{2.1.4}
\varlimsup_{n\rightarrow\infty}\|f_n\|_\infty\le 3\|f\|_\infty \quad {\rm and}\quad \varlimsup_{n\rightarrow\infty}|f_n|_{V_p^k}\le 5|f|_{V_p^k},
\end{equation}
and, moreover, for every finite signed Borel measure on $Q^d$,
\begin{equation}\label{2.1.5}
\lim_{n\rightarrow\infty}\int_{Q^d}(f-f_n)\,d\mu=0.
\end{equation}
\end{Th}
Aa a consequence of the theorem we obtain the following  generalization of the classical Helly selection theorem.
\begin{C}\label{helly}
For every bounded sequence $\{f_n\}_{n\in\N}\subset  \dot V_p^k$ there are sequences $J\subset\N$ and $\{p_n\}_{n\in J}\subset\mathcal P_{k-1}^d$ such that the sequence $\{f_n+p_n\}_{n\in J}$ pointwise converges to some $f\in \dot V_p^k$ and, moreover,
 \[
 |f|_{V_p^k}\le\varliminf_{n\in J}|f_n|_{V_p^k}.
 \]
 \end{C}

 The next result used in the proof of the second approximation result relates spaces $\dot V_p^k$ and $N\dot V_p^k$ to the space $\dot V_{p\infty}^k$ defined in the very same way as $\dot V_p^k$ but with $\ell_\infty$ replaced by $L_\infty$ and with the respectively modified variation, see Definition \ref{def-n1.9}. In its formulation, $L:\dot V_p^k\rightarrow L_\infty$ is a linear map  sending a function from $\dot V_p^k$ to its equivalence class in $L_\infty$. 
\begin{Th}\label{te2.6}
{\rm (a)} ${\rm range}(L)=\dot V_{p\infty}^k$ and the operator $L:\dot V_{p}^k\rightarrow \dot V_{p\infty}^k$ has norm $1$;\smallskip

\noindent {\rm (b)} ${\rm ker}(L)={\rm ker}(P)=\ell_p\subset\dot V_p^k$.\smallskip

\noindent {\rm (c)} $L$ maps $N\dot V_p^k$ isometrically onto $\dot V_{p\infty}^k$.
\end{Th}

The second approximation result characterizes $\dot V_p^k$ functions admitting $C^\infty$ approximation, i.e., functions of the subspace $\dot{\textsc{v}}_p^k$.
\begin{Th}\label{te2.1.5}
Let a function $f\in\dot V_p^k$ and smoothness $s$ of $\dot V_p^k$ satisfy the condition
\begin{equation}\label{2.1.6}
k>s\, (:=d/p).
\end{equation}
Then  $f\in \dot{\textsc{v}}_p^k$, i.e., $f$ is approximated by $C^\infty$ functions in $\dot V_p^k$, see \eqref{eq-n1.20}, iff
\begin{equation}\label{2.1.7}
\lim_{\varepsilon\rightarrow 0}\sup_{d(\pi)\le\varepsilon}\left(\sum_{Q\in\pi} E_k(f;Q)^p\right)^{\frac 1 p}=0;
\end{equation}
hereafter $d(\pi):=\sup_{Q\in\pi}|Q|$.
\end{Th}
\begin{R}
{\rm In fact, we will see that $C^\infty\subset \dot V_p^k$ and $\dot{\textsc{v}}_p^k\subset C$  for $s\le k$. }
\end{R}
\subsection{Relations Between Multivariate $BV$ Spaces and Lipschitz-Besov Spaces}
Embeddings connected Jordan-Wiener spaces $BV_p$ and Lipschitz spaces over $L_p$ were discovered by Hardy and Littlewood \cite{HL-28} and then intensively studied by many researches, see, in particular, \cite{Ge-54}, \cite{KL-09} and references therein.

Below we present relations of these and some other types whose proofs however require essentially new tools including Sobolev type embeddings, interpolation space results and multivariate generalization of Whitney type inequality.

We begin with the last result where we deal with the definition of  $(k,p)$-variation via {\em $k$-oscillation}, see \eqref{1.2.17a}.
\begin{Th}\label{cor2.1.2}
For every $f\in\ell_\infty^{\rm loc}(\RR^d)$ the following two-sided inequality
\begin{equation}\label{2.1.2}
|f|_{V_p^k}\approx\sup_{\pi\in\Pi}\left\{\sum_{Q\in\pi}{\rm osc}_k(f;Q)^p\right\}^{\frac 1 p}
\end{equation}
holds with constants of equivalence depending only on $k,d$.
\end{Th}

Hereafter $\Pi$ stands for the set of all finite packings in $Q^d$.
\begin{R}\label{rem2.2.2}
{\rm The main point of this result is a generalized Whitney inequality given by
\begin{equation}\label{equat2.1.1}
E_k(f;Q)\le c(k,d)\, {\rm osc}_k(f;Q).
\end{equation}
The result is known for functions from $L_p^{\rm loc}(\RR^d)$, $1\le p\le\infty$, and many other Banach function spaces \cite{Br1-70} but its proof exploits tools of measure theory which is impossible for the derivation of \eqref{2.1.2}.
}
\end{R}
The next result shows that a $V_p^k$ function of smoothness $s\in (0,k]$ ``almost'' coincides with a Lipschitz function of the same smoothness. Here Lipschitz space $\Lambda^{k,s}$ consists of continuous on $Q^d$ functions satisfying for $x, x+kh\in Q^d$ the condition
\begin{equation}\label{e2.2.4}
|\Delta_h^kf(x)|\le c\|h\|^s,
\end{equation}
where $\|h\|:=\max_{1\le i\le d}|h_i|$, $h=(h_1,\dots, h_d)\in\RR^d$.

The equality
\begin{equation}\label{e2.2.5}
|f|_{\Lambda^{k,s}}:=\inf c
\end{equation}
defines a Banach seminorm of $\Lambda^{k,s}$. Moreover, $\Lambda^{k,s}$ can be introduced using continuous derivatives and differences of order at most two. In fact, setting
\begin{equation}\label{e2.2.6}
r=r(s):=\max\{n\in\Z_+\, ;\, n<s\}
\end{equation}
we have the following result, see, e.g., \cite[Thm.\,2.32]{BBI-11}.
\begin{Prop}\label{prop2.2.3}
{\rm (a)} If $s\in (0,k)$ and noninteger or $s=k$, then
\begin{equation}\label{e2.2.7}
\Lambda^{k,s}=C^r Lip^{s-r}.
\end{equation}
Here $Lip^{\,\sigma}:=\Lambda^{1,\sigma}$, $0<\sigma\le 1$.

In particular, $\Lambda^{k,k}=C^{k-1,1}$.

\noindent {\rm (b)} If $s\in (0,k)$ and integer, then
\begin{equation}\label{e2.2.8}
\Lambda^{k,s}=C^{s-1}Z=C^r\Lambda^{2,1},
\end{equation}
where $Z:=\Lambda^{2,1}$ denotes Zygmund space.

In both cases,
\begin{equation}\label{e2.2.9}
|f|_{C^k\Lambda^{k,s}}:=\max_{|\alpha|=r}|D^\alpha f|_{\Lambda^{k,s}}.
\end{equation}
\end{Prop}
The next result goes back to Luzin's classical theorem on approximation of measurable functions on $[0,1]$ by continuous functions in a metric given for $f,g\in [0,1]\rightarrow\RR$ by
\[
\rho(f,g):=|\{x\in [0,1]\, ;\, f(x)\ne g(x)\}|.
\]
\begin{Th}\label{teo2.2.5}
Let $f$ belong to the space $\dot V_p^k$ of smoothness $s\in (0,k]$ and $\varepsilon\in (0,1)$.
The following is true.

\noindent {\rm (a)} If $s<k$, then there is a function $f_\varepsilon\in\Lambda^{k,s}$ such that
\begin{equation}\label{eq2.2.10}
\rho(f,f_\varepsilon)<\varepsilon .
\end{equation}

\noindent {\rm (b)} If $s=k$, then there is a function $f_\varepsilon\in C^k$ such that \eqref{eq2.2.10} is fulfilled.
\end{Th}
Using some facts established in the proof of this theorem we will then study pointwise behaviour of $\dot V_p^k$ functions. Model cases for this are, respectively, Lebesgue's differentiation theorem for $BV$ functions and the Marcinkiewicz theorem \cite{Ma-64} asserting that $f\in BV_p$ satisfies the pointwise Lipschitz condition
\[
\varlimsup_{h\rightarrow 0} h^{-\frac 1 p}|f(x+h)-f(x)|<\infty
\]
for almost each $x\in [0,1]$.

To formulate the corresponding multivariate results we use the following notions.
\begin{D}\label{def2.2.6}
A function $f\in\ell_\infty^{\rm loc}(\RR^d)$ belongs to a one-pointed Lipschitz class $\Lambda^{k,s}(x_0)$ if
\begin{equation}\label{e2.2.11}
\varlimsup_{r\rightarrow 0} r^{-s}{\rm osc}_k(f; Q_r(x_0))<\infty;
\end{equation}
here $k\in\N$, $0<s\le k$ and $x_0\in\mathring Q^d$.
\end{D}
Similarly to the definition of functions from $\Lambda^{k,s}$ we use the notations
\[
Lip^s (x_0):=\Lambda^{1,s}(x_0),\quad Z(x_0):=\Lambda^{2,1}(x_0).
\]

Further, we define so-called {\em Taylor spaces}  introduced by Peano \cite{Pe-1891} for $d=1$ and studied (and named) by Calder\'{o}n and Zygmund \cite{CZ-61} for functions
from $L_p^{\rm loc}(\RR^d)$, $1\le p<\infty$. 
\begin{D}\label{def2.2.7}
{\rm (a)} A function $f\in\ell_\infty^{\rm loc}(\RR^d)$ belongs to the space $T^s(x_0)$, $s>0$, if there is a (Taylor) polynomial $T_{x_0}(f)$ 
\begin{equation}\label{e2.2.12}
T_{x_0}(f;x):=\sum_{|\alpha|<s}\frac{f_\alpha (x_0)}{\alpha !} (x-x_0)^\alpha
\end{equation}
 such that
\begin{equation}\label{e2.2.13}
|f-T_{x_0}(f;x)|=O(\|x-x_0\|^s)\quad {\rm as}\quad x\rightarrow x_0.
\end{equation}

\noindent {\rm (b)} If, moreover, $s$ is integer, then $f\in\ell_\infty^{\rm loc}(\RR^d)$ belongs to the space $t^s(x_0)$ whenever there exists a polynomial $T_{x_0}(f)$ of degree $s$ such that
\begin{equation}\label{e2.2.14}
|f-T_{x_0}(f;x)|=o(\|x-x_0\|^s)\quad {\rm as}\quad x\rightarrow x_0.
\end{equation}
\end{D}
\begin{R}\label{rem2.2.8}
{\rm (a) It is readily seen that the polynomials $T_{x_0}(f)$ are unique.

\noindent (b) If $f\in t^k(x_0)$, then the $k$-th homogeneous part of $T_{x_0}(f)$
\[
D_{x_0}^k(f;x):=\sum_{|\alpha|=k}\frac{f_\alpha(x_0)}{\alpha !}(x-x_0)^\alpha
\]
is called the {\em Peano $k$-th differential at $x_0$}. It is the matter of definition to check that the $k$-th differential in the classical sense
\[
\partial_{x_0}^k f(x):=\sum_{|\alpha|=k}\frac{D^\alpha f(x_0)}{\alpha !} (x-x_0)^\alpha
\]
coincides with that of Peano for $k=1$. However, this is, in general, incorrect for $k\ge 2$. In particular, Denjoy \cite{De-35} proved that for every nowhere dense closed set $S\subset [0,1]$ and $k\ge 2$ there is a function $f$ whose $k$-th derivative in the classical sense exists only at points of the completion $S^c$ of $S$ while the Peano $k$-th derivative exists at all points of $[0,1]$. The converse to this assertion that also proved in \cite{De-35} leads in the multivariate case to the following.
\begin{Con}
If the Peano $k$-th differential $D_x^k f$, $k\ge 2$, exists at every point of $x\in Q^d$, then $k$-th differential $\partial_x^k f$ in the classical sense exists a.e.
\end{Con}
\noindent (c) Unlike the classical case, the Peano definition can be extended to a much more complicated setting, see, in particular, \cite{Br2-15} where this is done for multivariate functions on weak Markov sets including Ahlfors $a$-regular sets in $\RR^d$ with $d-1<a\le d$ and fractals of such complicated structure as the von Koch curves, the Cantor sets etc. Such kind of generalization has useful applications in harmonic analysis and singular integral operators theory, see \cite{MZ-36}, \cite{CZ-61} and the following up papers.
}
\end{R}
\begin{Th}\label{te2.2.8}
Let  a function $f$ belong to the space $\dot V_p^k$ of smoothness $s:=\frac d p \in (0,k]$. Then the following is true for almost each $x_0\in Q^d$.\smallskip

\noindent {\rm (a)} 
\[
f\in\Lambda^{k,s}(x_0);\]

\noindent {\rm (b)} If
$s=\ell+\lambda<k$,
where $\ell\in\Z_+$ and $0<\lambda\le 1$, then
\[
f\in T^s(x_0);
\] 

\noindent {\rm (c)}
If  $s=k$, then 
\[
f\in t^k(x_0).
\]
\end{Th}

Finally, we present results connecting $\dot V_p^k$ spaces with Lipschitz-Besov spaces over $L_p$, $1\le p<\infty$. The latter family of 
spaces is denoted by $\{\Lambda_{pq}^{ks}\}$, where $k\in\N$, $0<s\le k$ and $1\le p,q<\infty$. For its definition we need the following:
\begin{D}\label{def2.2.10}
$(k,p)$-modulus of continuity is a function $L_p\times\RR_+\rightarrow\RR_+$ given for $f\in L_p$ and $0<t\le 1$ by
\begin{equation}\label{eq2.2.15}
\omega_{kp}(f;t):=\sup_{\|h\|\le t}\|\Delta_h^k f\|_{L_p(Q^d_{kh})},
\end{equation}
where $\|h\|:=\max_{1\le i\le d}|h_i|$ and
\begin{equation}\label{e2.2.16}
Q^d_{kh}:=\{x\in Q^d\, :\, x+kh\in Q^d\}.
\end{equation}
is the domain of the function $x\mapsto (\Delta_h^k f)(x)$.
\end{D}
Now the space $\Lambda_{pq}^{ks}$ is defined for $q<\infty$ by the seminorm
\begin{equation}\label{e2.2.17}
|f|_{\Lambda_{pq}^{ks}}:=\left(\int_0^1\left(\frac{\omega_{kp}(f;t)}{t^s}\right)^q\frac{dt}{t}\right)^{\frac 1 q}
\end{equation}
and for $q=\infty$ by
\begin{equation}\label{eq2.2.18}
|f|_{\Lambda_{p\infty}^{ks}}:=\sup_{t>0}\frac{\omega_{kp}(f;t)}{t^s}.
\end{equation}
For the special case $q=\infty$, $s=k$ we also use the notation
\begin{equation}\label{eq2.2.19}
Lip_p^k:=\Lambda_{p\infty}^{kk}.
\end{equation}
\begin{R}\label{rem2.2.11}
{\rm (1) We also use the introduced notions for functions on a subcube $Q\subset Q^d$. In this case, $f\in L_p(Q)$ and the corresponding notions are denoted by
$\omega_{kp}(f;t;Q)$, $\Lambda_{pq}^{ks}(Q)$ and $Lip_p^k(Q)$.\smallskip

\noindent (2) The space $\Lambda_{pq}^{ks}$ with $k-1\le s<k$ coincides with the homogeneous Besov space $\dot B_{pq}^s$ while the space $\Lambda_{pq}^{ks}/\mathcal P_{k-1}^d$ up to equivalence of norms coincides with Besov space $B_{pq}^s$, see, e.g., \cite[Sec.II.2.3]{Tr-92}.\smallskip

\noindent (3) In the case of one variable, the Hardy-Littlewood theorem \cite{HL-28} asserts (in our notations) that
\[
Lip^{\frac 1 p}\subset N\dot V_p\hookrightarrow Lip_p^{\frac1 p}
\]
for $1<p<\infty$ and
\[
Lip^1\subset N\dot V\cong Lip_1^1.
\]

In more details, let $L$ be a linear operator sending a Lebesgue integrable function to its class of equivalence in $L_1$. Then the cited result asserts that the trace $L|_{N\dot V_p}$ is a continuous injection into $Lip_{p}^{1/p}$ if $1<p<\infty$ and a continuous bijection {\em onto} $Lip_1^1$ if $p=1$.
}
\end{R}
\begin{Th}\label{teo2.2.11}
{\rm (a)} If $s:=\frac d p\in (0,k]$, then the operator $L|_{N\dot V_p^k}$ gives the map
\begin{equation}\label{e2.2.19}
N\dot V_p^k\hookrightarrow\Lambda_{p\infty}^{ks}.
\end{equation}
(Let us recall that if $s=k$ the right-hand side is denoted by $Lip_p^k$.)

Moreover, if $0<s<k$, then there is a linear continuous operator right inverse to
$L|_{N\dot V_p^k}$ that gives the map
 \begin{equation}\label{e2.2.19a}
\Lambda_{p1}^{ks}\hookrightarrow N\dot V_p^k.
\end{equation}

\noindent {\rm (b)} If $s=k=d$, hence, $p=1$, then $L|_{N\dot V_1^d}$ gives the linear isomorphism
\begin{equation}\label{e2.2.21}
N\dot V_1^d\cong Lip_1^d.
\end{equation}
\end{Th}
Let us note that embedding \eqref{e2.2.19} and isomorphism \eqref{e2.2.21} give the above mentioned Hardy-Littlewood results
while embedding \eqref{e2.2.19a} is apparently new even for this case.
\begin{R}\label{rem2.18}
{\rm The left-side Hardy-Littlewood embeddings also can be extended to the multivariate case, we have
\begin{equation}\label{eq-n2.27}
\Lambda^{k,s}\subset\dot V_p^k,\quad s:=\frac d p \in (0,k].
\end{equation}

In fact, by definition, see \eqref{e2.2.4}, \eqref{e2.2.5},
\[
{\rm osc}_k(f;Q)\le |Q|^{\frac s d}|f|_{\Lambda^{k,s}},\quad Q\subset Q^d,
\]
while by Theorem \ref{cor2.1.2} $|f|_{V_p^k}$ is majorated by
\[
\sup_{\pi}\left\{\sum_{Q\in\pi}{\rm osc}_k(f;Q)^p\right\}^{\frac 1 p}\le\sup_{\pi}\left\{\sum_{Q\in\pi}|Q|^{\frac{sp}{d}}\right\}^{\frac 1 p } |f|_{\Lambda^{k,s}}=|f|_{\Lambda^{k,s}}.
\]
}
\end{R}
As a consequence of Theorem \ref{teo2.2.11}\,(c) we present the multivariate generalization of the classical Lebesgue-Vitali theorem on absolutely continuous functions  formulated as property (d) of the space $\dot V:=\dot V_1^d$ in Section~1.1. Let us recall that 
$AC$ denotes the space of functions $f\in \dot V$ satisfying the condition
\begin{equation}\label{eq-n9.30}
\lim_{\varepsilon\rightarrow 0}\sup_{|\pi|\le\varepsilon}\sum_{Q\in\pi}{\rm osc}_d(f;Q)=0;
\end{equation}
here $|\pi|:=\sum_{Q\in\pi}|Q|$ and ${\rm osc}_d$ can be replaced by $E_d$, see Theorems \ref{cor2.1.2},   \ref{teor2.1.1}.
\begin{C}\label{cor9.7}
The map sending a function from $\dot V$ to its equivalence class in $L_1$ gives rise to the isomorphism
\begin{equation}\label{eq-n9.31}
AC\cong\dot W_1^d.
\end{equation}

Moreover, there are a continuous linear integral operator $J: L_1(Q^d;\RR^N)\rightarrow C$ and a linear projection $P: L_1\rightarrow\mathcal P_{d-1}^d$ such that for every $f\in AC$
\begin{equation}\label{eq-n9.32}
f=Pf+J(\nabla^d f);
\end{equation}
here $\nabla^d f:=(D^\alpha f)_{|\alpha|=d}$ and $N:={\rm card}\{\alpha\in\Z^d_+\, :\, |\alpha|=d\}$.
\end{C}

\subsection{Duality Theorems} Now we introduce a predual to $V_p^k$ space denoted by $U_p^k$ and describe its properties.

The space under consideration is constructed by the following building blocks called {\em $(k,p)$-atoms}.
\begin{D}\label{def2.2.1}
A function $a\in\ell_1\, (=\ell_1(Q^d))$ is said to be a $(k,p)$-atom if it is such that
\begin{itemize}
\item[(a)]
$a$ is supported by a subcube $Q\subset Q^d$;\smallskip
\item[(b)]
$\|a\|_1\le 1$;\smallskip
\item[(c)]
$\displaystyle \sum_{x\in Q^d} x^\alpha a(x)=0$ for all $|\alpha|\le k-1$.
\end{itemize}
\end{D}
The subject of the definition is denoted by $a_Q$.
\begin{D}[$(k,p)$-chain]\label{def2.2.2}
A function $b\in\ell_1$ is said to be a $(k,p)$-chain subordinate to a packing $\pi\in\Pi$ if $b$ belongs to the linear hull of the family $\{a_Q\}_{Q\in\pi}$.
\end{D}
The subject of the definition is denoted by $b_\pi$.

Moreover, we set
\begin{equation}\label{2.2.1}
[b_\pi]_{p'}:=\|\{c_Q\}_{Q\in\pi}\|_{p'}:=\left(\sum_{Q\in\pi}|c_Q|^{p'}\right)^{\frac{1}{p'}}
\end{equation}
whenever
\begin{equation}\label{2.2.2}
b_\pi=\sum_{Q\in\pi}c_Q a_Q,\qquad \{c_Q\}_{Q\in\pi}\subset\RR.
\end{equation}
(As usual, $\frac{1}{p}+\frac{1}{p'}=1$.)

Further, we define a dense subspace of the required predual space denoted by $(U_p^k)^0$ by setting
\begin{equation}\label{2.2.3}
(U_p^k)^0:={\rm linhull}(\mathcal A_p^k),
\end{equation}
where $\mathcal A_p^k$ is the family of all $(k,p)$-atoms.

Equivalently, 
\begin{equation}\label{2.2.4}
(U_p^k)^0:={\rm linhull}(\mathcal B_p^k),
\end{equation}
where $\mathcal B_p^k$ is the family of all $(k,p)$-chains.

Moreover, we equip $(U_p^k)^0$ with a seminorm\footnote{In fact, norm, see Theorem \ref{te2.2.5}\,(a) below.} $f\mapsto \|f\|_{U_p^k}$ given for $f\in (U_p^k)^0$ by
\begin{equation}\label{2.2.5}
\|f\|_{U_p^k}:=\inf\sum_\pi\, [b_\pi]_{p'},
\end{equation}
where infimum is taken over all presentations
\begin{equation}\label{2.2.6}
f=\sum_\pi b_\pi .
\end{equation}
\begin{D}\label{def2.2.3}
The space $U_p^k$ is completion of the space $((U_p^k)^0,\|\cdot\|_{U_p^k})$.
\end{D}
Now we describe the properties of the space $U_p^k$. 

In the formulation of the result, we use a space $U_{p\infty}^k$ defined by the very same construction as $U_p^k$ but with $\ell_1$ replaced by $L_1$. Hence, in this case $a_Q$ is an $L_1$ function supported by a subcube $Q$ orthogonal to $\mathcal P_{k-1}^d$ and such that $\int_{Q^d}|a_Q|\,dx\le 1$. 

It follows from \cite[Th.\,2.6]{BB-18} that 
\[
(U_{p\infty}^k)^*\equiv V_{p\infty}^k.
\] 
\begin{Th}\label{te2.2.5}
{\rm (a)} $U_p^k$ is a Banach space.\smallskip

{\rm (b)} $B(U_p^k)$ is the closure of the symmetric convex hull of the set $\{b_\pi\in (U_p^k)^0\, :\, [b_\pi]_{p'}\le 1\}$.\smallskip

{\rm (c)} The space $U_p^k$ is nonseparable and contains a separable subspace, say, $\hat U_p^k$, such that
\[
\hat U_p^k\equiv U_{p\infty}^k\quad {\rm and}\quad (U_p^k/\hat U_p^k)^*\cong\ell_p.
\]
\end{Th}
Finally, we present the duality theorem for the Banach space $V_p^k\, (:=\dot V_p^k/\mathcal P_{k-1}^d$, see \eqref{1.2.17}).
\begin{Th}\label{te2.2.6}
It is true that
\begin{equation}\label{2.2.11}
(U_p^k)^*\equiv V_p^k.
\end{equation} 
In more details, every linear functional $f\in (U_p^k)^*$ is assigned at $u\in (U_p^k)^0$ by
\begin{equation}\label{2.2.12}
\langle f,u\rangle:=\sum_{x\in Q^d}v_f(x)u(x),
\end{equation}
where $v_f$ is a uniquely defined function of $\dot V_p^k$ linearly depending on $f$. 

Moreover, it is true that
\begin{equation}\label{2.2.13}
\|f\|_{(U_p^k)^*}=|v_f|_{V_p^k}.
\end{equation}
\end{Th}
\subsection{Two Stars Theorems}  Now we present duality results  connecting the space $\textsc{v}_p^k$ with the spaces $V_p^k$ and $U_p^k$. In their proofs an essential role plays an extension of the space $U_p^k$ denoted by $\widetilde U_p^k$. In its definition, we use
the Banach space $M$ of finite signed Borel measures on $Q^d$ equipped  with a norm given for $\mu\in M$ by
\begin{equation}\label{2.3.1}
\|\mu\|_M:=\underset{Q^d}{\rm var}\,\mu.
\end{equation}
We define $\widetilde U_p^k$ as a direct analog of the space $U_p^k$ using as building blocks atoms $\{\mu_Q\}_{Q\subset Q^d}\subset M$ satisfying the conditions
\begin{itemize}
\item[(i)] ${\rm supp}\,\mu\subset Q$;\smallskip
\item[(ii)] $\|\mu\|_M\le 1$;\smallskip
\item[(iii)] $\displaystyle \int_{Q^d}x^\alpha\,d\mu=0$ for all $|\alpha|\le k-1$.
\end{itemize}

The corresponding linear hull of the set $\{\mu_Q\}_{Q\subset Q^d}$ is denoted by $(\widetilde U_p^k)^0$ and by $\widetilde U_p^k$ is denoted its completion with respect to a seminorm defined similarly to that for $(U_p^k)^0$.

Repeating line-by-line the proof of Theorem \ref{te2.2.5}\,(a) we will show that the space $\widetilde U_p^k$ is Banach.

Since every chain $b\in (U_p^k)^0$ can be identified with a discrete measure $\mu_b\in M$ with $\mu_b(\{x\}):=b(x)$, $x\in Q^d$, the space $(U_p^k)^0$ is embedded (as a vector space) in $(\widetilde U_p^k)^0$.
\begin{Th}\label{te2.3.2}
{\rm (a)} The map $(U_p^k)^0\rightarrow (\widetilde U_p^k)^0$, $b\mapsto\mu_b$, extends to an isometric embedding of $U_p^k$ into $\widetilde U_p^k$ denoted by $\mathcal I$.\smallskip

\noindent {\rm (b)} There exists a linear continuous surjection of $\widetilde U_p^k$ onto $U_p^k$ denoted by $\mathcal E$ such that
\[
{\rm ker}(\mathcal E)\cap (\widetilde U_p^k)^0=(0)\quad {\rm and}\quad \mathcal E\circ\mathcal I={\rm id}.
\]
\end{Th}
The result allows to identify $(\widetilde U_p^k)^0\subset M$ with a dense subspace of $U_p^k$ containing $(U_p^k)^0$ and to apply the Riesz representation theorem in the forthcoming proofs.

Now we formulate a result giving another presentation of the predual space $U_p^k$ and derive from here the mentioned two stars theorem. 
\begin{Th}\label{te2.3.3}
If the space $V_p^k$ has smoothness $s<k$ and $p>1$, then
\begin{equation}\label{2.3.5}
(\textsc{v}_p^k)^*\cong U_p^k.
\end{equation}

In more details, every linear continuous functional on $\textsc{v}_p^k$ is of the form 
\[
T_u(f)=f(u),\qquad f\in \textsc{v}_p^k,
\]
where $u\in U_p^k$  is a uniquely defined element such that
\[
\|T_u\|_{(\textsc{v}_p^k)^*}\approx \|u\|_{U_p^k}
\] 
with the constants of equivalence independent of $u$.
\end{Th}

As a direct consequence we have the following 
\begin{C}[Two stars theorem]\label{cor2.3.4}
If $V_p^k$ satisfies the conditions $s<k$ and $p>1$, then
\begin{equation}\label{2.3.6}
(\textsc{v}_p^k)^{**}\cong V_p^k.
\end{equation}
\end{C}
%
\sect{Proofs of Theorem \ref{te2.1.3} and Corollary \ref{cor2.1.3}}
\subsection{Auxiliary Results}
Here we prove some auxiliary results that are interesting in their own right. 
\begin{Lm}\label{lemm4.1}
Let $\{Q_i\}_{i\in\N}$ be the family of nonoverlapping cubes in $Q^d$. Then for each $f\in\dot V_p^k$
the limit $\displaystyle \lim_{i\rightarrow\infty}E_k(f;Q_i)$ exists and equals $0$.
\end{Lm}
\begin{proof}
By the definition of the seminorm on $\dot V_p^k$, see \eqref{1.2.14}, we have
\[
\sum_{i=1}^\infty E_k(f;Q_i)^p\le |f|_{V_p^k}^p<\infty.
\]
Since $1\le p<\infty$, this implies the required statement.
\end{proof}

Further, let $\mathcal Q_x$ denote the family of subcubes of $Q^d$ containing a point $x\in Q^d$. We set for a function $f\in\dot V_p^k$, 
\begin{equation}\label{fk}
f_k(x):=\varlimsup_{\mathcal Q_x\ni Q\rightarrow \{x\}}E_k(f;Q);
\end{equation} 
here the convergence is in the Hausdorff metric on subsets of $\mathbb R^d$. 

Clearly, $0\le f_k(x)\le |f|_{V_p^k}$, $x\in Q^d$. In fact, the following result holds.
\begin{Lm}\label{lem1}
The function $f_k\in\ell_p$ and satisfies the inequality $\|f_k\|_{p}\le |f|_{V_p^k}$.
\end{Lm}
\begin{proof}
Assume that $c$ is a non-isolated point of ${\rm range}(f_k):=f_k(Q^d)$. Then there exists a sequence of distinct points $\{x_i\}_{i\in\N}\subset Q^d$ such that $\lim_{i\rightarrow\infty}f_k(x_i)=c$. In turn, by the definition of $f_k$, there is a family of nonoverlapping cubes $Q_i\subset\mathcal Q_{x_i}$ such that $E_k(f;Q_i)\ge \frac{i-1}{i} f_k(x_i)$, $i\in\N$. Hence, due to Lemma \ref{lemm4.1},
\[
0\le c=\lim_{i\rightarrow\infty} f_k(x_i)\le\lim_{i\rightarrow\infty}E_k(f;Q_i)=0.
\]
Hence,  the set ${\rm range}(f_k)\setminus\{0\}$ consists of isolated points only and is at most countable.

Next, if $c\in {\rm range}(f_k)\setminus\{0\}$, then the set $f_k^{-1}(c)\subset Q^d$ is finite. (In fact, otherwise there is a family of nonoverlapping cubes $\{Q_i\}_{i\in\N}$ such that $\lim_{i\rightarrow\infty}E_k(f;Q_i)=c\ne 0$ which contradicts Lemma \ref{lemm4.1}).

These, in turn, imply that  ${\rm supp}\, f_k:=f_k^{-1}\bigl({\rm range}(f_k)\setminus\{0\}\bigr)$ is at most countable.
Moreover, $\|f_k;{\rm supp}\, f_k\|_p\le |f|_{V_p^k}$, by the definitions of $f_k$ and the seminorm on $\dot V_p^k$.
\end{proof}
\begin{Prop}\label{lem2}
Let $x\in Q^d$. Then $f_k(x)=0$ iff $f$ is continuous at $x$.
\end{Prop}
\begin{proof}
Suppose that $x\in Q^d$ is such that $f_k(x)=0$. We prove that $f$ is continuous at $x$.

Since $f\in\ell_\infty$, each infinite subset of ${\rm range}(f)$ has a limit point. Hence, we should prove that if $\{x_i\}_{i\in\N}\subset Q^d$ is a sequence converging to $x$ and the sequence $\{f(x_i)\}_{i\in\N}$ converges to some $c\in\RR$, then $c=f(x)$.
 
To this end, passing to a subsequence of $\{x_i\}_{i\in\N}$, if necessary, without loss of generality we can assume that 
\begin{equation}\label{eq1}
\lim_{i\rightarrow\infty}\frac{\|x_{i+1}-x\|_{\ell_\infty^d}}{\|x_{i+1}-x_i\|_{\ell_\infty^d}}=0.
\end{equation}

Since the projections of the interval $[x,x_i]\subset Q^d$ onto the coordinate axes have lengths bounded from above by $\|x-x_i\|_{\ell_\infty^d}\, (\le 1)$, there exists a cube
$Q_i\in \mathcal Q_x$ of sidelength $\|x-x_i\|_{\ell_\infty^d}$ containing $[x,x_i]$. 

In particular, the sequence $\{Q_i\}_{i\in\N}$ converges to $\{x\}$ in the Hausdorff metric.  

Moreover, sets $Q_i\cup Q_{i+1}$ are connected so that their projections on the coordinate axes are compact subintervals of $[0,1]$ of lengths $\le \|x-x_i\|_{\ell_\infty^d}+\|x-x_{i+1}\|_{\ell_\infty^d}$. 

Hence, there exists a cube $Q_{i,i+1}\in\mathcal Q_x$ of sidelength $\le \|x-x_i\|_{\ell_\infty^d}+\|x-x_{i+1}\|_{\ell_\infty^d}$ containing $Q_i\cup Q_{i+1}$. In particular, $Q_{i,i+1}\rightarrow\{x\}$ in the Hausdorff metric as $x\rightarrow\infty$. 

Now let $m_{i,i+1}\in\mathcal P_{k-1}^d$, $i\in\N$,  be such that
\[
E_k(f;Q_{i,i+1})=\|f-m_{i,i+1};Q_{i,i+1}\|_{\infty}.
\]
Since $\|m_{i,i+1};Q_{i,i+1}\|_\infty\le \|f\|_{\infty}$ and $Q_{i,i+1}$ contains the interval $[x_i,x_{i+1}]$, the mean-value theorem and the A. Markov inequality for derivatives of polynomials give for a  positive constant $c:=c(k,d)$\smallskip
\begin{equation}\label{eq2}
\begin{array}{l}
\displaystyle
|f(x)-f(x_{i+1})|\le |f(x)-m_{i,i+1}(x)|+|f(x_{i+1})-m_{i,i+1}(x_{i+1})|\\
\\
\displaystyle \quad\qquad\qquad\qquad+\,|m_{i,i+1}(x)-m_{i,i+1}(x_{i+1})| \le 2E_k(f;Q_{i,i+1})+
\max_{Q_{i,i+1}}\|\nabla m_{i,i+1}\|_{\ell_\infty^d}
\\
\\
\displaystyle \quad\qquad\qquad\qquad
\le 2E_k(f;Q_{i,i+1})+\frac{c}{{\rm diam}(Q_{i,i+1})} \|f\|_{\infty}\|x-x_{i+1}\|_{\ell_\infty^d}\\
\\
\displaystyle \quad\qquad\qquad\qquad\le 2E_k(f;Q_{i,i+1}) + \frac{c\|x-x_{i+1}\|_{\ell_\infty^d}}{\|x_i-x_{i+1}\|_{\ell_\infty^d}}\|f\|_{\infty}.
\end{array}
\end{equation}

Further, since $f_k(x)=0$, condition \eqref{fk} implies that
\[
\lim_{\mathcal Q_x\ni Q\rightarrow \{x\}}E_k(f;Q)=0.
\]
From here and \eqref{eq1} we obtain that the right-hand side of \eqref{eq2} tends to $0$ as $i$ tends to $\infty$. Thus, $c=\lim_{i\rightarrow\infty}f(x_{i+1})=f(x)$, as required, i.e., $f$ is continuous at $x$. \smallskip

Conversely, suppose that $f\in\dot V_p^k$ is continuous at $x\in Q^d$.
Then
\[
0\le f_k(x):=\varlimsup_{\mathcal Q_x\ni Q\rightarrow \{x\}}E_k(f;Q)\le \varlimsup_{\mathcal Q_x\ni Q\rightarrow \{x\}}E_1(f;Q)\le \varlimsup_{\mathcal Q_x\ni Q\rightarrow \{x\}}\|f-f(x);Q\|_\infty=0.
\]
Hence, $f_k(x)=0$. This completes  the proof  of the proposition.\end{proof}

\begin{Prop}\label{lem3}
Let $\ell$ be a line intersecting $Q^d$ by a nontrivial interval $I$. Then for each $f\in\dot V_p^k$ its restriction to $I$ belongs to the space $\dot V_p^k(I)$.
\end{Prop}
\begin{proof}
We need the following geometric result.
\begin{Lm}\label{lemma3.5}
Let $\{I_j\}_{1\le j\le s}$ be a family of nonoverlapping nontrivial closed intervals in $I$. There exists a family of nonoverlapping cubes $\{Q_j\}_{1\le j\le s}$ in $Q^d$ such that 
\begin{equation}\label{eq-n3.4}
|Q_j|^{\frac 1 d}=\ell_\infty(I_j)\quad
{\rm and}\quad Q_j\cap I=I_j,\quad 1\le j\le s;
\end{equation}
here $\ell_\infty([x,y]):=\displaystyle\max_{1\le i\le d}\{|x_i-y_i|\}$ is the $\ell_\infty^d$-diameter of $[x,y]\subset\RR^d$.
\end{Lm}
\begin{proof}
Let $\alpha_i$ be the angle between a vector parallel to $I$ and the coordinate vector $e_i$, $1\le i\le d$. Renumerating the coordinates axes, if necessary, without loss of generality we assume that the sequence 
$\{|\cos\alpha_i |\}_{1\le i\le d}$ is nonincreasing. Then the lengths of the orthogonal projections of every $I_j$ on the coordinate axes $x_i$, $1\le i\le d$, also form a nonincreasing sequence.

Using this we proceed  by induction on $d$ starting from the trivial case $d=1$ and assuming that the result holds for all dimensions $\le d-1$ with $d\ge 2$. 

To prove the lemma for $d\ge 2$ we use the map $\pi_d: x\mapsto (x_1,\dots, x_{d-1},0)$ sending $Q^d$ onto $Q^{d-1}:=\{x\in Q^d\, :\, x_d=0\}$. Then $I':=\pi_d(I)$ is a closed interval in the cube $Q^{d-1}$
and $I_j':=\pi_d(I_j)$, $1\le j\le  s$, is a family of nonoverlapping closed intervals in $I'$. Moreover, $I'$ is nontrivial (hence, every $I_j'$ is) since otherwise $|\cos\alpha_d |=0$ and, hence, by our assumption all  $|\cos\alpha_j|=0$ contradicting nontriviality of $I$.

Now by the induction hypothesis there is a family of nonoverlapping cubes $Q_j'\subset Q^{d-1}$,  $1\le j\le s$, satisfying condition \eqref{eq-n3.4} for the family $\{I_j'\}_{1\le j\le s}$.

Next, preimages of $Q_j'$ with respect to the map $\pi_d|_{Q^d}$ form the family of nonoverlapping rectangular parallelotopes, say $\tilde Q_j$, with bases $Q_j'$ and heights being the unit intervals parallel to the $x_d$-axis; moreover, $I_j=\ell\cap\tilde Q_j$, $1\le j\le s$.  

Let $I_{j,d}$ be the image of $I_j$ under the orthogonal projection onto  the $x_d$ axis. Then  $I_{j,d}$ is a (possibly trivial) interval
of length  $\le \ell_{\infty}(I_j)\le 1$ in the unit interval in this axes. In particular, there is a subinterval $I_{j,d}'$ in the unit interval  of length $\ell_{\infty}(I_j)$ containing $I_{j,d}$. 

The intersection of the preimage of this interval under the projection onto $x_d$ with $\tilde Q_j$ is clearly a cube $Q_j$ of sidelength $\ell_{\infty}(I_j)$ satisfying the required statement.
\end{proof}

Finally, using the result and notations of Lemma \ref{lemma3.5}, for each $f\in \dot V_p^k$ we have
\[
E_k(f;I_j)\leq E_k(f;Q_j),\quad 1\le j\le s. 
\] 
In particular,
\[
\left(\sum_{j=1}^s E_k(f;I_j)^p\right)^{\frac 1 p}\le
\left(\sum_{j=1}^s E_k(f;Q_j)^p\right)^{\frac 1 p}\le |f|_{V_p^k}.
\]
Taking in the left-hand side supremum over all packings $\{I_j\}\subset I$ we obtain that $f|_I\in \dot V_p^k(I)$ and $|f|_I|_{V_p^k(I)}\le |f|_{V_p^k}$.

The proof of the proposition is complete.
\end{proof}
\subsection{Proof of Theorem \ref{te2.1.3}}
\begin{proof}
(a) Suppose that $d=1$ and $f\in\dot V_p^k([0,1])$. We should prove that $f$ has one-sided limits at each point $x\in [0,1]$. 

Since $f\in\ell_\infty$, each infinite subset of ${\rm range}(f)$ has a limit point. Hence,  we should prove that if $\{x_{j,i}\}_{i\in\N}\subset [0,1]$, $j=1,2$, are sequences converging to $x$ from the same side such that sequences $\{f(x_{j,i})\}_{i\in\N}$ converge to some $c_j\in\RR$, $j=1,2$, then $c_1=c_2$.

To this end, passing to subsequences of $\{x_{j,i}\}_{i\in\N}$, $j=1,2$, if necessary, without loss of generality we may assume that 
\begin{equation}\label{equ4.4}
x_{1,i+1}<x_{2,i+1}<x_{1,i}<x_{2,i}\quad {\rm for\ all}\quad i\in\N\quad {\rm and}\quad
\lim_{i\rightarrow\infty}\frac{|x_{2, i+1}-x_{1,i+1}|}{|x_{2,i+1}-x_{1,i}|}=0;
\end{equation}
here all intervals $I_i:=[x_{1,i},x_{2,i}]\subset [0,1]$, $i\in\N$, belong to the same connected component of the set $\RR\setminus\{x\}$.

By definition, $\{I_i\}_{i\in\N}\subset [0,1]$ is a family of  nonoverlapping intervals converging to $\{x\}$ in the Hausdorff metric. The second condition \eqref{equ4.4} implies that there exists a family of nonoverlapping intervals $\{I_i'\}_{i\in\N}\subset [0,1]$ such that 
\begin{equation}\label{equ4.5}
I_i\subset I_i'\quad {\rm for\ all}\quad i\in\N,\quad \lim_{i\rightarrow\infty}\frac{|I_i|}{|I_i'|}=0,\quad {\rm and}\quad  \lim_{i\rightarrow\infty}I_i'=\{x\}
\end{equation}
in the Hausdorff metric.

In particular, due to Lemma \ref{lemm4.1}, 
\begin{equation}\label{equ4.6}
\lim_{i\rightarrow\infty}E_k(f;I_i')=0.
\end{equation}

Let $m_{i}\in\mathcal P_{k-1}^1$, $i\in\N$,  be such that
\[
E_k(f;I_i')=\|f-m_{i};I_i'\|_{\infty}.
\]
Since $\|m_{i};I_i'\|_{\infty}\le \|f\|_{\infty}$ and $I_{i}'\supset I_i$,  the mean-value theorem and the A. Markov inequality for derivatives of polynomials give \smallskip
\[
\begin{array}{l}
\displaystyle
|f(x_{2,i})-f(x_{1,i})|\le |f(x_{2,i})-m_{i}(x_{2,i})|+|f(x_{1,i})-m_{i}(x_{1,i})|+\,|m_{i}(x_{2,i})-m_{i}(x_{1,i})|\\
\\
\displaystyle  \qquad\qquad\qquad\qquad\!\le 2E_k(f;I_{i}')+\frac{2(k-1)^2}{|I_{i}'|} \|f\|_{\infty}|I_i|.
\end{array}
\]
Since the right-hand side of the inequality tends to $0$ as $i\rightarrow\infty$, see \eqref{equ4.5}, \eqref{equ4.6}, 
\[
c_1=\lim_{i\rightarrow\infty}f(x_{1,i})=\lim_{i\rightarrow\infty}f(x_{2,i})=c_2.
\]
This  completes the proof of   part (a) of the theorem.\smallskip

\noindent (b) Now suppose that $f\in \dot V_p^k(Q^d)$ with $d\ge 2$. We should prove that $f$ has limits at each point $x\in Q^d$.

First, due to Proposition \ref{lem3}, for each line $\ell$ through $x$ such that $I:=\ell\cap Q^d\ne\{x\}$ the restriction $f|_{I}$ belongs to $\dot V_p^k(I)$. In particular, according to (a)  $f|_{I}$ has one sided limits at $x$. This implies that $f$ has radial limits at $x$.

Next, we show that all these limits coincide.

Let $H_x$ be a hyperplane through $x$ parallel to a coordinate hyperplane and $x_j$ be the coordinate axis orthogonal to $H_x$. Then $\RR^d\setminus H_x$ is the union of two nonintersecting open half spaces $S_1,S_2$ such that at least one of them, say $S_1$, intersects $Q^d$. 
\begin{Lm}\label{lem4}
 Let $y_1,y_2\in S_1\cap Q^d$ be distinct. Then
\begin{equation}\label{eq-n3.8}
\lim_{(x,y_1]\ni z\rightarrow x}f(z)=\lim_{(x,y_2]\ni z\rightarrow x}f(z).
\end{equation}
\end{Lm}
\begin{proof}
By our assumption, the orthogonal projection onto the  $x_j$ axis $p_j$ maps lines $(x,y_i)$, $i=1,2$,  bijectively onto $x_j$. Let 
\[
I:=p_j\bigl((x,y_1]\bigr)\cap p _j\bigl((x,y_2]\bigr)
\]
be the half open interval with the endpoint $p_j(x)$. We choose a monotonic (with respect to the natural order on $x_j$) sequence $\{z_n\}_{n\in\N}\subset I$ converging to $p_j(x)$ and set
\[
y_{n,i}:=p_j^{-1}(z_n)\cap (x,y_i),\quad i=1,2.
\]
Passing to a subsequence of $\{z_n\}_{n\in\N}$, if necessary, without loss of generality we assume that
\begin{equation}\label{eq3}
\lim_{n\rightarrow\infty}\frac{\max_{i=1,2} \|y_{n+1,i}-x\|_{\ell_\infty^d}}{\|z_{n+1}-z_{n}\|_{\ell_\infty^d}}=0.
\end{equation} 

Next, let $P_n:=p_j^{-1}([z_{n},z_{n+1}])\cap Q^d$ be the rectangular parallelotope between hyperplanes $p_j^{-1}(z_{n+1})$ and 
$p_j^{-1}(z_{n})$ inside $Q^d$. Then $P_n$ contains points $y_{n+1,i}$, $i=1,2$, and the smallest length of projections of $P_n$ onto coordinate axes is $\|z_{n+1}-z_{n}\|_{\ell_\infty^d}$. Moreover, due to \eqref{eq3},
\[
\|y_{n+1,2}-y_{n+1,1}\|_{\ell_\infty^d}\le 2\max_i\|y_{n+1,i}-x\|_{\ell_\infty^d}=\lambda_n\cdot \|z_{n+1}-z_{n}\|_{\ell_\infty^d},
\]
where the sequence $\{\lambda_n\}_{n\in\N}$ converges to $0$.

These imply that for each $n$ such that $\lambda_n<1$ there exists a cube $Q_n\subset P_n$ of sidelength $\|z_{n+1}-z_{n}\|_{\ell_\infty^d}$ containing points $y_{n+1,i}$, $i=1,2$. 

By the definition, $\{Q_n\}_{n\in\N}$ is a sequence of nonoverlapping cubes converging to $\{x\}$ in the Hausdorff metric.  Hence, according to Lemma \ref{lemm4.1} 
\begin{equation}\label{equ4.8}
\lim_{n\rightarrow\infty}E_k(f;Q_n)=0.
\end{equation}
Then as in \eqref{eq2} using the mean-value theorem and the Markov inequality for derivatives of polynomials on $Q_n$, and after that using \eqref{eq3}, \eqref{equ4.8}, we obtain for some $c=c(k,d)$
\[
\varlimsup_{n\rightarrow\infty}|f(y_{n+1,1})-f(y_{n+1,2})|\le \varlimsup_{n\rightarrow\infty}\left(2E(f;Q_n)+\frac{c\|y_{n+1,2}-y_{n+1,1}\|_{\ell_\infty^d}}{\|z_{n+1}-z_{n}\|_{\ell_\infty^d}}\|f\|_{\infty}\right)=0.
\]

This implies \eqref{eq-n3.8}.
\end{proof}
Using Lemma \ref{lem4} we obtain the following:
\begin{Lm}\label{lem4.6}
If $d\ge 2$, then all radial limits of $f\in \dot V_p^k(Q^d)$ at a point $x\in Q^d$ coincide.
\end{Lm}
\begin{proof}
First, observe that $Q^d\setminus\{x\}$ is covered by the union of open  sets 
$\RR^d\setminus H_x^i$, $1\le i\le d$, where $H_x^i$ is a hyperplane through $x$ orthogonal to the $x_i$ axes; this is true because 
\[
\bigcap_{1\le i\le d} H_x^i=\{x\}.
\]
Let $\{S_j\}_{1\le j\le k}$, $d\le k\le 2d$, be the family of all connected components (open half spaces)
 of the sets $\RR^d\setminus H_x^i$, $1\le i\le d$, intersecting $Q^d$.
According to Lemma \ref{lem4}, $f$ has the same radial limit 
at $x$ along each ray in $S_j$ emanating from $x$ denoted by $L_{j}$. 

\noindent Clearly,
\[
Q^d\setminus\{x\}\subset\bigcup_{1\le j\le k} S_j.
\]
In particular, for a path $\gamma: [0,1]\rightarrow Q^d\setminus\{x\}$ there is 
a subset $J\subset\{1,\dots, k\}$ such that
\[
\gamma([0,1])\subset\bigcup_{j\in J}S_j\qquad {\rm and}\qquad  S_j\cap \gamma([0,1])\ne\emptyset,\quad j\in J.
\]
Thus, the function 
\[
F_f(t):=\lim_{(x,\gamma(t)]\ni z\rightarrow x}f(z),\qquad t\in [0,1],
\]
attains value $L_{j}$ on the open set $\gamma^{-1}(S_j)$, $j\in J$.   Since these open sets cover the connected interval $[0,1]$, the function $F_f$ must be constant.

\noindent Hence, $f$ has the same radial limits at $x$ along the rays emanating from $x$ passing through $\gamma(0)$ and  $\gamma(1)$. 

Finally, since $Q^d\setminus\{x\}$ is path connected, every two its points are joined by a path. Together with the previous argument this completes the proof of the lemma.
\end{proof}

To proceed with the proof of the theorem, we require the following geometric result.
\begin{Prop}\label{lem4.7}
Let $\{R_n\}_{n\in\N}\subset\RR^d$ be a sequence of rays with the origin at $x\in Q^d$ converging to a ray $R$. If $R_n\cap Q^d\setminus\{x\}\ne\emptyset$ for all $n$, then $R\cap Q^d\setminus\{x\}\ne\emptyset$ as well.
\end{Prop}
\noindent Here convergence is defined by the points of intersection of the rays with the unit sphere centered at $x$. 
\begin{proof}
We apply induction on dimension. 

For $d=1$ the result trivially holds. Assuming that it is valid for all dimensions $\le d-1$, $d\ge 2$, we prove it for $d$. 

To this end, consider intervals $I_n:=[x,x_n]:=R_n\cap Q^d$, $n\in\N$. By the assumption of the lemma $\|x_n-x\|_{\ell_\infty^d}>0$. Passing to a subsequence of $\{R_n\}_{n\in\N}$, if necessary, we assume without loss of generality that all $x_n$ are contained in the same face, say $F$, of $Q^d$ and the sequence $\{x_n\}_{n\in\N}$ converges to a point $\tilde x\in F$. 

First, let us consider $x\not\in F$. Then 
\[
0<\delta=:{\rm dist}_\infty (x,F)\le \varliminf_{n\rightarrow\infty} \|x-x_n\|_{\ell_\infty^d}=\|x-\tilde x\|_{\ell_\infty^d}.
\]
In particular, $(x,\tilde x]$ is a nontrivial interval in $R$, i.e., $R\cap Q^d\ne \emptyset$ in this case. 

Now let $x\in F$. Then all intervals $I_n$ are contained in $F$ as well. This reduces the problem to the case of the $(d-1)$-dimensional cube $F$ and the sequence of rays $\{R_n\}_{n\in\N}$ such that $R_n\cap F\setminus\{x\}\ne\emptyset$ for all $n$.  Clearly, the limit ray $R$ belongs to the hyperplane containing $F$ and, hence, by the induction hypothesis $R\cap F\setminus\{x\}\ne\emptyset$.

This proves  the lemma.
\end{proof}
Finally, let us complete the proof of Theorem \ref{te2.1.3}\,(b).

To this end, it remains to prove for $f\in \dot V_p^k$ that if $\{x_n\}\subset Q^d\setminus\{x\}$ is a sequence converging to $x$  such that $\lim_{n\rightarrow\infty}f(x_n)$ exists, then the limit coincides with the  the radial  limit of $f$ at $x$, see Lemma \ref{lem4.6}.

Passing to a subsequence of $\{x_n\}_{n\in\N}$, if necessary, we  assume without loss of generality that the rays $R_n$ with the origin at $x$ passing through $x_n$, $n\in\N$, converge to a ray $R$. By Proposition \ref{lem4.7}  $R\cap (Q^d\setminus\{x\})\ne\emptyset$. Then there is a hyperplane $H_x$ passing through $x$ orthogonal to a coordinate axis and an open half space $S\subset\RR^d$ with the boundary $H_x$ such that $R\setminus\{x\}\subset S$. We show that for all sufficiently large $n$ points $x_n$ belong to $S$ as well. 

In fact, otherwise the corresponding rays $R_n$ are located in the closed half space $\RR^d\setminus S$ with the boundary $H_x$, hence, $R$ as the limit of the sequence $\{R_n\}_{n\in\N}$ is contained in $\RR^d\setminus S$  as well, a contradiction. 

Now, we proceed as in the proof of Lemma \ref{lem4}. Specifically, we draw hyperplanes through points $x_n$ parallel to $H_x$. They intersect $R$ at some points $y_n$. Then passing to a subsequence  $\{x_j\}_{j\in J\subset\N}$ of $\{x_n\}_{n\in\N}$ we  construct a sequence of nonoverlapping cubes $\{Q_j\}_{j\in J}\subset S\cap Q^d\setminus\{x\}$ converging to $\{x\}$ in the Hausdorff metric  such that $[x_j,y_j]\subset Q_j$, $j\in J$, and
\[
\lim_{J\ni j\rightarrow\infty}\frac{\|x_j-y_j\|_{\ell_\infty^d}}{{\rm diam}(Q_j)}=0.
\]
Using these and Lemma \ref{lemm4.1} and arguing 
 as in the proof of Lemma \ref{lem4} (i.e., applying the mean-value theorem and Markov inequality) we obtain that
\[
\lim_{J\ni j\rightarrow\infty}|f(x_j)-f(y_j)|=0.
\]
This shows that $\lim_{n\rightarrow\infty}f(x_n)$ coincides with the radial limit of $f$ at $x$ and completes the proof of part (b) of the theorem.\smallskip

\noindent (c) We should prove that if $s>k$, then the vector space $\dot V_p^k$  is the direct sum of subspaces $\ell_p$ and $\mathcal P_{k-1}^d$.

Due to parts (a) and (b) of the theorem, each $f\in\dot V_p^k$ is bounded and Lebesgue measurable. In particular, the map $L:\dot V_p^k\rightarrow L_\infty$ which sends a function to its equivalence class in $L_\infty$ is well defined.
\begin{Lm}\label{lem5.8}
It is true that ${\rm range}(L)\subset \dot V_{p\infty}^k$ and the norm of $L:\dot V_p^k\rightarrow \dot V_{p\infty}^k$ is $\le 1$.
\end{Lm}
\begin{proof}
Let us recall that $\dot V_{p\infty}^k$ is the subspace of $L_\infty$ functions on $Q^d$ equipped with the seminorm
\begin{equation}
|f|_{V_{p\infty}^k}:=\sup_{\pi\in\Pi}\left(\sum_{Q\in\pi}E_{k\infty}(f;Q)^p  \right)^{\frac 1 p},
\end{equation}
where
\[
E_{k\infty}(f;Q):=\inf_{m\in\mathcal P_{k-1}^d}\|f-m\|_{L_\infty(Q)}.
\]
Since for each Lebesgue measurable function $f\in\ell_\infty$,
$E_{k\infty}(f;Q)\le E_{k}(f;Q)$,  the required statement follows directly from the definitions of seminorms of $\dot V_p^k$ and $\dot V_{p\infty}^k$.
\end{proof}
Further, since $s>k$, Lemma 3.1 in \cite{BB-18} implies that each $L(f)$ equals a polynomial, say $m_f\in \mathcal P_{k-1}^d$, a.e. on $Q^d$. Since by the definition of the map $L$ for every $m\in \mathcal P_{k-1}^d$ function $L(m)$ equals $m$ a.e. on $Q^d$, map $L:\dot V_{p}^k\rightarrow\dot V_{p\infty}^k$ is surjective and
$L|_{\mathcal P_{k-1}^d}:\mathcal P_{k-1}^d\rightarrow \dot V_{p\infty}^k$ is an isomorphism; hence, ${\rm ker}(L)\cap\mathcal P_{k-1}^d=(0)$. This implies that $\dot V_p^k$ is isomorphic to the direct sum of ${\rm ker}(L)$ and $\mathcal P_{k-1}^d$. To complete the proof of  part (c) it remains to show the following.
\begin{Lm}\label{lem5.9}
${\rm ker}(L)$ coincides with $\ell_p$.
\end{Lm}
\begin{proof}
First, let us show that $\ell_p\subset\dot V_p^k$. 

In fact, we have for $f\in\ell_p$
\[
|f|_{V_p^k}:=\sup_{\pi\in\Pi}\left(\sum_{Q\in\pi}E_k(f;Q)^p\right)^{\frac 1 p}\le 
\left(\sum_{Q\in\pi}\left\{\sup_{Q}|f|\right\}^p\right)^{\frac 1 p}\le \|f\|_{p},
\]
i.e., $f\in\dot V_p^k$ as required.

Clearly, $L(f)=0$ for each $f\in\ell_p$, i.e., $\ell_p\subset {\rm ker}(L)$. 

Conversely, if $f\in {\rm ker}(L)$, then $f$ equals $0$ a.e.\,on $Q^d$ and,
moreover, this and Lemmas \ref{lemm4.1}, \ref{lem1} imply that ${\rm supp}\, f$ is an at most countable subset of $Q^d$. 

Let $S\subset {\rm supp}\, f$ be a finite subset. By definition the function  
\[
f_S(y):=\left\{
\begin{array}{ccc}
0&{\rm if}&y\in S\\
f(y)&{\rm if}&y\not\in S
\end{array}
\right.
\]
is continuous at each point of $S$. Due to Proposition  \ref{lem2} we then have, see \eqref{fk},
\begin{equation}\label{eq5.10}
(f-f_S)_k(y)\le f_k(y)+(f_S)_k(y)= f_k(y)\quad {\rm for\ all}\quad y\in S.
\end{equation}
On the other hand,  $f-f_S$ is supported by $S$ and $(f-f_S)(y)=f(y)$ for each $y\in S$. The direct computation of the limit
in \eqref{fk} gives for this case
\begin{equation}\label{eq5.11}
(f-f_S)_k(y)=\frac 1 2 |f(y)|,\quad y\in S.
\end{equation}
Now \eqref{eq5.10}, \eqref{eq5.11} and Lemma \ref{lem1} imply that
\[
\frac 1 2 \left(\sum_{y\in S}|f(y)|^p\right)^{\frac 1 p}=\|(f-f_S)_k\|_{p}\le  \|f_k\|_{p}\le |f|_{V_p^k}.
\]
This shows that $f\in\ell_p$ and $\|f\|_p\le 2|f|_{V_p^k}$ and implies that ${\rm ker}(L)\subset\ell_p$.
\end{proof}

The proof of Theorem \ref{te2.1.3} is complete.
\end{proof}

\subsection{Proof of Corollary \ref{cor2.1.3}}
\begin{proof}
(a) We begin with the case of $d\ge 2$. Then by definition
\[
\hat f(x):=\lim_{y\rightarrow x} f(y),\quad x\in Q^d,
\]
and $P$ maps $f\in\dot V_p^k$ to $\hat f$.

We should prove that $P$ is a linear projection of norm $1$ onto the closed subspace $C\cap \dot V_p^k $ denoted by $N\dot V_p^k$. 

First we show that for $f\in\dot V_p^k$ the function $\hat f\in\dot V_p^k$ and $|\hat f|_{V_p^k}\le |f|_{V_p^k}$. Since clearly $P\hat f=\hat f$ and $C^\infty\subset N\dot V_p^k$ (see the proof of  Theorem \ref{te2.1.5} in Section~5), this will prove the corollary for this case.

To this end, it suffices to prove that for each closed cube $Q\subset Q^d$
\begin{equation}\label{e5.12a}
E_k(\hat f;Q)\le E_k(f;Q).
\end{equation}

Let $m\in\mathcal P_{k-1}^d$ be such that
\begin{equation}\label{e5.12}
E_k(f;Q)=\|f -m;Q\|_{\infty}.
\end{equation}

Since by Theorem \ref{te2.1.3} function $\hat f$ is continuous on $Q^d$ and $f-\hat f$ has at most countable support, there are $x\in Q^d$ such that
\begin{equation}\label{e5.13}
|\hat f(x)-m(x)|=\|\hat f-m;Q\|_{\infty}
\end{equation}
and a sequence $\{x_n\}_{n\in\N}\not\subset {\rm supp} (f-\hat f)$ in $Q$ converging to $x$. Hence, due to the continuity of $\hat f$,
\begin{equation}\label{e5.14}
\lim_{n\rightarrow\infty}|\hat f(x_n)-m(x_n)|=|\hat f(x)-m(x)|.
\end{equation}
Using \eqref{e5.12}, \eqref{e5.13} and \eqref{e5.14}, and the definition of $\hat f$ we obtain
\[
E_k(\hat f;Q)\le \|\hat f-m;Q\|_{\infty}=\lim_{n\rightarrow\infty}|\hat f(x_n)-m(x_n)|=
\lim_{n\rightarrow\infty}|f(x_n)-m(x_n)|\le E_k(f;Q)
\]
proving \eqref{e5.12a} and assertion (a) in this case.

Now let $d=1$, hence, $m$ in \eqref{e5.12} belongs to $\mathcal P_{k-1}^d$ and $Q$ is a closed interval in $[0,1]$. Moreover, $\hat f(x):=f(x)$ if $f$ is a point of continuity for $f$, and $\hat f(x):=\frac 1 2 \bigl(f(x^-)+f(x^+)\bigr)$ for the remaining 
at most countable set of points $x$ in $(0,1)$ (with the corresponding correction for $x=0,1$). As in the previous case, we  show that
\eqref{e5.12a} holds for $d=1$.

Let $\{x_n\}_{n\in\N}\subset  Q$ be such that
\begin{equation}\label{e5.16}
\lim_{n\rightarrow\infty}|\hat f(x_n)-m(x_n)|=\|\hat f-m;Q\|_{\infty}.
\end{equation}
If $\{x_n\}_{n\in\N}$ contains an infinite subsequence of points of continuity of $\hat f$, then we can slightly change the latter to get a sequence outside of ${\rm supp} (f-\hat f)$ (which is at most countable) satisfying \eqref{e5.16}. In this case the argument of the proof of the previous case gives the required inequality \eqref{e5.12a}. 

\noindent Otherwise, $\{x_n\}_{n\in\N}$ contains an infinite subsequence, say $\{x_{n_k}\}_{k\in\N}$,  of points of discontinuity for $f$ in $(0,1)$. Then 
by the definition of $\hat f$ we derive from \eqref{e5.16}:
\[
\begin{array}{l}
\displaystyle
E_k(\hat f;Q)\le \|\hat f-m;Q\|_{\infty}=\lim_{k\rightarrow\infty}|\hat f(x_{n_k})-m(x_{n_k})|\\
\\
\displaystyle \qquad\qquad=
\lim_{k\rightarrow\infty}\left| \frac{f({x_{n_k}}^{\!\!-})-m(x_{n_k})}{2}+\frac{f({x_{n_k}}^{\!\!+})-m(x_{n_k})}{2}\right|\le \|f-m;Q\|_{\infty}=E(f;Q),
\end{array}
\] 
as required.

\noindent Now, let $x_n=0$ for all $n\in\N$. Then we have by \eqref{e5.16} and the definition of $\hat f$
\[
E_k(\hat f;Q)\le \|\hat f-m;Q\|_{\infty}=|\hat f(0)-m(0)|:=\lim_{x\rightarrow 0^+}
|f(x)-m(x)|\le \|f-m;Q\|_{\infty}=E(f;Q).
\]
The case of $x_n=1$ for all $n\in\N$ is considered similarly.

The proof of \eqref{e5.12a} is complete, i.e., $P:\dot V_p^k\rightarrow N\dot V_p^k\, (\subset\dot V_p^k)$ is a linear projection of norm $1$ also for $d=1$. \smallskip

\noindent (b) By the definition of the map $P$ its kernel contains functions in $\dot V_p^k$ with at most countable supports. Then as in the proof of part (c) of Theorem \ref{te2.1.3} we obtain that ${\rm ker}(P)=\ell_p$. Moreover, we have proved in Lemma \ref{lem5.9} that for each $f\in\ell_p$
\begin{equation}\label{e5.17}
|f|_{V_p^k}\le \|f\|_p\le 2|f|_{V_p^k}.
\end{equation}
This shows that the identity map embeds $\ell_p$ into $\dot V_p^k$ as a Banach subspace.

The proof of the corollary is complete.
\end{proof}
\sect{Proofs of Theorems \ref{te2.1.4},  \ref{te2.6} and Corollary \ref{helly}}
In the proofs, we use the duality between $V_p^k$ and $U_p^k$ established in Theorem \ref{te2.2.6} and between $V_{p\infty}^k$ and $U_{p\infty}^k$ established in \cite[Th.\,2.6]{BB-18}. 
\subsection{Auxiliary Results}
\subsubsection{}
In what follows, we fix an interpolating subset $S\subset Q^d$ for polynomials of $\mathcal P_{k-1}^d$. Hence, ${\rm card}\, S={\rm dim}\,\mathcal P_{k-1}^d+1$ and for each function $f\in \ell_\infty(S)$ there is a unique polynomial $m_f\in\mathcal P_{k-1}^d$ such that $m_f|_{S}=f$. The norm of the linear operator $f\mapsto m_f$ is clearly bounded by a constant $c=c(S)$. 

Let $\dot V_{p;S}^k\subset\dot V_p^k$ denote the subspace of functions vanishing on $S$. Clearly, ${\rm codim}\,\dot V_{p;S}^k={\rm card}\,S$ and for each $f\in \dot V_p^k$ we have the unique representation $f=m_f+(f-m_f)$ where $f-m_f\in \dot V_{p;S}^k$. This implies that $\dot V_{p;S}^k$ equipped with the seminorm $|\cdot|_{V_p^k}$ is a Banach space such that 
\begin{equation}\label{eq6.0}
\dot V_{p;S}^k\equiv V_p^k,
\end{equation}
 where the isometry is given by the restriction of the quotient map $q:\dot V_p^k\rightarrow  V_p^k\, (:=\dot V_p^k/\mathcal P_{k-1}^d)$ to $\dot V_{p;S}^k$. We equip $\dot V_{p;S}^k$ with the weak$^*$ topology pulled back from that of $V_p^k$ by this isometry. The weak$^*$ topology on $V_p^k$ is defined by the duality relation
 \begin{equation}\label{eq-n4.2}
 V_p^k\equiv (U_p^k)^*
 \end{equation}
 that will be proved in Theorem  \ref{te2.2.6}.
\begin{Lm}\label{lemma6.1}
The closed unit ball $B(\dot V_{p;S}^k)$  is compact in the topology of pointwise convergence on $Q^d$.
\end{Lm}
\begin{proof}
Since $B(\dot V_{p;S}^k)$ is contained in $B(\ell_\infty)=[-1,1]^{Q^d}$ which is compact in the topology of pointwise convergence by  the Tychonoff theorem, $B(\dot V_{p;S}^k)$
is relatively compact in this topology. If now
$\{f_\alpha\}_{\alpha\in\Lambda}$ is a net in $B(\dot V_{p;S}^k)$ converging in the topology of pointwise convergence to a function $f\in\ell_\infty$, then we must show that $f\in B(\dot V_{p;S}^k)$. 

Since by the Banach-Alaoglu theorem $B(\dot V_{p;S}^k)$ is weak$^*$ compact, see \eqref{eq-n4.2}, we may assume, passing if necessary to a subnet of $\{f_\alpha\}_{\alpha\in\Lambda}$, that $\{f_\alpha\}_{\alpha\in\Lambda}$ converges in the weak$^*$ topology to some $\tilde f\in  B(\dot V_{p;S}^k)$. 

Let us show that $f=\tilde f$. 

In fact, let $\delta_x$ be the evaluation functional at $x\in Q^d$. By the definition of $S$, there exists  an element $c_x\in\ell_1(S)$ such that
\[
m(x)=\sum_{s\in S}c_x(s)m(s)\quad {\rm for\ all}\quad m\in\mathcal P_{k-1}^d.
\]
Hence, $\delta_x':=\delta_x-\sum_{s\in S}c_x(s)\delta(s)$ regarded as a function of $\ell_1$ supported by $S\cup\{x\}$ is
orthogonal to $\mathcal P_{k-1}^d$; hence, $\delta_x'$ belongs to $(U_p^k)^0\subset U_p^k$, see \eqref{2.2.3}. 

Now the weak$^*$ convergence of $\{f_\alpha\}_{\alpha\in\Lambda}$ in $B(\dot V_{p;S}^k)$ and its pointwise convergence  to $f$ imply for each $x\in Q^d$,
\[
\tilde f(x)=\langle f, \delta_x'\rangle=\lim_{\alpha}\langle f_\alpha, \delta_x'\rangle
=\lim_{\alpha}f_\alpha(x)=f(x).
\]
Hence, $\tilde f=f$, as required.
\end{proof}
\subsubsection{}
In the next proofs, we use the approximation theorem from  \cite[Th.\,2.12]{BB-18} asserting that for every $g\in\dot V_{p\infty}^k$ there is a sequence $\{g_n\}_{n\in\N}\subset C^\infty$ such that
\begin{itemize}
\item[(1)]
\[
\lim_{n\rightarrow\infty}|g_n|_{V_{p\infty}^k}=|g|_{V_{p\infty}^k};
\]
\item[(2)]
\[
\lim_{n\rightarrow\infty}\int_{Q^d}(g-g_n)h\,dx=0\quad  {\rm for\ each}\quad h\in L_1.
\]
\end{itemize}

The sequence $\{g_n\}_{n\in\N}$ is defined as follows.

Let $\varphi\in C^\infty(\RR^d)$ be a function supported by the  unit cube $[-1,1]^d$ such that
\[
0\le\varphi\le 1\quad {\rm and}\quad \int_{\RR^d}\varphi\,dx=1.
\]
We denote by $g^0$ the extension of $g\in L_\infty$ by $0$ outside $Q^d$ and then define $g_n^0$, $n\in\mathbb N$, by \begin{equation}\label{eq6.2}
g_n^0(x):=\int_{\|y\|_{\ell_\infty^d}\le 1}g^0\bigl(x-\mbox{$\frac{y}{n+1}$}\bigr)\,\varphi(y)\, dy,\quad x\in\RR^d.
\end{equation}

Next, let $a_n:\RR^d\rightarrow\RR^d$ be the $\lambda_n$-dilation with the center $c$, where $\lambda_n:=\frac{n-1}{n+1}$ and $c$ is the center of $Q^d$.

Then $a_n(Q^d)$ is a subcube of $Q^d$ centered at $c_{Q^d}$ and 
\begin{equation}\label{eq6.3}
{\rm dist}_{\ell_\infty^d} (\partial Q^d, a_n(Q^d))=\frac{1}{n+1}.
\end{equation}

We define $g_n:=T_n(g)\in C^\infty$ by setting 
\begin{equation}\label{eq6.4}
g_n=g_n^0\circ a_n.
\end{equation}
\subsection{Proof of Theorem \ref{te2.1.4}}
Let $f\in\dot V_p^k$. We should find a sequence $\{f_n\}_{n\in\N}\subset C^\infty$ satisfying the limit relations \eqref{2.1.4}, \eqref{2.1.5}.

Let $L:\dot V_p^k\rightarrow L_\infty$ be the map sending a function of $\dot V_p^k$ to its equivalence class in $L_\infty$.
Since
\[
\|L(f)\|_{L_\infty(Q)}\le \|f\|_{\ell_\infty(Q)}\quad {\rm for\ all}\quad f\in \dot V_p^k,\quad Q\subset Q^d,
\]
the range of $L$ is a subspace of $\dot V_{p\infty}^k$ and the linear map $L:\dot V_p^k\rightarrow
\dot V_{p\infty}^k$ is bounded of norm $\le 1$, cf. Lemma \ref{lem5.8}. 

Let $f\in\dot V_p^k$ and $g:=L(f)\in\dot V_{p\infty}^k$. Then there is a sequence $\{g_n\}_{n\in\N}\subset C^\infty$ satisfying conditions (1), (2) of Section~4.1.2.

Further, $L$  maps $C^\infty\subset \dot V_p^k$ isometrically onto $C^\infty\subset \dot V_{p\infty}^k$ (because $E_{k\infty}(f;Q)=E_k(f;Q)$ for all $f\in C^\infty$, cf. the argument of the proof of Lemma \ref{lem5.8}). In particular, there exists a sequence $\{\tilde f_n\}_{n\in\N}\subset C^\infty\subset\dot V_{p}^k$ such that $L(\tilde f_n)=g_n$ for all $n$ and 
\begin{equation}\label{equ6.14}
\lim_{n\rightarrow\infty}|\tilde f_n|_{V_p^k}=\lim_{n\rightarrow\infty}|g_n|_{V_{p\infty}^k}=|g|_{V_{p\infty}^k}\le |f|_{V_p^k}.
\end{equation}
\begin{Lm}\label{lem6.2}
For each point of continuity $x\in Q^d$ of $f$
\begin{equation}\label{e6.5}
\lim_{n\rightarrow\infty}\tilde f_n(x)=f(x).
\end{equation}
\end{Lm}
\begin{proof}
We have by the definitions of $L, a_n, g_n, \tilde f_n$, see \eqref{eq6.2} -- \eqref{eq6.4},
\begin{equation}\label{e6.6}
\begin{array}{l}
\displaystyle
|\tilde f_n(x)-f(x)|=\left|\int_{\|y\|_{\ell_\infty^d}\le 1}\left(f^0\bigl(a_n(x)-\mbox{$\frac{y}{n+1}$}\bigr)-f(x)\right)\varphi(y)\, dy\right|\\
\\
\displaystyle
=\left|\int_{\|y\|_{\ell_\infty^d}\le 1}\left(f\bigl(a_n(x)-\mbox{$\frac{y}{n+1}$}\bigr)-f(x)\right)\varphi(y)\, dy\right| \le\sup_{\|y\|_{\ell_\infty^d}}\left|f\bigl(a_n(x)-\mbox{$\frac{y}{n+1}$}\bigr)-f(x)\right|.
\end{array}
\end{equation}
Since $a_n(x)\rightarrow x$ as $n\rightarrow\infty$, the last term in \eqref{e6.6} tends to $0$ as $n\rightarrow\infty$ by continuity of $f$ at $x$.

The lemma is proved.
\end{proof}

Further, by the definition of the sequence $\{g_n\}_{n\in\N}$, see \eqref{eq6.2}, \eqref{eq6.4},
\begin{equation}\label{e6.8}
\|\tilde f_n\|_\infty\le \|g\|_{L_\infty}\, (\le\|f\|_\infty)\quad {\rm for\ all}\quad n\in\N.
\end{equation}
Since by Theorem \ref{te2.1.3} the set $S_f$ of discontinuities of $f$ is at most countable, the Cantor diagonal procedure gives a subsequence $\{\tilde f_{n_k}\}_{k\in\N}$ of $\{\tilde f_n\}_{n\in\N}$ such that the limit
\[
\lim_{k\rightarrow\infty}\tilde f_{n_k}(x)\le  \|g\|_{L_\infty}
\]
exists for all $x\in S_f$.

We define $\tilde f\in \ell_\infty$ as the pointwise limit of the sequence $\{\tilde f_{n_k}\}_{k\in\N}$ (in particular, $\tilde f$ coincides with $f$ at the points of continuity of $f$).
\begin{Lm}\label{lem6.3}
$\tilde f\in \dot V_p^k$ and $|\tilde f|_{V_p^k}\le |f|_{V_p^k}$.
\end{Lm}
\begin{proof}
We set for brevity 
\begin{equation}\label{e6.9}
h_k:=\tilde f_{n_k},\qquad k\in\N,
\end{equation}
and use the interpolating set $S\subset Q^d$ for $\mathcal P_{k-1}^d$ of Section~4.1.1
to decompose $h_k$ as follows
\[
h_k=m_{h_k}+(h_k-m_{h_k}),
\]
where $m_{h_k}\subset\mathcal P_{k-1}^d$ interpolates $h_{k}$ on $S$.

The second term belongs to the space $\dot V_{p;S}^k$ consisting of all functions from $\dot V_p^k$ vanishing on $S$ so that $|h_{k}-m_{h_{k}}|_{V_p^k}=|h_{k}|_{ V_p^k}$.
Further, due to  \eqref{e6.8}
\[
\|m_{h_{k}};S\|_\infty\le \|g\|_{L_\infty}.
\]
Hence, the sequence $\{m_{h_{k}}\}_{k\in\N}$ is uniformly bounded on $Q^d$ and, in particular, it contains a uniformly converging subsequence. Passing to such a subsequence, if necessary, without loss of generality we will assume that $\{m_{h_k}\}_{k\in\N}$ itself converges uniformly to a polynomial $ m\in\mathcal P_{k-1}^d$. Therefore the sequence $\{h_k-m_{h_k}\}_{k\in\N}$ is uniformly bounded and pointwise converges to  $\tilde f- m\in\ell_\infty$. According to  Lemma \ref{lemma6.1}, the closed ball $B(\dot V_{p;S}^k)$ is compact in the topology of pointwise convergence on $Q^d$.
Hence, $\tilde f-m\in V_{p;S}^k$ and by \eqref{equ6.14} 
\[
|\tilde f|_{V_p^k}=|\tilde f-m|_{V_p^k}\le\varlimsup_{k\rightarrow\infty}|h_k-m_{h_k}|_{V_p^k}= \varlimsup_{k\rightarrow\infty}|h_k|_{V_p^k}\le |f|_{V_p^k}.
\]

The proof of the lemma is complete.
\end{proof}

Now set 
\begin{equation}\label{e6.10}
l:=f-\tilde f.
\end{equation}
Since $\tilde f$ coincides with $f$ on the continuity set of $f$, see Lemma \ref{lem6.2}, function $l\in\ell_p$ and by Lemma \ref{lem6.3}, 
\begin{equation}\label{e6.11}
|l|_{V_p^k}\le 2|f|_{V_p^k}.
\end{equation}
\begin{Lm}\label{lem6.4}
Let $l\in\ell_p$. There is a sequence $\{l_n\}_{n\in\N}\subset C^\infty$ pointwise converging to $l$ such that
\begin{equation}\label{eq-n4.13}
\sup_{n\in\N}\|l_n\|_\infty\le \|l\|_\infty\quad {\rm and}\quad  \varlimsup_{n\rightarrow\infty}|l_n|_{V_p^k}\le \|l\|_p. 
\end{equation}
\end{Lm}
\begin{proof}
Since the set of  functions with finite supports is dense in $\ell_p$, without loss of generality we may assume that $l$ is supported by a finite set, say $\{x_1,\dots, x_m\}\subset Q^d$. In this case, 
Euclidean balls of radius $r$ centered at points $x_i$ are mutually disjoint for some $r>0$.

Let $s_{jn}$ be a $C^\infty$ function supported by the Euclidean ball $B_{\frac r n}(x_j)$ such that $s_{ji}(x_j)=1$ and $0\le s_{ji}\le 1$, $n\in\N$, $1\le j\le m$. 

We set
\begin{equation}\label{elp}
l_n:=\sum_{l=1}^m l(x_j)s_{jn}.
\end{equation}
By the definition, for a nontrivial closed cube $Q\subset Q^d$
\[
E_{k}(l_n;Q)\le \|l_n;Q\|_{\infty}\le \|l;Q\|_\infty\le \left(\sum_{x_i\in Q}|l(x_i)|^p\right)^{\frac 1 p}.
\]
This implies that for all $n$
\[
\|l_n\|_\infty\le \|l\|_\infty\quad {\rm and}\quad 
 |l_n|_{V_p^k}\le \|l\|_{p}.
\]
and proves the required inequalities.

It remains to prove that 
\[
\lim_{n\rightarrow\infty}(l(y)-l_n(y))=0\quad {\rm for\ all}\quad y\in Q^d.
\]

Indeed, if $y$ coincides with some $x_j$, $1\le j\le m$, then
\[
l_n(y)=l(y)\quad {\rm for\ all}\quad n
\]
and the result follows.

Otherwise,  $l(y)=0$ and there is some $n_0\in\N$ such that $y\not\in {\rm supp}\,l_n$ for all $n\ge n_0$. Hence,
$l_n(y)=0$ for such $n$ and
\[
\lim_{n\rightarrow\infty}(l(y)-l_n(y))=0
\]
as well.

The proof is complete.
\end{proof}

Let $\{l_k\}_{k\in\N}\subset C^\infty$ and $\{h_k\}_{k\in\N}\subset C^\infty$ be as above, see \eqref{eq-n4.13}, \eqref{e6.9}. Then we set
\begin{equation}\label{e6.13}
f_n:=h_n+l_n,\qquad n\in\N.
\end{equation}
and show that $\{f_k\}_{k\in\N}$ is the required sequence.

In fact, by \eqref{e6.8}, \eqref{e6.10} and \eqref{eq-n4.13}
\[
\varlimsup_{n\rightarrow\infty}\|f_n\|_\infty\le \varlimsup_{n\rightarrow\infty}\|h_n\|_\infty+\varlimsup_{n\rightarrow\infty}\|l_n\|_\infty\le \|f\|_\infty+\|f-\tilde f\|_{\infty}\le 3\|f\|_\infty.
\]
This give the first inequality \eqref{2.1.4}.

Further, by \eqref{equ6.14} and \eqref{eq-n4.13}
\[
\varlimsup_{n\rightarrow\infty}\|h_n\|_{V_p^k}+\varlimsup_{n\rightarrow\infty}\| l_n\|_{V_p^k}\le \|f\|_{V_p^k}+\|l\|_{p}.
\]
Moreover, applying the inequality proved in Lemma \ref{lem5.9}  and \eqref{e6.10} we get
\[
\|l\|_p\le 2| l |_{V_p^k}\le 4|f|_{V_p^k}.
\]
Combining these inequalities we finally have
\[
\varlimsup_{n\rightarrow\infty}\|f_n\|_{V_p^k}\le 5\|f\|_{V_p^k}
\]

This proves the second  inequality \eqref{2.1.4}.

Further, by its definition the sequence $\{f_k\}_{k\in\N}\subset C^\infty $ converges pointwise to $f$, hence, belongs to the first Baire class, and is uniformly bounded. By the Rosenthal Main Theorem \cite{Ro-77} every such sequence satisfies
\begin{equation}\label{equ6.2a}
\lim_{k\rightarrow\infty}\int_{Q^d}(f-f_k) \,d\mu=0
\end{equation}
for all measures $\mu\in M$.

This proves \eqref{2.1.5} and completes the proof of the theorem.
\subsection{Proof of Corollary \ref{helly}}
We retain notation of Section~4.1.1.

Let $\{f_n\}_{n\in\N}\subset  \dot V_p^k$ be bounded. 
We choose a subsequence $\{f_{n_i}\}_{i\in\N}$ such that
\begin{equation}\label{eq-n4.17}
\lim_{i\rightarrow\infty}|f_{n_i}|_{V_p^k}=\varliminf_{n\rightarrow\infty}|f_n|_{V_p^k}.
\end{equation}

Next, we write
\[
f_{n_i}=m_{f_{n_i}}+(f_{n_i}-m_{f_{n_i}}),
\]
where $m_{f_{n_i}}\in \mathcal P_{k-1}^d$, $f_{n_{i}}-m_{f_{n_i}}\in \dot V_{p;S}^k$,
$i\in\N$. 

By the definition,
\[
\sup_{i\in\N}|f_{n_i}-m_{f_{n_i}}|_{V_p^k}=\sup_{i\in\N}|f_{n_i}|_{V_p^k}\le \sup_{n\in\N}|f_n|_{V_p^k}=:r<\infty.
\]
Hence, the sequence $\{f_{n_i}-m_{f_{n_i}}\}_{i\in\N}$ belongs to the closed ball of $\dot V_{p;S}^k$ of radius $r$ centered at $0$ and so by Lemma \ref{lemma6.1} is relatively  compact in the ball  equipped with the topology of pointwise convergence on $Q^d$. In addition, by Theorem \ref{te2.1.4} this sequence is in the first Baire class. Therefore the Rosenthal Main Theorem \cite{Ro-77} imply that it contains a pointwise convergent subsequence $\{f_{n_j}-m_{f_{n_j}}\}_{j\in J}$, $J\subset\N$. Due to \eqref{eq-n4.17} its limit, say $f\in \dot V_{S;p}^k$,
belongs to the ball of radius $\varliminf_{n\in \N}|f_n|_{V_p^k}$.

This completes the proof of the corollary.
\subsection{Proof of Theorem \ref{te2.6}}
(a) Let as above $L:\dot V_p^k\rightarrow L_\infty$ be the map sending a function in $\dot V_p^k$ to its equivalence class in $L_\infty$.
It was shown in Section~4.2 that the range of $L$ is a subset of $\dot V_{p\infty}^k$ and the linear map $L:\dot V_p^k\rightarrow\dot V_{p\infty}^k$ has norm $\le 1$. We have to check that the range of $L$ is $\dot V_{p\infty}^k$ and that $\|L\|=1$.

Let $f\in \dot V_{p\infty}^k$. Due to \cite[Th.\,2.12]{BB-18} there is a sequence $\{g_n\}_{n\in\N}\subset C^\infty\subset\dot V_p^k$ such that
\begin{equation}\label{e6.15}
\sup_{n\in\N}\|g_n\|_\infty\le \|f\|_{L_\infty}\quad {\rm and}\quad \lim_{n\rightarrow\infty} |g_n|_{V_p^k}=|f|_{V_{p\infty}^k},
\end{equation}
and, moreover, $\{L(g_n)\}_{n\in\N}\subset C^\infty\subset V_{p\infty}^k$ weak$^*$ converges to $f$ in the weak$^*$ topology of $L_\infty$ defined by the duality $L_1^*=L_\infty$.

Using the interpolating set $S\subset Q^d$ for $\mathcal P_{k-1}^d$ of Section~4.1.1
we decompose $g_n$ in a sum
\[
g_n=m_{g_n}+(g_n-m_{g_n}),
\]
where $m_{g_n}\subset\mathcal P_{k-1}^d$ interpolates $g_{n}$ on $S$
and the second term belongs to the space $\dot V_{p;S}^k$ of functions from $\dot V_p^k$ vanishing on $S$ so that $|g_{n}-m_{g_{n}}|_{V_p^k}=|g_{n}|_{ V_p^k}$.

According to \eqref{e6.15} the sequence $\{m_{g_n}\}_{n\in\N}$ is uniformly bounded on the interpolating set $S$, hence, on $Q^d$.  Passing to a subsequence of $\{g_n\}_{n\in\N}$, if necessary, without loss of generality we may assume that $\{m_{g_n}\}_{n\in\N}$ converges uniformly on $Q^d$ to some $m\in\mathcal P_{k-1}^d$. Moreover, \eqref{e6.15} implies that
\[
\lim_{n\rightarrow\infty} |g_n-m_{g_n}|_{V_p^k}=|f|_{V_{p\infty}^k},
\]
i.e., the sequence $\{g_n-m_{g_n}\}_{n\in\N}$ is bounded in $\dot V_{p;S}^k$. Using the weak$^*$ topology on $\dot V_{p;S}^k$ induced by the duality between $\dot V_{p;S}^k\, (\equiv V_p^k)$ and $U_p^k$, we find a subnet $\{n_\alpha\}_{\alpha\in\Lambda}$ of $\N$ such that the subnet
 $\{g_{n_\alpha}-m_{h_{n_\alpha}}\}_{\alpha\in\Lambda}$  weak$^*$ converges to some $h\in \dot V_{p;S}^k$. Since, in addition, the subnet $\{m_{g_{n_\alpha}}\}_{\alpha\in\Lambda}$  converges uniformly on $Q^d$ to $m$, the subnet 
$\{g_{n_\alpha}\}_{\alpha\in\Lambda}$ of $\{g_n\}_{n\in\N}$ pointwise converges to $g:=h+m$, cf. the argument of Lemma \ref{lemma6.1}. Moreover, this subnet is uniformly bounded in $\ell_\infty$ by \eqref{e6.15}.

Hence, by the Rosenthal Main Theorem \cite{Ro-77} 
\[
\lim_{\alpha}\int_{Q^d}g_{n_\alpha} \,d\mu=\int_{Q^d}g\,d\mu
\]
for all measures $\mu\in M$.  

Taking here $d\mu= h\,dx$, $h\in L_1$, we obtain that
\[
\lim_{\alpha}\int_{Q^d}g_{n_\alpha}h\,dx=\int_{Q^d}gh\,dx\quad {\rm for\ all}\quad h\in L_1.
\]
The latter implies that the net $\{L(g_{n_\alpha})\}_{\alpha\in\Lambda}$ weak$^*$ converges to $L(g)$.  Hence, by \eqref{e6.15}, $L(g)=f$. 

This proves surjectivity of the map $L:\dot V_p^k\rightarrow\dot V_{p\infty}^k$.

Finally, by \eqref{e6.15} and the inequality $\|L\|\le 1$ we get
\[
|g|_{V_p^k}=|L(g)|_{V_{p\infty}^k}\le |g|_{V_p^k},
\]
i.e., $|L(g)|_{V_{p\infty}^k}=|g|_{V_p^k}$. This implies that $\|L\|=1$.

The proof of part (a) of the theorem is complete.\smallskip

\noindent (b) Now we should prove that ${\rm ker}(L)={\rm ker}(P)=\ell_p$.

By the definition of the map $L$ its kernel consists of all functions in $\dot V_p^k$ which are zeros outside sets of Lebesgue measure zero. Then by Lemma \ref{lem5.9}  ${\rm ker}(L)=\ell_p$. 
The same is established in Corollary \ref{cor2.1.3}\,(b)  for ${\rm ker}(P)$.

This completes the proof of this part of the theorem.\smallskip

\noindent (c) It remains to show that $L$ maps $N\dot V_p^k$ 
isometrically onto $\dot V_{p\infty}^k$.

As in  part (a) for each $f\in \dot V_{p\infty}^k$ we take $g\in \dot V_p^k$ such that $L(g)=f$ and $|g|_{V_p^k}=|f|_{V_{p\infty}^k}$. 

\noindent Then we set
\[
h:=P(g)\, (\in N\dot V_p^k).
\]
Since $h-g\in\ell_p$, see Corollary \ref{cor2.1.3}, $\|L\|=1$ and $\|P\|=1$,
\[
L(h)=L(g)=f\quad {\rm and}\quad |f|_{V_{p\infty}^k}\le |h|_{V_p^k}\le |g|_{V_p^k}=|f|_{V_{p\infty}^k}.
\]
Moreover, ${\rm ker}(L)\cap N\dot V_p^k=\ell_p\cap N\dot V_p^k=(0)$ by Corollary \ref{cor2.1.3}. These imply existence for each $f\in\dot V_{p\infty}^k$ a unique $h\in N\dot V_p^k$ such that $L(h)=f$ and $|h|_{V_p^k}=|f|_{V_{p\infty}^k}$.

Thus,  $L: N\dot V_p^k\rightarrow V_{p\infty}^k$ is an isometric isomorphism of semi-Banach spaces. 

The proof of Theorem \ref{te2.6} is complete.
\sect{Proof of Theorem \ref{te2.1.5}}
Let $(\dot V_p^k)^0\subset\dot V_p^k$ denote the subset of functions $f$ satisfying the condition
\begin{equation}\label{e7.1}
\lim_{\varepsilon\rightarrow 0}\sup_{d(\pi)\le\varepsilon}\left(\sum_{Q\in\pi} E_k(f;Q)^p\right)^{\frac 1 p}=0;
\end{equation}
where $d(\pi):=\sup_{Q\in\pi}|Q|$.

We have to prove that if $k>s:=\frac d p $, then $(\dot V_p^k)^0=\dot{\textsc{v}}_p^k$. 

First, let us show that $(\dot V_p^k)^0$ {\em is a closed subspace of the space} $\dot V_p^k$.

To this end we define a seminorm  $T: \dot V_p^k\rightarrow\RR_+$ given for $ f\in \dot V_p^k$ by 
\[
T( f):=\lim_{\varepsilon\rightarrow 0}\sup_{d(\pi)\le\varepsilon}\left(\sum_{Q\in\pi}E_{k}(f;Q)^p\right)^{\frac 1 p}.
\]
Since $T(f)\le | f|_{V_p^k}$ for all $f\in \dot V_p^k$, seminorm $T$ is continuous on $\dot V_p^k$. This implies that the preimage $T^{-1}(\{0\})=(\dot V_p^k)^0$ is a closed subspace of $\dot V_p^k$.\smallskip

Next, we show that $\dot{\textsc{v}}_p^k$ {\em is a closed subspace of} $(\dot V_p^k)^0$.\smallskip

Since $\dot{\textsc{v}}_p^k={\rm clos}(C^\infty\cap\dot V_p^k,\dot V_p^k)$ and $(\dot V_p^k)^0$ is closed in $\dot V_p^k$, it suffices to prove that $C^\infty\subset (\dot V_p^k)^0$.

To this end we estimate  
$E_{k}(f;Q)$, $Q\in\pi$, with $f\in C^\infty$ by the Taylor formula as follows
\begin{equation}\label{eq7.5}
E_{k}(f; Q)\le c(k,d)|Q|^{\frac k d }\max_{|\alpha|=k}\max_Q |D^\alpha f|\le c(k,d,f)|Q|^{\frac k d }.
\end{equation}
This implies that 
\begin{equation}\label{eq7.6}
\gamma(\pi;f)\le c\left(\sum_{Q\in\pi}|Q|^{\frac{k-s}{d}p+1}\right)^{\frac 1 p}\le c\max_{Q\in\pi}|Q|^{\frac{k-s}{d}}\left(\sum_{Q\in\pi}|Q|\right)^{\frac 1 p};
\end{equation}
hereafter we set
\begin{equation}\label{eq-n5.4}
\gamma(\pi;f):=\left(\sum_{Q\in\pi}E_k(f;Q)^p\right)^{\frac 1 p}.
\end{equation}
Since $s<k$ and the sum here $\le 1$, this implies that
\[
\sup_{d(\pi)\le\varepsilon}\gamma(\pi;f)\le c\varepsilon^{\frac{k-s}{d}}\rightarrow 0\quad {\rm as}\quad \varepsilon\rightarrow 0,
\]
i.e., $f\in (\dot V_p^k)^0$ as required.\smallskip

Finally, let us prove that  $(\dot V_p^k)^0=\dot{\textsc v}_p^k$.\smallskip

To this end we need the embedding
\begin{equation}\label{eq-n5.5}
(\dot V_p^k)^0\subset C
\end{equation}
proved, in fact, in \cite[Lm.\,3.6]{BB-18}.

Let us recall that as in this lemma the inequality
\[
\omega_k(f;t)\le\sup_{|Q|\le (kt)^d}{\rm osc}_k(f;Q)\le 2^d \sup_{|Q|\le (kt)^d} E_k(f;Q)\le 2^d\sup_{d(\pi)\le (kt)^d}\gamma(\pi;f)
\]
implies that for $f\in (\dot V_p^k)^0$
\[
\lim_{t\rightarrow 0}\,\omega_k(f;t)=0
\]
while the Marchaud inequality, see, e.g., \cite[Ch.\,2,\,App.\,E2]{BBI-11}, implies from here that $\omega_1(f;t)\rightarrow 0$ as $t\rightarrow 0$, i.e., $f\in C$.

From \eqref{eq-n5.5} we obtain  that $(\dot V_p^k)^0\subset N\dot V_p^k$, hence, due to 
Theorem \ref{te2.6}\,(c) the operator $L:\dot V_p^k\rightarrow \dot V_{p\infty}^k$ (sending a function of $\dot V_p^k$ to its class of equivalence in $L_\infty$) embeds $(\dot V_p^k)^0$ isometrically into $\dot V_{p\infty}^k$.

Further, since $L$ isometrically maps  $C^\infty$ as a subset of $\dot V_p^k$  {\em onto} $C^\infty$ as that of $\dot V_{p\infty}^k$, it isometrically maps $\dot{\textsc{v}}_p^k$ onto
$\dot{\textsc{v}}_{p\infty}^k\, (:={\rm clos}(C^\infty,\dot V_{p\infty}^k))$. In turn, by \cite[Thm.\,2.13]{BB-18},
\begin{equation}\label{eq-n5.6}
\dot{\textsc{v}}_{p\infty}^k=(\dot V_{p\infty}^k)^0,
\end{equation}
where
 \[
(\dot V_{p\infty}^k)^0:=\left\{f\in\dot V_{p\infty}^k\,:\, \lim_{\varepsilon\rightarrow 0}\sup_{|\pi|\le\varepsilon}\left(\sum_{Q\in\pi} E_{k\infty}(f;Q)^p\right)^{\frac 1 p}=0\right\};
\]
$E_{k\infty}(f;Q):=\displaystyle\inf_{m\in\mathcal P_{k-1}^d}\|f-m\|_{L_\infty(Q)}$.

Comparing the definitions of $(\dot V_p^k)^0\, (\subset C)$ and $(\dot V_{p\infty}^k)^0$ we see that $L((\dot V_p^k)^0)\subset (\dot V_{p\infty}^k)^0$. 

Since $\dot{\textsc v}_p^k\subset (\dot V_p^k)^0$,  the above implications yield
\[
 L(\dot{\textsc v}_p^k)\subset L((\dot V_p^k)^0)\subset (\dot V_{p\infty}^k)^0=\dot{\textsc v}_{p\infty}^k=L(\dot{\textsc v}_p^k).
\]
Hence, $L((\dot V_p^k)^0)=L(\dot{\textsc v}_p^k)$. This and the injectivity of $L|_{(\dot V_p^k)^0}$ imply that
\[
(\dot V_p^k)^0=\dot{\textsc v}_p^k.
\]
 
The proof of the theorem is complete.
%

\sect{Proof of Theorem  \ref{cor2.1.2}}
The proof of the theorem is based on the following result of independent interest.
\begin{Th}\label{teor2.1.1}
Let $Q\Subset\RR^d$ be a nontrivial closed cube and $f\in\ell_\infty(Q)$. Then the following two-sided inequality with equivalence constants depending only on $k,d$ holds:
\begin{equation}\label{2.1.1}
E_k(f;Q)\approx {\rm osc}_k(f;Q).
\end{equation}
\end{Th}
\begin{proof}
We begin with the proof of the inequality
\begin{equation}\label{e1}
E_k(f;Q)\le c(k,d)\,{\rm osc}_k(f;Q),\quad Q\subset Q^d,
\end{equation}
where throughout the proof $c(k,d)$ denotes a positive constant depending on $k, d$ and changing from line to line or within a line.

Without loss of generality it suffices to prove \eqref{e1} for $Q=Q^d$; in this case, we write in \eqref{e1} $E_k(f)$ and ${\rm osc}_k\, f$ omitting $Q^d$, see Stipulation \ref{stip1.2.6}.

For $d=1$ the inequality was proved in \cite{Wh-59}. In more details, let $f\in\ell_\infty[0,1]$ and $L_k f$ be a polynomial of degree $k-1$ interpolating $f$ at points $\frac{i}{k-1}$, $i=0,\dots, k$, if $k>1$ and equal $\frac{1}{2}(f(0)+f(1))$ if $k=1$. Clearly, $L_k$ is a projection of $\ell_\infty [0,1]$ onto $\mathcal P_{k-1}^1$.

Now according to \cite{Wh-59}
\begin{equation}\label{e2}
\|f-L_kf\|_\infty\le c\,{\rm osc}_k\,f,
\end{equation}
where $c=c(k)$.

Further, let $\mathscr Q_k^i$ be a linear subspace of $\ell_\infty$ consisting of polynomials in $x_i$ of degree $k-1$ with coefficients depending on the remaining variables $x_j\ne x_i$. This is extended 
to $k=0$ by setting $\mathscr Q^d=\{0\}$. Moreover, we define a projection of $\ell_\infty$ onto $\mathscr Q_k^i$ denoted by $L_k^i$ given for $f\in\ell_\infty$ by applying the interpolation operator $L_k$ to the function $x_i\mapsto f(x)$, $0\le x_i\le 1$, with fixed variables $x_j\ne x_i$.

As a direct consequence of this definition and inequality \eqref{e2} we have
\begin{equation}\label{e3}
\|f-L_k^i f\|_\infty\le c\, {\rm osc}_k^i\, f,
\end{equation}
where $c=c(k)$ and
\begin{equation}\label{e4}
{\rm osc}_k^i\, f:=\sup_{x,h}\bigl\{|\Delta_h^k\, f(x)|\bigr\},
\end{equation}
where supremum is taken for the $x,h\in\RR^d$ satisfying the condition\smallskip

\noindent (*) $h$ {\em is parallel to the $x_i$-axis and} $x,x+kh\in Q^d$.\smallskip

Further, again directly from the definition we have the following
\begin{Lm}\label{lm1}
If $i\ne i'$, then projections $L_k^i, L_{k'}^{i'}$ commute.
\end{Lm}
Now let $\alpha\in\Z_+^d$ and
\begin{equation}\label{e5}
L_\alpha:=\prod_{i=1}^d L_{\alpha_i}^i.
\end{equation}
Since the projections here pairwise commute, $L_\alpha f$ is a polynomial in $x_i$ of degree $\alpha_i-1$ for each $1\le i\le d$. 

Hence, $L_\alpha f$ is a polynomial of vector degree $\alpha-e$, where $e:=\{1,\dots, 1\}$. The space of these polynomials is denoted by $\mathcal P_\alpha$.
\begin{Lm}\label{lm2}
It is true that
\begin{equation}\label{e6}
\|f-L_\alpha f\|_\infty\le c\,\sum_{i=1}^d {\rm osc}_{\alpha_i}^i f,
\end{equation}
where $c=c(\alpha, d)$.
\end{Lm}
\begin{proof}
Using the identity
\[
1-L_\alpha=(1-L_{\alpha_1}^1)+\sum_{i=2}^d\left(\prod_{j=1}^{i-1} L_{\alpha_j}^j\right) (1-L_{\alpha_i}^i),
\]
the estimate $\|L_k^i\|\le\|L_k\|$, $1\le i\le d$, and Lemma \ref{lm1} we conclude that
\[
\|f-L_\alpha f\|_\infty\le \left(\prod_{i=1}^{d-1}\|L_{\alpha_i}\|\right)\sum_{i=1}^d\|f- L_{\alpha_i}^i f\|_\infty\le c(\alpha,d)\,\sum_{i=1}^d {\rm osc}_{\alpha_i}^i\, f. 
\]
\end{proof}
\begin{C}\label{c3}
A function $f\in\ell_\infty$ belongs to the space $\mathcal P_\alpha$ if and only if
\begin{equation}\label{e7}
{\rm osc}_{\alpha_i}^i\, f=0\quad {\rm for\ all}\quad 1\le i\le d.
\end{equation}
\end{C}
\begin{proof}
If \eqref{e7} holds, then by \eqref{e6} $f=L_\alpha f\in\mathcal P_\alpha$.

Conversely, if $f\in\mathcal P_\alpha$, then $\Delta_h^{\alpha_i} f=0$ for every $h$ parallel to the $x_i$ axis, as $f$ is a polynomial in $x_i$ of degree $\alpha_i-1$, $1\le i\le d$.

This and definition \eqref{e4} imply \eqref{e7}.
\end{proof}

At the next step, we define a subspace of $\ell_\infty$ by setting 
\begin{equation}\label{e8}
\mathscr Q_\alpha:=\sum_{i=1}^d \mathscr Q_{\alpha_i}^i
\end{equation}
and an operator acting in $\ell_\infty$ by
\begin{equation}\label{e9}
\mathcal L_\alpha:=1-\prod_{i=1}^d (1-L_{\alpha_i}^i).
\end{equation}
Since $1-L_{\alpha_i}^i$ annihilates $\mathscr Q_{\alpha_i}^i$, $1\le i\le d$, and these operators pairwise commute, $\mathcal L_\alpha$ is a projection of $\ell_\infty$ onto $\mathscr Q_\alpha$ of norm $\le \prod_{i=1}^\alpha (1+\|L_{\alpha_i}^i\|)=:c(\alpha, d)$.

Now we estimate order of approximation by $\mathcal L_\alpha$ using {\em mixed $\alpha$-oscillation} given for $\alpha\in\Z_+^d$ and $f\in\ell_\infty$ by
\begin{equation}\label{e10}
{\rm osc}_\alpha\, f:=\sup_{x,h}\left\{\left|\left(\prod_{i=1}^d
\Delta_{h_i e^i}^{\alpha_i}f\right) (x)\right|\right\},
\end{equation}
where supremum is taken over $x,h\in\RR^d$ satisfying the condition
\begin{equation}\label{e11}
x\in Q^d,\quad x_i+\alpha_i h_i\in [0,1],\ 1\le i\le d,
\end{equation}
and $\{e^i\}_{1\le i\le d}$ is the standard orthonormal basis of  $\RR^d$.

In the sequel, we write
\begin{equation}\label{e12}
\Delta_h^\alpha:=\prod_{i=1}^d\Delta_{h_i e^i}^{\alpha_i}.
\end{equation}
Here $\Delta_{h_i e^i}^{\alpha_i}=1$ if $\alpha_i=0$; in particular,
\begin{equation}\label{e13}
{\rm osc}_\alpha\, f={\rm osc}_k^i\, f\quad {\rm if}\quad \alpha=ke^i.
\end{equation}
\begin{Lm}\label{lm4}
It is true that
\begin{equation}\label{e14}
\|f-\mathcal L_\alpha f\|_\infty\le c(\alpha, d)\,{\rm osc}_\alpha\, f.
\end{equation}
\end{Lm}

\begin{proof}[Proof {\rm (induction on $d$)}]
For $d=1$ the assertion coincides with inequality \eqref{e2}. 

Now let \eqref{e14} holds for $d-1\ge 0$. Setting  $\hat x:=(x_1,\dots, x_{d-1})$, $\hat\alpha:=(\alpha_1,\dots,\alpha_{d-1})$, etc. for $x\in \RR^d$, $\alpha\in\Z_+^d$ and $\hat Q:=[0,1]^{d-1}$ we then have for $g\in\ell_\infty(\hat Q)$
\begin{equation}\label{e15}
\|g-\mathcal L_{\hat\alpha}g\|_{\ell_\infty(\hat Q)}\le c(\hat\alpha, d-1)\,{\rm osc}_{\hat\alpha}\, g.
\end{equation}
Now, for $f\in\ell_\infty(Q^d)$, $\alpha\in\Z_+^d$, we have from \eqref{e9}
\[
f-\mathcal L_{\alpha}f=\left(\prod_{i=1}^{d-1}(1-L_{\alpha_i}^i)\right)(f-L_{\alpha_d}^d f)=:(1-\mathcal L_{\hat\alpha})\varphi_{x_d},
\]
where $\varphi_{x_d}:\hat x\mapsto (f-L_{\alpha_d}^d f)(\hat x, x_d)$, $\hat x\in\hat Q^d$ and $x_d\in [0,1]$ is fixed.

Taking here the $\ell_\infty(\hat Q)$-norm and applying \eqref{e15} we obtain that
\begin{equation}\label{e16}
\|f-\mathcal L_\alpha f\|_{\ell_\infty(\hat Q)}\le c(\hat\alpha, d-1)\, {\rm osc}_{\hat \alpha} (\varphi_{x_d})= c(\hat\alpha, d-1)\,\sup_{\hat x,\hat h}\left\{\left|(\Delta_{\hat h}^{\hat\alpha}(1-L_{\alpha_d}^d)f)(x)\right|\right\},
\end{equation}
where $\hat x,\hat h$ satisfy \eqref {e11} for $d-1$ (instead of $d$).

Further, denoting the function $x_d\mapsto\Delta_{\hat h}^{\hat\alpha} f(x)$, $0\le x_d\le 1$, with fixed $\hat x,\hat h$ by $\psi_{\hat x,\hat h}$, changing the order of  $\Delta_{\hat h}^{\hat\alpha}$ and $1-L_{\alpha_d}^d$ and taking supremum over $0\le x_d\le 1$ we obtain from \eqref{e16}
\[
\|f-\mathcal L_\alpha f\|_\infty\le c(\hat\alpha, d-1)\, \sup_{\hat x,\hat h}\left\{\|\psi_{\hat x,\hat h}-L_{\alpha_d}\psi_{\hat x,\hat h}\|_{\ell_\infty[0,1]}\right\}.
\]
Estimating the right-hand side by \eqref{e2} we finally have
\[
\|f-\mathcal L_\alpha f\|_\infty\le c(\hat\alpha, d-1)c(\alpha_d)\,\sup_{x,h}\left\{\left|\Delta_{\hat h}^{\hat\alpha}\Delta_{h_de^d}^{\alpha_d} f\right|\right\}=:c(\alpha, d)\, {\rm osc}_\alpha\, f.
\]
\end{proof}

Now let
\begin{equation}\label{e17}
\mathcal L_k:=\prod_{|\alpha|=k}\mathcal L_\alpha=\prod_{i=1}^N\mathcal L_{\alpha^i},
\end{equation}
where $N$ is the cardinality of the set $\{\alpha\in\Z_+^d\, :\, |\alpha|=k\}$ and $\{\alpha^i\}_{1\le i\le N}$ is its arbitrary enumeration.
\begin{Lm}\label{lm5}
It is true that
\begin{equation}\label{e18}
\|f-\mathcal L_k f\|_\infty\le c(k,d)\,\sum_{|\alpha|=k}{\rm osc}_\alpha\, f.
\end{equation}
\end{Lm}
\begin{proof}
Using the identity
\[
1-\mathcal L_k=\sum_{i=1}^N\left(\prod_{j=1}^{i-1}\mathcal L_{\alpha^j}\right)(1-\mathcal L_{\alpha^i}),
\]
where $\prod_{j=1}^{i-1}:=1$ for $i=1$ and applying Lemma \ref{lm4} we have
\[
\|f-\mathcal L_k f\|_\infty\le\left(\prod_{i=1}^{n-1}\|\mathcal L_{\alpha^i}\|\right)\sum_{|\alpha|=k}c(\alpha,d)\, {\rm osc}_\alpha\, f\le c(k,d)\,\sum_{|\alpha|=k}{\rm osc}_\alpha\, f.
\]
\end{proof}
\begin{Lm}\label{lm6}
$E_k(f)\le\|f-\mathcal L_k f\|_\infty$.
\end{Lm}
\begin{proof}
It suffices to prove that $\mathcal L_k f$ is a polynomial of degree $k-1$. 

To this end we note that $\mathcal L_\alpha f=f$ if
\begin{equation}\label{e19}
{\rm osc}_\alpha\, f=0\quad {\rm for\ all}\quad |\alpha|=k,
\end{equation}
see \eqref{e18}.

Since ${\rm osc}_k^i\, f={\rm osc}_{ke^i}\, f$, see \eqref{e13}, and the latter satisfies \eqref{e19} for all $1\le i\le d$, Corollary \ref{c3} implies that $f$ is a polynomial of degree $k-1$ in each variable.

Due to \eqref{e19} for every $|\alpha|=k$ the mixed difference $\Delta_h^\alpha f=0$. Dividing by $h^\alpha$ and sending $h$ to $0$ we obtain that $D^\alpha f=0$, $|\alpha|=k$. Since for every monomial $x^\beta$ with $|\beta|\ge k$ there is $|\alpha|=k$ such that $D^\alpha (x^\beta)\ne 0$, the polynomial $f$, hence, $\mathcal L_\alpha f$, contains in its decomposition only monomials $x^\beta$ with $|\beta|\le k-1$.

Hence, $\mathcal L_\alpha f\in\mathcal P_{k-1}^d$.
\end{proof}

It remains to use Corollary~E.4 of \cite[Ch.\,2,\,App.\,E]{BBI-11} which, in particular, implies that
\begin{equation}\label{e20}
{\rm osc}_\alpha\, f\le c(k,d)\,{\rm osc}_k,\quad |\alpha|=k.
\end{equation}
The corollary is formulated for continuous $f$ but its proof is based   only on the combinatorial identity relating $\Delta_h^\alpha$ with a linear combination of shifted $k$-differences, see Theorem~E.1 there. Hence, inequalities \eqref{e20} hold for $f\in\ell_\infty$ as well.

Finally, we combine Lemmas \ref{lm5}, \ref{lm6} and inequalities \eqref{e20} to obtain the required inequality
\[
E_k f\le c(k,d)\, {\rm osc}_k\, f,
\]
see \eqref{e1}.

The proof of the converse inequality is essentially simpler. In fact, since
by the definition, see \eqref{eq-n1.4}, ${\rm osc}_k\, f\le 2^k\|f\|_\infty$ and ${\rm osc}_k|_{\mathcal P_{k-1}^d}=0$, for every polynomial $m\in\mathcal P_{k-1}^d$
\[
{\rm osc}_k\, f\le{\rm osc}_k(f-m)\le 2^k\|f-m\|_\infty.
\]
Taking here infimum in $m$ we get
\[
2^{-k}{\rm osc}_k\, f\le E_k\, f.
\]
The proof of Theorem \ref{teor2.1.1} is complete.
\end{proof}

Now Theorem \ref{cor2.1.2} follows directly from Theorem \ref{teor2.1.1}
 and the definition of $|\cdot |_{V_p^k}$, see \eqref{1.2.14}.

\sect{Proof of Theorem \ref{teo2.2.5} }
\noindent (a) Let $f\in\dot V_p^k$ and $s:=\frac d p \in (0,k)$. Given $\varepsilon>0$ we should find a function $f_\varepsilon\in\Lambda^{k,s}$ such that
\begin{equation}\label{equ1}
|\{x\in Q^d\, :\, f(x)\ne f_\varepsilon (x)\}|<\varepsilon.
\end{equation}
To this end we first find a set, say, $S_f\subset Q^d$ of measure $1$ such that $f$ is locally Lipschitz at every its point. Then we find a subset of $S_f$ denoted by $S_\varepsilon$ such that its Lebesgue measure is at least $1-\varepsilon$ and the trace $f|_{S_\varepsilon}$ satisfies the conditions of the extension theorem for Lipschitz functions. Extending $f|_{S_\varepsilon}$ to a function of $\Lambda^{k,s}$ we finally obtain the required $f_\varepsilon$ of \eqref{equ1}.

We begin with a result on the structure of $E_k(f)$ as a function of $Q\subset Q^d$. Parameterizing the set of cubes in $\RR^d$ by the bijection $\RR^d\times\RR_+\ni (x,r)\leftrightarrow Q_r(x)\subset\RR^d$, where $Q_r(x)$ is a closed cube of sidelength $2r>0$ centered at $x$, we consider $E_k(f)$ as a function of $x,r$. It is easily seen that the subset $\Omega_0\subset \RR^d\times\RR_+$ of cubes $Q_r(x)\subset Q^d$ is a convex body.
\begin{Prop}\label{prop1}
Let $f\in \dot V_p^k$, $s\in (0,k]$. There is a function $\mathcal E_k(f):\Omega_0\rightarrow\RR_+$  such that the following is true.

\noindent {\rm (a)} $\mathcal E_k(f)$ is Lebesgue measurable in $x\in Q^d$.

\noindent {\rm (b)} For every $Q\subset Q^d$
\begin{equation}\label{equ2}
E_k(f;Q)\le\mathcal E_k(f;Q).
\end{equation}

\noindent {\rm (c)} For every packing $\pi\in\Pi$
\begin{equation}\label{equ3}
\left\{\sum_{Q\in\pi}\mathcal E_k(f;Q)^p\right\}^{\frac 1 p}\le 5 |f|_{V_p^k}.
\end{equation}
\end{Prop}
\begin{proof}
Let $\{f_n\}_{n\in\N}\subset C^\infty$ be a sequence pointwise converging on $Q^d$ to $f$ such that
\begin{equation}\label{equ4}
\varlimsup_{n\rightarrow\infty}|f_n|_{V_p^k}\le 5|f|_{V_p^k},\qquad \varlimsup_{n\rightarrow\infty}\|f_n\|_\infty\le 3\|f\|_\infty,
\end{equation}
see Theorem \ref{te2.1.4}.

Then we define the required function by setting for $Q\subset Q^d$
\begin{equation}\label{equ5}
\mathcal E_k(f;Q):=\varliminf_{n\rightarrow\infty} E_k(f_n;Q)
\end{equation}
and prove that $\mathcal E_k(f)$ satisfies the declared properties.

\noindent (a) First we show that $E_k(g)$ is continuous in $x\in Q^d$ if $g$ is.

Let $Q\subset Q^d$ and $h\in\RR^d$ be such that $Q+h\subset Q^d$. Then we have
\[
\begin{array}{c}
\displaystyle
|E_k(g;Q)-E_k(g;Q+h)|=|E_k(g;Q)-E_k(g(\cdot - h);Q)|\le E_k(g-g(\cdot - h);Q)
\\
\\
\displaystyle \qquad\qquad \qquad\quad \le \max_{Q}|g(x)-g(x-h)|\rightarrow 0\quad {\rm as}\quad h\rightarrow 0.
\end{array}
\]
In other words,
\[
\lim_{h\rightarrow 0} E_k(g;x+h,r)=E_k(g;x,r)
\]
at every $x$ from the set
\begin{equation}\label{equ6}
\mathcal Q_r:=\{x\in Q^d\, :\, Q_r(x)\subset Q^d\},
\end{equation}
i.e., $E_k(g)$ is continuous on the compact convex set $\mathcal Q_r$.

Now since each $f_n\in C^\infty$ approximating $f$ is continuous on $Q^d$, the function $E_k(f_n;\cdot ; r)$ is continuous in $x$, hence, $\mathcal E_k(f;\cdot ; r)$ is Lebesgue measurable in $x\in \mathcal Q_r$, see \eqref{equ5}. Therefore it is Lebesgue measurable in $x$ on the set
$\cup_{r>0}\,\mathcal Q_r=\mathring{Q}^d$.\smallskip

\noindent (b) To prove that $E_k(f;Q)\le\mathcal E_k(f;Q)$, $Q\subset Q^d$, we for every $n\in\N$ choose a polynomial $P_n\in\mathcal P_{k-1}^d$ such that
\[
E_k(f_n;Q)=\|f_n-P_n\|_{\ell_\infty(Q)}.
\]
We have 
\[
\|P_n\|_{\ell_\infty(Q)}\le\|f_n\|_\infty+ E_k(f_n,Q)\le 2\|f_n\|_\infty
\]
and then by \eqref{equ4}
\[
\varlimsup_{n\rightarrow\infty}\|P_n\|_{\ell_\infty(Q)}\le 2\varlimsup_{n\rightarrow\infty}\|f_n\|_\infty\le 6\|f\|_\infty.
\]
Hence, $\{P_n\}_{n\in\N}$ is bounded in $\ell_\infty(Q)$ and so contains a subsequence $\{P_{n_i}\}_{i \in\N}$ uniformly converging on $Q$ to some $P\in\mathcal P_{k-1}^d$ such that 
\[
\lim_{i\rightarrow\infty}\|f-P_{n_i}\|_{\ell_\infty(Q)}=\varliminf_{n\rightarrow\infty}\|f_n-P_n\|_{\ell_\infty(Q)}.
\]
This then implies that 
\[
E_k(f;Q)\le \|f-P\|_{\ell_\infty(Q)}\le\lim_{i\rightarrow\infty}\|f_{n_i}-P_{n_i}\|_{\ell_\infty(Q)}=
\varliminf_{n\rightarrow\infty}\|f_n-P_n\|_{\ell_\infty(Q)}=:\mathcal E_k(f;Q).
\]

This proves \eqref{equ2} and property (b).\smallskip

\noindent (c) To prove \eqref{equ3} we write
\[
\sum_{Q\in\pi}\mathcal E_k(f;Q)^p:=\sum_{Q\in\pi}\varliminf_{n\rightarrow\infty}E_k(f_n;Q)^p\le \varliminf_{n\rightarrow\infty}\sum_{Q\in\pi} E_k(f_n;Q)^p\le
\varliminf_{n\rightarrow\infty}|f_n|^p_{V_p^k}\le (5|f|_{V_p^k})^p.
\]
This proves \eqref{equ3} and the proposition.
\end{proof}
\begin{Prop}\label{prop2}
Under the assumptions of Proposition \ref{prop1}
\begin{equation}\label{equ7}
\varlimsup_{Q\rightarrow x}\bigl(|Q|^{-\frac 1 p} E_k(f;Q)\bigr)<\infty
\end{equation}
for almost all $x\in Q^d$.
\end{Prop}
\begin{proof}
Let $\rho_a: \mathcal Q_a\rightarrow\RR_+$, $a>0$, be a function given by
\begin{equation}\label{equ8}
\rho_a(x):=\sup\bigl\{r^{-\frac d p}\mathcal E_k(f;x,r)\, :\, r\le a\bigr\};
\end{equation}
here $\mathcal Q_a$ is the subset of $Q^d$ given by \eqref{equ6}.

Given $N\in\N$ we then set
\begin{equation}\label{equ9}
S_{a,N}:=\{x\in \mathcal Q_a\, :\, \rho_a(x)> N\}.
\end{equation}
\begin{Lm}\label{le3}
Let $|f|_{V_p^k}=1$. Then 
\begin{equation}\label{equ10}
|S_{a,N}|\le c(d,p)N^{-p}.
\end{equation}
\end{Lm}
\begin{proof}
Let $S\subset Q^d$ and 
\begin{equation}\label{equ11}
M(S):=\inf_{\Delta}\left\{\sum_{Q\in\Delta}|Q|\right\},
\end{equation}
where $\Delta$ runs over all coverings of $S$ by subcubes of $Q^d$.

It is known that for a measurable set $S$
\begin{equation}\label{equ12}
|S|\le M(S),
\end{equation}
see, e.g., \cite[Thm.\,I.1]{Ca-67}.

Now let $x\in S_{a,N}$. By definitions \eqref{equ8},\eqref{equ9} there is a cube
$Q_x$ of radius $\le a$ centered at $x\in \mathcal Q_a$ such that
\begin{equation}\label{equ13}
\mathcal E_k(f;Q_x)> N |Q_x|^{\frac 1 p}.
\end{equation}
Since the family $\{Q_x\}$ covers $S_{a,N}$ by centers of its cubes, the Besicovich covering theorem, see, e.g., \cite[Thm.\,1.2]{Gu-75} asserts that there is a countable subcovering of $\{Q_x\}$, say $\{Q_i\}_{i\in\N}$, which is the union of at most $c(d)$ packings $\pi_i\in\Pi$. This and \eqref{equ11}-\eqref{equ13} imply that
\[
|S_{a,N}|\le\sum_{i\le c(d)}\sum_{Q\in\pi_i}|Q|\le N^{-p}\sum_{i\le c(d)}\sum_{Q\in\pi_i}\mathcal E_k(f;Q)^p.
\]
By \eqref{equ3} the inner sum in the right-hand side is bounded from above by $(5|f|_{V_p^k})^p=5^p$, hence,
\[
|S_{a,N}|\le 5^p c(d) N^{-p}.
\]
\end{proof}

Now we have from \eqref{equ10}
\[
|S_{a,\infty}|=\left|\bigcap_{N\in\N} S_{a,N}\right|=0;
\]
moreover, by monotonicity of sets $S_{a,\infty}$ under inclusions, $S_\infty:=\cup_{a>0}S_{a,\infty}$ is of (Lebesgue) measure zero.

Finally, for every $x\in\mathcal Q_a\setminus S_{a,\infty}$
\begin{equation}\label{equ14}
\sup_{r\le a} r^{-\frac d p}\mathcal E_k(f;x,r)<\infty\quad {\rm and}\quad \bigcup_{0<a\le\frac 1 2}\mathcal Q_a=\mathring Q^d,
\end{equation}
see \eqref{equ6}.

Hence, for every point $x\in Q^d$ outside set $S_\infty\cup\partial Q^d$ of measure $0$
\[
\varlimsup_{Q\rightarrow x}|Q|^{-\frac 1 p}\mathcal E_k(f;Q)<\infty.
\]
This and  \eqref{equ2} give \eqref{equ7}.

Proposition \ref{prop2} is proved.
\end{proof}

In the next auxiliary result we use the following:
\begin{D}\label{def4}
A Lebesgue measurable set $S\subset\RR^d$ is said to be Ahlfors $d$-regular (briefly, regular) if for every $x\in S$ and some constants $\delta\in (0,1)$, $r_0>0$,
\begin{equation}
|Q_r(x)\cap S|\ge\delta |Q_r(x)|,\quad 0<r\le r_0.
\end{equation}
\end{D} 
\begin{Lm}\label{lem5}
Let $f\in V_p^k$, $s:=\frac d p\in (0,k]$. Given $\varepsilon\in (0,1)$ there is a regular set, say, $\Sigma_\varepsilon\subset Q^d$ such that
\begin{equation}\label{equ16}
|\Sigma_\varepsilon|\ge 1-\varepsilon
\end{equation}
and, moreover, there are positive constants $c_\varepsilon,a_\varepsilon$ such that  \begin{equation}\label{equ17}
\sup_{x\in\Sigma_\varepsilon} E_k(f;x,r)\le c_\varepsilon r^s,\quad 0\le r\le a_\varepsilon .
\end{equation}
\end{Lm}
\begin{proof}
We set
\begin{equation}\label{equ18a}
a_\varepsilon:=\frac 1 2 \left(1-\sqrt[d]{1-\frac \varepsilon 3 } \right).
\end{equation}
By \eqref{equ14} we have for every $x\in\mathcal Q_{a_\varepsilon}\setminus S_{a_\varepsilon,\infty}$ and $r_n:=2^{-n}a_\varepsilon$
\begin{equation}\label{equ18}
\varphi(x):=\sup_{n\in\Z_+}\{r_n^{-s}\mathcal E_k(f;x,r_n)\}<\infty.
\end{equation}
Since the sequence under supremum is measurable in $x$, see Proposition \ref{prop1}\,(a), $\varphi$ is measurable and finite at every point of $\mathcal Q_{a_\varepsilon}\setminus S_{a_\varepsilon,\infty}$. Hence, there are a set  $\Omega_{\varepsilon}\subset\mathcal Q_{a_\varepsilon}\setminus S_{a,\infty}$ and a constant $N_\varepsilon$ such that
\begin{equation}\label{equ19}
\varphi\le N_\varepsilon\quad {\rm on}\quad \Omega_\varepsilon\quad {\rm and}\quad |\Omega_\varepsilon|\ge |\mathcal Q_{a_\varepsilon}|-\frac \varepsilon 3 .
\end{equation}
Further, the Lebesgue density theorem implies that
\begin{equation}\label{equ20}
\lim_{r\rightarrow 0}\frac{|\Omega_\varepsilon\cap Q_r(x)|}{|Q_r(x)|}=1\quad {\rm a.e.}\quad {\rm in}\quad \Omega_\varepsilon.
\end{equation}
By Egorov's theorem there is a measurable subset $\Sigma_\varepsilon\subset\Omega_\varepsilon$ of measure
\[
|\Sigma_\varepsilon|>|\Omega_\varepsilon|-\frac \varepsilon 3
\] 
such that \eqref{equ20} converges {\em uniformly} on $\Sigma_\varepsilon$. Assuming without loss of generality that all $x\in\Sigma_\varepsilon$ are simultaneously density points of $\Omega_\varepsilon$ and $\Sigma_\varepsilon$ we also have
\[
1=\lim_{r\rightarrow 0}\frac{|\Omega_\varepsilon\cap Q_r(x)|}{(2r)^d}=
\lim_{r\rightarrow 0}\frac{|\Sigma_\varepsilon\cap Q_r(x)|}{(2r)^d}
\]
for every $x\in\Sigma_\varepsilon$.

This implies that the second limit also converges on $\Sigma_\varepsilon$ uniformly.

Hence, for every $x\in\Sigma_\varepsilon$ and some $0<r_0<1$ (depending on $\Omega_\varepsilon$)
\[
\frac{|\Sigma_\varepsilon\cap Q_r(x)|}{(2r)^d}>\frac 1 2\quad {\rm for}\quad 0<r\le r_0,
\]
i.e., $\Sigma_\varepsilon$ is regular.

Let us show that $\Sigma_\varepsilon$ satisfy inequality \eqref{equ16}. In fact, due to \eqref{equ6} and \eqref{equ18a}
\[
|\Sigma_\varepsilon|>|\Omega_\varepsilon|-\frac {\varepsilon}{3} >|\mathcal Q_{a_\varepsilon}|-\frac{2\varepsilon}{3}=(1-2a_\varepsilon)^d-\frac{2\varepsilon}{3}=1-\varepsilon,
\]
as required.

Finally, $\varphi\le N_\varepsilon$ on the regular set $\Sigma_\varepsilon$, hence, for $x\in\Sigma_\varepsilon$
\[
\mathcal E_k(f;x,2^{-n}a_\varepsilon)\le N_\varepsilon (2^{-n}a_\varepsilon)^{s},\quad n\in\Z_+.
\]
Enlarging $N_\varepsilon$ (and denoting its new value by $c_\varepsilon$)  we can replace here $2^{-n}a_\varepsilon$ by an arbitrary $r\in (0,a_\varepsilon]$. Moreover, by Proposition \ref{prop1}\,(b), $E_k(f)\le\mathcal E_k(f)$, hence, for $x\in\Sigma_\varepsilon$
\[
E_k(f;x,r)\le c_\varepsilon r^s,\quad 0<r\le a_\varepsilon.
\]
This gives \eqref{equ17} and proves the lemma.
\end{proof}

Proceeding the proof of Theorem we apply inequality \eqref{equ17} to the trace
$f|_{\Sigma_\varepsilon}$. This function satisfies inequality
\begin{equation}\label{equ21}
E_k(f|_{\Sigma_\varepsilon};Q_r(x))\le c_\varepsilon r^s,\quad 0<r\le a_\varepsilon,
\end{equation}
on the regular set $\Sigma_\varepsilon$ and therefore meets the conditions of the extension theorem \cite{Br1-70}, see \cite[Thm.\,9.30]{BBII-11} for the general version of this result. Due to this theorem there is a function $f_\varepsilon:\RR^d\rightarrow\RR$ such that
\[
f_\varepsilon|_{\Sigma_\varepsilon}=f|_{\Sigma_\varepsilon}
\]
and, moreover, 
\[
\sup_{x\in Q^d} E_k(f_\varepsilon;x,r)\le c r^s,\quad 0\le r\le 1;
\]
here $c$ depends on $\varepsilon, d, k$ and $f$.

Using further Theorem \ref{teor2.1.1} we replace $E_k(f_\varepsilon;x,r)$ by ${\rm osc}_k(f;Q_r(x))$ and finally obtain
\[
|\Delta_h^k f_\varepsilon (x)|\le {\rm osc}_k(f;x, k\|h\|)\le c\|h\|^s,
\]
where $\|h\|:=\max_{1\le i\le d}|h_i|$ and $x, x+h$ belong to $Q^d$.

By definition, see \eqref{e2.2.4}, \eqref{e2.2.5}, this means that $f_\varepsilon\in\Lambda^{k,s}$, $0<s\le k$.

Thus, we have proved that for $f\in\dot V_p^k$ with $s\in (0,k]$ and $\varepsilon>0$ there is a Lipschitz function $f_\varepsilon\in\Lambda^{k,s}$ such that
\begin{equation}\label{equa22}
f=f_\varepsilon\quad {\rm outside\ a\ set\ of\ Lebesgue\ measure}\ <\varepsilon.
\end{equation}
For $0<s<k$, this establishes assertion (a) of Theorem \ref{teo2.2.5}.

Finally, in part (b) of this theorem we should prove that if $s=k$ we can replace the function $f_\varepsilon\in\Lambda^{k,k}$ by a $C^k$ function. To this end we use the linear continuous isomorphism
\begin{equation}\label{equa23}
\Lambda^{k,k}=C^{k-1,1}\cong \dot W_\infty^k.
\end{equation}
see Proposition \ref{prop2.2.3} and \cite{Br1-70}, respectively.

More precisely, every class $f\in\dot W_\infty^1$ contains a unique $C^{k-1,1}$ function, say $\hat f$, such that the map $f\mapsto\hat f$ is a linear continuous bijection of $\dot W_\infty^k$ onto $C^{k-1,1}$.

Further, $\dot W_\infty^k\subset \dot W_p^k$, $p>1$, while for every function $\hat f$ from the class $f\in\dot W_p^k$ there is a function $f_\varepsilon\in C^k$ such that $f$ coincides with $f_\varepsilon$ outside a set of measure $<\varepsilon$, see \cite[Thm.\,4.13]{CZ-61}.
Combining this with \eqref{equa23} we conclude that for every $g\in C^{k-1,1}$ there is $g_\varepsilon\in C^k$ coinciding with $g$ outside of a set of measure $<\varepsilon$.

Taking now $f\in\dot V_p^k$ with $s=k$ and $g=f_\varepsilon$ from \eqref{equa22} we obtain the function $g_\varepsilon\in C^k$ coinciding with $f$ outside a set of measure $<2\varepsilon$.

This proves assertion (b) and the theorem.
\sect{Proof of Theorem \ref{te2.2.8}}
(a) We should prove that if $f\in\dot V_p^k$ and $s:=\frac d p \in (0,k]$, then
\begin{equation}\label{one}
f\in\Lambda^{k,s}(x)\quad {\rm a.\, e.}
\end{equation}
To this end we use Proposition \ref{prop2} asserting that for this $f$
\begin{equation}\label{two}
\varlimsup_{r\rightarrow 0} r^{-s}E_k(f;x,r)<\infty\quad {\rm for\ all}\quad x\in S_f,
\end{equation}
where $S_f$ is a subset of $\mathring Q^d$ of measure $1$.

Due to inequality \eqref{equat2.1.1} $E_k(f)$ here can be replaced by ${\rm osc}_k(f)$; after this change condition \eqref{two} at $x\in S_f$ coincides with the condition of Definition \ref{def2.2.6} introducing the space $\Lambda^{k,s}(x_0)$. Hence, we obtain that $f\in\Lambda^{k,s}(x)$ for all $x\in S_f\subset\mathring Q^d$, i.e., almost everywhere on $Q^d$.
\smallskip

\noindent (b) We should prove that the function under consideration belongs to the Taylor space $T^s(x)$ for almost all $x\in Q^d$.

To this end we use Theorem 3 from \cite[\S 2]{Br1-94}
that, in particular, asserts the following:\smallskip

{\em If a function $f\in\ell_\infty(\mathring Q^{d})$ satisfies at a given point $x_0$ the condition}
\[
\varlimsup_{r\rightarrow 0} r^{-s} E_k(f;x,r)<\infty,
\]
{\em where} $0<s\le k$, {\em then} $f\in T^s(x_0)$.\smallskip

It then follows from here and condition \eqref{two} that $f\in T^s(x)$ for every $x\in S_f$, i.e., for almost all $x\in Q^d$.\smallskip

\noindent (c) Now we should show that if $f\in\dot V_p^k$ and $s=k$, then $f\in t^k(x_0)$ a.e.

To this end we first note that due to the previous result $f\in T^k(x)$ for every $x\in S_f$, i.e., for every such $x$ there is the Taylor polynomial $T_x(f)\in\mathcal P_{k-1}^d$ and constants $c_1,c_2>0$ independent of $r$ such that
\begin{equation}\label{three}
\max_{Q_r(x)}|f-T_x(f)|\le c_1 r^k\quad {\rm for}\quad
0<r\le c_2. 
\end{equation}

Further, we use Theorem \ref{teo2.2.5} asserting that given $\varepsilon>0$ there is a regular subset $S_\varepsilon\subset\mathring Q^d$ and a function $f_\varepsilon\in C^k$ such that
\begin{equation}\label{four}
|S_\varepsilon|>1-\varepsilon\quad {\rm and}\quad f=f_\varepsilon\quad {\rm on}\quad S_\varepsilon .
\end{equation}
Now we derive from here  and \eqref{three} that
\begin{equation}\label{five}
f\in t^k(x)\quad {\rm for\ almost\ all}\quad x\in S_\varepsilon .
\end{equation}
Since $|\cup_{\varepsilon>0} S_\varepsilon |=1$, this implies that $f\in t^k(x)$ for almost all $x$ in $Q^d$.

Let $T_x(f_\varepsilon)\in\mathcal P_{k}^d$ be the Taylor polynomial of $f_\varepsilon$ at $x$, i.e.,
\begin{equation}\label{six}
\max_{Q_r(x)}|f_\varepsilon -T_x(f_\varepsilon)|\le c(r) r^k,\quad 0<r\le\frac 1 2,
\end{equation}
where $c(r)\rightarrow 0$  as $r\rightarrow 0$.

We write
\[
T_x(f_\varepsilon):=\hat T_x(f_\varepsilon)+\sum_{|\alpha|=k}c_\alpha(f_\varepsilon )(\cdot - x_0)^\alpha,
\]
where $\hat T_x(f_\varepsilon)$ is the Taylor polynomial for $f_{\varepsilon}$ at $x$ of degree $k-1$. In particular,
\begin{equation}\label{eq-n8.7}
\max_{Q_r(x)}|f_\varepsilon-\hat T_x(f_\varepsilon)|\le c_1 r^k\quad {\rm for}\quad 0< r\le\frac 1 2 
\end{equation}
with a constant independent of $r$.
\begin{Lm}\label{lemm1}
If $x\in S_\varepsilon\cap S_f$, then 
\begin{equation}\label{seven}
T_x(f)=\hat T_x(f_\varepsilon).
\end{equation}
\end{Lm}
\begin{proof}
By regularity of $S_\varepsilon$ there are positive constants $\gamma, r_0$ such that for every $x\in S_\varepsilon$
\begin{equation}\label{eight}
|S_\varepsilon \cap Q_r(x)|\ge \gamma |Q_r(x)|\quad {\rm for}\quad 0<r\le r_0.
\end{equation}
Moreover, $f=f_\varepsilon$ on $S_\varepsilon\cap Q_r(x)$; hence,
\[
\max_{S_\varepsilon \cap Q_r(x)}|T_x(f)-\hat T_x(f_\varepsilon)|\le\max_{Q_r(x)}|f-T_x(f)|+\max_{Q_r(x)}|f_\varepsilon- \hat T_x(f_\varepsilon)|.
\]
Estimating the right-hand side by inequalities \eqref{three} and \eqref{eq-n8.7} we have for every $x\in S_\varepsilon\cap S_f$
\begin{equation}\label{nine}
\max_{S_\varepsilon\cap Q_r(x)}|T_x(f)-\hat T_x(f_\varepsilon)|\le c_3 r^k,\quad 0\le r\le c_2
\end{equation}
for some $c_3>0$ independent of $r$.

Further, using a Remez type inequality \cite{BG-73} and \eqref{eight} we can estimate the left-hand side of \eqref{nine} from below by $c(k,d)\gamma^{k-1}\max_{Q_r(x)}|T_x(f)-\hat T_x(f_\varepsilon)|$.

This implies for sufficiently small $r>0$ the inequality
\[
\max_{Q_r(x)}|T_x(f)-\hat T_x(f_\varepsilon)|\le c r^k
\]
with $c$ independent of $r$.

Dividing this by $r^k$ and letting $r$ to $0$ we conclude that every coefficient of this polynomial (of degree $k-1$) should be $0$.
\end{proof}

Finally, we show that $f\in t^k(x)$ for almost all $x\in S_\varepsilon $. 

In fact, for the function $F:=f-f_\varepsilon$ and $x\in S_\varepsilon\cap S_f$ we have by Lemma \ref{lemm1} and inequalities  \eqref{three}, \eqref{eq-n8.7}
\begin{equation}\label{eq-n8.11}
\max_{Q_r(x)}|F|\le \max_{Q_r(x)}|f_\varepsilon-\hat T_x(f_\varepsilon)|+\max_{Q_r(x)}|f- T_x(f)|\le c_1 r^k,\quad  0<r\le c_2,
\end{equation}
where $c_1,c_2>0$ are independent of $r$.

Inequality \eqref{eq-n8.11} and Theorem 10 from \cite{CZ-61}  
imply that for almost all $x\in S_\varepsilon\cap S_f$
\[
\max_{Q_r(x)}|F|=o(r^k),\quad r\rightarrow 0.
\]
In turn, for such $x$ we derive from here and \eqref{six}
\[
\max_{Q_r(x)}|f-T_x(f_\varepsilon)|\le \max_{Q_r(x)}|F|+\max_{Q_r(x)}|f_\varepsilon-T_x(f_\varepsilon)|=o(r^k),\quad r\rightarrow 0.
\]
Since $|S_\varepsilon\cap S_f|=|S_\varepsilon |$, $f\in t^k(x)$ for almost all $x\in S_\varepsilon$.

This proves assertion (c) of Theorem \ref{te2.2.8}.
\sect{Proofs of Theorem \ref{teo2.2.11} and Corollary \ref{cor9.7}}
\begin{proof}[Proof of Theorem \ref{teo2.2.11}]
(a) We should show that the map $L: f\mapsto\tilde f$ sending a function $f\in N\dot V_p^k$ to its equivalence class $\tilde f\in L_p$ is a linear continuous injection  of $N\dot V_p^k$ in $\Lambda_{p\infty}^{ks}$, $0<s\le k$. 

In fact, the injectivity of the map follows from the definition of the class $N\dot V_p^k$, see Corollary \ref{cor2.1.3}, and 
so it remains to show that its image belongs to $\Lambda_{p\infty}^{ks}$ and the map  acts to this space continuously.

To this end, we again apply Theorem 4 from \cite[\S2]{Br1-71} to write
\[
t^{-s}\omega_{kp}(\tilde f;t)\le c\, t^{-s}\left(\sum_{Q\in\pi} E_{kp}(\tilde f;Q)^p\right)^{\frac 1 p},
\]
where $\pi$ is some packing in $\Pi(Q^d)$ containing cubes of volume $\le t^d$.

For these cubes, we clearly have
\[
E_{kp}(\tilde f;Q)\le |Q|^{\frac 1 p} E_{k\infty}(\tilde f;Q)\le c\, t^{\frac d p} E_k(f;Q):=c\, t^s E_k(f;Q).
\]
Together with the previous inequality this implies
\[
t^{-s}\omega_{kp}(\tilde f;t)\le c \left(\sum_{Q\in\pi} E_k(f;Q)^p\right)^{\frac 1 p}.
\]
Taking here supremums over $t>0$ and all $\pi\in\Pi$, we finally have
\[
|\tilde f|_{\Lambda_{p\infty}^{ks}}:=\sup_{t>0} t^{-s}\omega_{kp}(\tilde f;t)\le
c\sup_{\pi}\left(\sum_{Q\in\pi} E_k(f;Q)^p\right)^{\frac 1 p}=:c\, |f|_{V_p^k}.
\]
Hence, the map under consideration acts continuously from $N\dot V_p^k$ in $\Lambda_{p\infty}^{ks}$, i.e., gives the required implication $N\dot V_p^k\hookrightarrow \Lambda_{p\infty}^{ks}$, see \eqref{e2.2.19}. \smallskip

Further, let the space $\dot V_p^k$ be  such that
\begin{equation}\label{10.1}
s:=\frac d p \in (0,k).
\end{equation}
We will show that every class $f\in\Lambda_{p1}^{ks}$ contains a unique representative $\hat f \in C$ and the map $f\mapsto\hat f$ is a linear continuous injection of $\Lambda_{p1}^{ks}$ into $\dot V_p^k$. This clearly proves the required implication
\begin{equation}\label{10.2}
\Lambda_{p1}^{ks}\hookrightarrow \dot V_p^k.
\end{equation}

This will complete the proof of part (a) of Theorem \ref{teo2.2.11}.

To prove \eqref{10.2} we given $f\in L_p$ denote by $\mathcal L_f$ the Lebesgue set of $f$, i.e., $Q^d\setminus\mathcal L_f$ is of measure zero and the limit over cubes $Q$ centered at $x$
\begin{equation}\label{10.4}
\hat f(x):=\lim_{Q\rightarrow x}\frac{1}{|Q|}\int_Q f\,dy
\end{equation}
exists at every $x\in\mathcal L_f$.
\begin{Lm}\label{le10.1}
If $f\in L_p$ is such that
\begin{equation}\label{10.5}
\int_0^1\frac{\omega_{kp}(f;t)}{t^{s+1}}\, dt<\infty
\end{equation}
for $0<s<k$, then $\hat f$ is uniformly continuous on $\mathcal L_f$ and for every subcube $Q\subset Q^d$
\begin{equation}\label{10.6}
E_{k\infty}(f;Q)\le c(k,d)\int_0^{|Q|^{\frac 1 d }}\frac{\omega_{kp}(f;t;Q)}{t^{s+1}}\, dt.
\end{equation}
\end{Lm}
Hereafter, we as above set
\begin{equation}\label{10.7a}
E_{kp}(f;Q):=\inf\{\|f-m\|_{L_p(Q)}\, :\, m\in\mathcal P_{k-1}^d\}
\end{equation}
and denote by $m_Q(f)$ an optimal polynomial for this relation.
\begin{proof}
Setting for brevity
\[
\omega(t):=\omega_{kp}(f;t;Q),\quad t>0,
\]
we estimate the nonincreasing rearrangement $(f-m_Q(f))^*$ by
\begin{equation}\label{10.7}
(f-m_Q(f))^*(t)\le c(k,d)\int_{\frac t 2}^{|Q|}\frac{\omega(u^{\frac 1 d})}{u^{1+\frac 1 p}}\, du,\quad 0\le t\le |Q|,
\end{equation}
see \cite[App.\,{\rm II},\,Cor.\,2$'$]{Br1-94}.

Sending here $t$ to $0$ and noting that $s=\frac d p$ we then obtain
\[
E_{k\infty}(f;Q)\le\|f-m_Q(f)\|_{L_\infty(Q)}=\lim_{t\rightarrow\infty}(f-m_Q(f))^*(t)\le c(k,d)\int_{0}^{|Q|^{\frac 1 d}}\frac{\omega(u)}{u^{s+1}}\, du.
\]

This proves \eqref{10.6}.

Now we show that the function $\hat f:\mathcal L_f\rightarrow\RR$ is uniformly continuous.

In fact, by \eqref{10.6}
\begin{equation}\label{10.8}
\lim_{Q\rightarrow x} E_{k\infty}(f;Q)=0
\end{equation}
and the convergence is uniform in $x$.  As in the proof of 
Theorem \ref{te2.1.5} we derive from here that
\begin{equation}\label{10.9}
\lim_{t\rightarrow 0}\omega_{1\infty}(f;t)=0.
\end{equation}

Finally, let $x,x+h\in\mathcal L_f\, (\subset \mathring Q^d)$. Then
\[
|\hat f(x+h)-\hat f(x)|\le\varlimsup_{Q\rightarrow x}\frac{1}{|Q|}\int_{Q}|f(y+h)-f(y)|\, dy\le\omega_{1\infty}(f;\|h\|)\rightarrow 0\quad {\rm as}\quad h\rightarrow 0.
\]
Hence, $\hat f$ is uniformly continuous on the dense in $Q^d$ set $\mathcal L_f$ and so can be continuously extended to $Q^d$; we preserve the same notation $\hat f$ for the extension.

Lemma \ref{le10.1} is proved.
\end{proof}

Now let us check that the regularization $\hat f$, see \eqref{10.4}, satisfies
\begin{equation}\label{10.12}
E_{k\infty}(f;Q)=E_k(\hat f;Q)\quad {\rm for\ all}\quad Q\subset Q^d.
\end{equation}

Indeed, let $m_Q\in\mathcal P_{k-1}^d$ be such that
\[
E_{k\infty}(f;Q)=\|f-m_Q\|_{L_\infty(Q)}.
\]
Given $\varepsilon\in (0,1)$ we choose $\hat f_\varepsilon\in\ell_\infty$ of the class $f\in L_\infty$ such that
\[
\sup_{Q}|\hat f_\varepsilon -m_Q|\le\|f-m_Q\|_{L_\infty(Q)}+\varepsilon.
\]
Since $\hat f$ and $\hat f_\varepsilon$ coincide on a set of complete Lebesgue measure, say $S_\varepsilon$, we have
\[
\hat f-m_Q=\hat f_\varepsilon-m_Q\quad {\rm on}\quad S_\varepsilon\cap\mathcal L_f.
\]
This, in turn, implies
\[
E_k(\hat f;Q)\le \sup_Q |\hat f-m_Q|-\sup_{S_\varepsilon\cap\mathcal L_f}|\hat f-m_Q|\le\sup_Q |\hat f_\varepsilon-m_Q|\le\|f-m_Q\|_{L_\infty(Q)}+\varepsilon=E_{k\infty}(f;Q)+\varepsilon.
\]
Hence,
\[
E_k(\hat f;Q)\le E_{k\infty}(f;Q).
\]
Since the converse inequality is evident, \eqref{10.12} is proved.

Now let $\pi\in\Pi$ be a packing. Due to Lemma \ref{le10.1}, equality \eqref{10.12} and Minkowski inequality we have
\begin{equation}\label{10.13}
\begin{array}{l}
\displaystyle
\left(\sum_{Q\in\pi} E_k(\hat f;Q)^p\right)^{\frac 1 p}\le c\left(\sum_{Q\in\pi}\left(\int_0^{|Q|^{\frac 1 d}}\frac{\omega_{kp}(f;t;Q)}{t^{s+1}}\, dt\right)^p\right)^{\frac 1 p}\\
\\
\displaystyle \hfill\le c\int_0^1 t^{-s-1}\left(\sum_{Q\in\pi}\omega_{kp}( f;t;Q)^p\right)^{\frac 1 p}\, dt.\quad
\end{array}
\end{equation}
Hereafter all constants depend only on $k,d$.
\begin{Lm}\label{le10.2}
It is true that
\begin{equation}\label{10.14}
\left(\sum_{Q\in\pi}\omega_{kp}( f;t;Q)^p\right)^{\frac 1 p}\le c\,\omega_{kp}( f;t).
\end{equation}
\end{Lm}
\begin{proof}
By Theorem 4 from \cite[\S2]{Br1-71}
\begin{equation}\label{10.15}
\omega_{kp}( f;t;Q)\approx \sup_{\pi\in\Pi_t(Q)}\left(\sum_{R\in\pi}E_{kp}(f;R)^p\right)^{\frac 1 p},
\end{equation}
where $\Pi_t(Q)$ consists of all packings in $Q$ with cubes of volume $\le t^d$.

Hence, for every $Q$ there is a packing, say $\pi_Q\in\Pi_t(Q)$, such that 
\[
\omega_{kp}( f;t;Q)\le c \left(\sum_{R\in\pi_Q} E_{kp}( f;R)^p\right)^{\frac 1 p}.
\]
Therefore the left-hand side of \eqref{10.14} is bounded by
$
c\bigl(\sum_{Q\in\pi}\sum_{R\in\pi_Q}E_{kp}( f;R)^p\bigr)^{\frac 1 p}.
$

\noindent Since the set of subcubes $\cup_{Q\in\pi}\pi_Q$ is a packing in $Q^d$ containing only cubes of volume $\le t^d$, the  equivalence \eqref{10.15} bounds the above expression by $c\,\omega_{kp}(f;t)\, (=c\,\omega_{kp}(f;t;Q^d))$.

This proves \eqref{10.14}.
\end{proof}

Now combining \eqref{10.13} and \eqref{10.14} and taking supremum over $\pi\in\Pi$ we finally have
\[
|\hat f|_{V_p^k}:=\sup_{\pi}\left(\sum_{Q\in\pi} E_k(\hat f;Q)^p\right)^{\frac 1 p}\le
c\int_0^1 \frac{\omega_{kp}( f;t)}{t^{s+1}}\, dt=:c | f|_{\Lambda_{p1}^{ks}}.
\]
This proves that the linear injection $f\mapsto\hat f$ maps $\Lambda_{p1}^{ks}$ in $\dot V_p^k\cap C\subset N\dot V_p^k$ and is bounded, i.e., implication \eqref{10.2}  holds.

Theorem \ref{teo2.2.11}\,(a) is proved.\smallskip

\noindent (b) Now we should prove that if $k=s=d$, then the following linear isomorphism is given by the operator $L|_{N\dot V_p^k}$
\begin{equation}\label{10.17}
Lip_1^d\cong N\dot V_1^d.
\end{equation}

Since the injection  $L|_{N\dot V_p^k}:N\dot V_p^k\rightarrow \Lambda_{p\infty}^{ks}$,  $0<s\le k$, of part (a) of the theorem implies for $k=s=d$ the injection $L|_{N\dot V_1^d}:N\dot V_1^d\rightarrow Lip_1^d$, it remains to prove that it is onto.  In the derivation, we need the following Sobolev type embedding
that is interesting in its own right result.
\begin{Prop}\label{pro10.4}
It is true that
\begin{equation}\label{10.18}
Lip_1^d\subset \dot V_{1\infty}^d.
\end{equation}
\end{Prop}
\begin{proof}
In the proof, we use the following known results.

Let $BV^k(Q)$ be a linear space of $L_1$ functions on $Q$ whose $k$-th distributional derivatives are finite Borel measures. It is endowed by a seminorm given for $f\in L_1(Q)$ by
\begin{equation}\label{10.19}
|f|_{BV^k}:=\max_{|\alpha|=k}\bigl\{{\rm var}\,D^\alpha f\bigr\}.
\end{equation}
\begin{Th A}[\cite{Br1-70} Theorem 4]
It is true that
\begin{equation}\label{10.20}
Lip_1^k(Q)= BV^k(Q)
\end{equation}
with equivalence constants of seminorms depending only on $k$ and $d$.
\end{Th A}

The second result is the embedding theorem for Sobolev space $W_1^d:= W_1^d(Q^d)$.
\begin{Th B}[\cite{Ga-59}]
It is true that
\begin{equation}\label{10.21}
W_1^d\hookrightarrow C,
\end{equation}
where the embedding constant depends only on $d$.
\end{Th B}

As a consequence we have the following:
\begin{Lm}\label{pr10.5}
For each function $f\in BV^d(Q)$, $Q\subset Q^d$, it is true that $f\in L_\infty$ and
\[
E_{d\infty}(f;Q)\le c(d) |f|_{BV^d(Q)}.
\]
\end{Lm}
\begin{proof}
We begin with some facts of interpolation space theory.

First, we recall that the $K$-{\em functional} of an embedded pair $X\subset Y$ of Banach spaces is a function on $Y\times\RR_+$ given for $y\in Y$, $t\in\RR_+$ by
\begin{equation}\label{10.24}
K(t;y;X,Y):=\inf_{y=x+z}\bigl\{\|x\|_X+t\|z\|_Y\bigr\}.
\end{equation}

Second, the relative (Gagliardo) {\em completion} of $X$ in $Y$ denoted by $X^{c,Y}$ is a Banach space whose closed unit ball is the closure of that of $X$ in $Y$.

The relation between these notions is given by the following:
\begin{Prop}[\cite{BK-91}, Prop.\,2.2.20]\label{pr.10.6}
It is true that
\begin{equation}\label{10.25}
\|y\|_{X^{c,Y}}=\sup_{t>0} t^{-1}K(t;y;X,Y).
\end{equation}
\end{Prop}

Further, we use the following well-known result, see, e.g., \cite{JH-77} and references therein,
\begin{equation}\label{10.26}
K(t;f;L_p;W_p^k)\approx t\|f\|_p+\omega_{kp}(f;t^{\frac 1 k}),\quad t>0,
\end{equation}
with constants of equivalence depending only on $k$ and $d$.

Along with \eqref{10.25} and Theorem A this implies
\begin{equation}\label{10.27a}
\|f\|_{(W_1^k)^{c,L_1}}\approx\|f\|_1+\sup_{t>0}\frac{\omega_{k1}(f;t)}{t^k}:=\|f\|_1+|f|_{Lip_1^k}\approx \|f\|_1+|f|_{BV^k}.
\end{equation}
This and Theorem $B$ further imply
\begin{equation}\label{10.27}
\|f\|_{C^{c,L_1}}\le c(d)\bigl(\|f\|_1+|f|_{Lip_1^d}\bigr).
\end{equation}

By definition the closed unit ball of $C^{c,L_1}$ is the closure of the closed unit ball of $C$ in $L_1$, hence, $(C^{c,L_1},\|\cdot\|_{C^{c,L_1}})=(L_\infty,\|\cdot\|_\infty)$. Using this and a homothety of $\RR^d$ mapping $Q^d$ onto $Q$ we transform \eqref{10.27} to the inequality
\[
\|f\|_{L_\infty(Q)}\le c(d)\bigl(|Q|^{-1}\|f\|_{L_1(Q)}+|f|_{Lip_1^d(Q)}\bigr)
\]
Applying this to a function $f-m_Q$, where $m_Q\in\mathcal P_{k-1}^d$ is optimal for the distance from $f$ to $\mathcal P_{k-1}^d$ in $L_1(Q)$, we have
\begin{equation}\label{10.28}
E_{d\infty}(f;Q)\le \|f-m_Q\|_{L_\infty(Q)}\le c(d)\bigl(|Q|^{-1} E_{d1}(f;Q)+|f|_{Lip_1^d(Q)}\bigr).
\end{equation}

Finally, by the Taylor formula for functions $g\in W_1^d(Q)$ we have
\begin{equation}\label{10.29}
E_{d1}(g;Q)\le c(d)\,|Q|\cdot |g|_{W_1^d(Q)}.
\end{equation}

Since due to \eqref{10.27a} the Banach space $(BV^d(Q),\|\cdot\|_{L_1(Q)}+|\cdot|_{BV^d(Q)})$ coincides (up to equivalence of norms) with the relative completion of $W_1^d(Q)$ in $L_1(Q)$, for every $f\in BV^d(Q)= Lip_1^d(Q)$ (see Theorem A) there is a sequence $\{f_n\}_{n\in\N}\subset W_1^d(Q)$ such that $f_n\rightarrow f$ in $L_1(Q)$ and $|f_n|_{W_1^d(Q)}\rightarrow |f|_{BV^d(Q)}$ as $n\rightarrow\infty$.

Applying \eqref{10.29} to every $f_n$, then passing to limit and using Theorem A we obtain from \eqref{10.28} the required inequality
\begin{equation}\label{10.30}
E_{d\infty}(f;Q)\le c(d)\bigl( |Q|^{-1}\cdot |Q|\cdot |f|_{BV^d(Q)}+|f|_{Lip_1^d(Q)}\bigr)\le c_1(d)\, |f|_{BV^d(Q)}.
\end{equation}

Lemma \ref{pr10.5} is proved.
\end{proof}

Now let $\pi$ be a packing in $Q^d$. Using Lemma \ref{pr10.5} for each $Q\in\pi$ and $f\in BV^d$ we have
\[
\begin{array}{l}
\displaystyle \sum_{Q\in\pi} E_{d\infty}(f;Q)\le c(d)\sum_{Q\in\pi} |f|_{BV^d(Q)}\le c(d)\max_{|\alpha|=d}\left\{\sum_{Q\in\pi}{\rm var}_Q\, D^\alpha f\right\}\\
\\
\displaystyle \qquad\qquad\qquad\quad\le c(d)\max_{|\alpha|=d}{\rm var}_{Q^d}\, D^\alpha f:= c(d)|f|_{BV^d}.
\end{array}
\]
Taking here supremum over $\pi$ we then obtain by Theorem A
\[
|f|_{V_{1\infty}^d}\le c(d) |f|_{BV^d}\le c_1(d)|f|_{Lip_1^d}.
\]
In other words, we proved the embedding 
\begin{equation}\label{10.31}
Lip_1^d\subset \dot V_{1\infty}^d.
\end{equation}
This completes the proof of Proposition \ref{pro10.4}
\end{proof}

To complete the proof of part (b) of the theorem, we consider the composition of continuous embeddings $i_1:N\dot V_1^d\hookrightarrow Lip_1^d$, see \eqref{e2.2.19}, and $i_2: Lip_1^d\subset \dot V_{1\infty}^d$. The resulting embedding $i_2\circ i_1$ coincides with $L|_{N\dot V_1^d}$, i.e., sends a function in $N\dot V_1^d\subset\ell_\infty$ to its equivalence class in $\dot V_{1\infty}^d\subset L_\infty$. But according to Theorem \ref{te2.6}\,(c) it maps
$N\dot V_1^d$ isometrically onto $\dot V_{1\infty}^d$. Hence, $Lip_1^d=\dot V_{1\infty}^d$ and $N\dot V_1^d\cong Lip_1^d$.

The proof of Theorem \ref{teo2.2.11} is complete.
\end{proof}
\begin{proof}[Proof of Corollary \ref{cor9.7}]
First, we prove that each $\mathring f\in\dot W_1^d$ can be represented by a function $f\in AC$ (i.e., $f\in \mathring f$) such that
\begin{equation}\label{eq-n9.30a}
|f|_{V_1^d}\le c |\mathring f|_{W_1^d}
\end{equation}
for some $c$ depending only on $d$.

To this end we use the Gagliardo theorem \cite{Ga-59} asserting that each $\mathring g\in\dot W_1^d(Q)$, $Q\subset Q^d$, can be represented
by a (unique) continuous function $g$ and there
is a linear projection $P_Q: L_1(Q)\rightarrow\mathcal P_{d-1}^d$   such that
\begin{equation}\label{eq-n9.30b}
\max_Q |g|\le c(d)\left(\max_Q |P_Q(\mathring g)|+|\mathring g|_{W_1^d (Q)}\right).
\end{equation}

Let $f\in C$ be the representative of $\mathring f\, (\in \dot W_1^d)$.
Applying \eqref{eq-n9.30b} to $\mathring g=\mathring f|_{Q}-P_Q(\mathring f|_{Q})$ and to its continuous representative $g:=f|_Q-P_Q(\mathring f|_{Q})$ we obtain (as $P_Q(g)=0$):
\[
E_d(f;Q)\le \max_Q |f-P_Q(\mathring f)|\le c(d)|\mathring f|_{W_1^d(Q)}.
\]
Then for every packing $\pi\in\Pi$
\begin{equation}\label{eq-n9.30c}
\sum_{Q\in\pi} E_d(f;Q)\le c(d)\sum_{|\alpha|=d}\,\int_{\cup\,\pi} |D^\alpha f|\, dx,
\end{equation}
where $\cup\,\pi:=\underset{Q\in\pi}{\cup}\, Q$.

By the absolute continuity of the Lebesgue integral the right-hand side here tends to $0$ as $|\cup\pi|\rightarrow 0$. Hence, $f$ satisfies condition \eqref{eq-n9.30}, i.e., belongs to $AC$.

This and \eqref{eq-n9.30c} prove the required statement \eqref{eq-n9.30a}. \smallskip

Next, let $f\in AC\, (\subset\dot V_1^d)$. 
We prove that $L(AC)\subset \dot W_1^d$, where $L$ sends $f\in\dot V_1^d$ to its class of equivalence in $L_\infty$, and that
\begin{equation}\label{eq-n9.33a}
|L(f)|_{W_1^d}\le c |f|_{V_1^d}
\end{equation}
for some $c$ depending only on $d$.

This and \eqref{eq-n9.30a} will show that  $L|_{AC}$ is an isomorphism of $AC$ onto $\dot W_1^d$.

We start with the embeddings $AC\subset(\dot V_p^k)^0\subset C$ implying that $AC\subset N\dot V_1^d$, see Section~5 and Corollary \ref{cor2.1.3}\,(c). The latter, in turn, implies that
$L(f)|_Q\in BV^d(Q)$, $Q\subset Q^d$, and, moreover,
\[
|f|_{V_1^d(Q)}\approx |L(f)|_{BV^d(Q)}
\]
with the constants of equivalence depending only on $d$, see \eqref{10.17},  \eqref{10.20}.

Hence, for every $\pi\in\Pi$
\begin{equation}\label{eq-n9.33}
\sum_{Q\in\pi}|L(f)|_{BV^d(Q)}\le c(d)\sum_{Q\in\pi}|f|_{V_1^d(Q)}.
\end{equation}

Choosing for every $Q\in\pi$ a packing $\pi_Q\in\Pi(Q)$ such that
\[
|f|_{V_1^d(Q)}\le\sum_{R\in\pi_Q} E_d(f;R)+\frac{\varepsilon}{{\rm card}\, \pi},
\]
where $\varepsilon>0$ is arbitrary small, we obtain from \eqref{eq-n9.33} that
\[
\sum_{|\alpha|=k}\underset{\cup\,\pi}{\rm var}\, D^\alpha L(f)\le c(d)\left(\sum_{R\in\tilde\pi}E_d(f;R)+\varepsilon\right),
\]
where $\tilde\pi:=\cup_{Q\in\pi}\,\pi_Q$ is a packing.

Sending $|\pi|\ge|\tilde\pi|$ to zero we conclude from here and \eqref{eq-n9.30} that for all closed sets $S$ being  unions of nonoverlapping cubes, hence, for all Lebesgue measurable sets $S$
\[
\lim_{|S|\rightarrow 0}\underset{S}{\rm var}\, D^\alpha L(f)=0,\quad |\alpha|=d.
\]
By the Radon-Nikodym theorem, see, e.g., \cite[Sec.\,\mbox{III}.10]{DSch-58},
\[
(D^\alpha L(f))(S)=\int_S f_\alpha\, dx, \quad |\alpha|=d,
\]
for all Lebesgue measurable sets $S\subset Q^d$ and some $f_\alpha\in L_1$.

In other words, $D^\alpha L(f)\in L_1$ for all $|\alpha|=d$ and $|L(f)|_{BV^d}=|Lf|_{W_1^d}$.

Hence, $L$ maps $f\in AC$ in $\dot W_1^d$ and \eqref{eq-n9.33a} is fulfilled.

This proves that $L|_{AC}:AC\rightarrow \dot W_1^d$ is an isomorphism.

Finally, presentation \eqref{eq-n9.32} is valid for continuous representatives $f$ of elements $\mathring f\in\dot W_1^d$, see, e.g., \cite[Sec.\,1.10]{Ma-85}. As we have proved such $f\in AC$,  hence, 
this completes the proof of the corollary.
\end{proof}
\sect{Proof of Theorem \ref{te2.2.5}}
\subsection{Auxiliary Result}
Let $\kappa:\dot V_{p\infty}^k\rightarrow V_{p\infty}^k:=\dot V_{p\infty}^k/\mathcal P_{k-1}^d$ be the quotient map defining the latter space.
We then set
\begin{equation}\label{equ8.1}
\tilde  L:=\kappa\circ \bigl(L |_{\dot V_{p;S}^k}\bigr),
\end{equation}
where $L:\dot V_p^k\rightarrow \dot V_{p\infty}^k$ is the surjective map of norm one sending a function in $\dot V_p^k$ to its equivalence class in $L_\infty$

By the definition, $\tilde  L:\dot V_{p;S}^k\rightarrow V_{p\infty}^k$ is a bounded surjective linear map of norm $1$. (Recall that $\dot V_{p;S}^k\subset\dot V_p^k$ is the subspace of functions vanishing on the interpolating set $S$ and that we naturally identify $\dot V_{p;S}^k$ with $V_p^k$, see Section~4.1.1.) 

In the next result, we use the duality between $V_{p;S}^k\, (\equiv V_p^k)$ and $U_p^k$ established in Theorem \ref{te2.2.6} below and between $V_{p\infty}^k$ and $U_{p\infty}^k$ established in \cite[Th.\,2.6]{BB-18}.
\begin{Prop}\label{prop8.2}
The map $\tilde L:\dot V_{p;S}^k\rightarrow V_{p\infty}^k$ is weak$^*$ continuous with respect to the weak$^*$ topologies on $\dot V_{p;S}^k$ and $V_{p\infty}^k$.
\end{Prop}
\begin{proof}
It suffices to check that $\tilde L^{-1}(U)\subset \dot V_{p;S}^k$ is weak$^*$ closed for each set $U$ of the form $\{f\in V_{p\infty}^k\, :\, |f(g)|\le 1\}$, where $g\in U_{p\infty}^k\setminus\{0\}$. 
Since $\tilde L^{-1}(U)$ is a convex absorbing subset of $\dot V_{p;S}^k$,
it suffices by  the Krein-Smulian theorem, see, e.g., \cite[Thm. V.5.7]{DSch-58}, to check that $\tilde L^{-1}(U)\cap B(\dot V_{p;S}^k)$ is closed in the weak$^*$ topology. 

Let $\{f_\alpha\}_{\alpha\in\Lambda}$ be a net of functions in  $\tilde L^{-1}(U)\cap B(\dot V_{p;S}^k)$ converging in the weak$^*$ topology to a function $f\in B(\dot V_{p;S}^k)$. According to Theorem \ref{te2.1.3}, the space $\dot V_p^k$ consists of functions of the {\em first Baire class} (pointwise limits of continuous on $Q^d$ functions).  Moreover, due to Lemma \ref{lemma6.1}, the closed ball $B(\dot V_{p;S}^k)$ is compact in the topology of pointwise convergence on $Q^d$. In particular, the net
$\{f_\alpha\}_{\alpha\in\Lambda}$ pointwise converges to $f$. Since, in addition, $\sup_\alpha \|f_\alpha\|_{\infty}\le 1$, the Rosenthal Main Theorem \cite{Ro-77} implies that 
\begin{equation}\label{equ6.2}
\lim_{\alpha}\int_{Q^d}f_\alpha \,d\mu=\int_{Q^d} f\, d\mu
\end{equation}
for all signed Borel measures $\mu$ on $Q^d$.

Now we apply \eqref{equ6.2} to $d\mu= h\,dx$ with $h\in \hat L_1:=\{g\in L_1\, :\, g\perp\mathcal P_{k-1}^d\}$.  Since every function of $\hat L_1$ is an atom for the space $U_{p\infty}^k$, see the text before Theorem \ref{te2.2.5} in Section~2.3, $\hat L_1:=(U_{p\infty}^k)^0\subset U_{p\infty}^k$. Hence, we can use 
the duality pairing for  $(V_{p\infty}^k, U_{p\infty}^k)$, see \cite[Th.\,2.6]{BB-18},
to get
\begin{equation}\label{equ6.3}
\lim_{\alpha}[\tilde L(f_\alpha)](h):=\lim_{\alpha}\int_{Q^d}f_\alpha  h\,dx=\int_{Q^d}f h\,dx=:[\tilde L(f)](h).
\end{equation}

Next, by the definition $U_{p\infty}^k$ is the completion of $(U_{p\infty}^k)^0$; hence, for every $g\in U_{p\infty}^k$  there is a sequence  $\{g_n\}_{n\in\N}\subset \hat L_1\subset U_{p\infty}^k$ such that $\|g-g_n\|_{U_{p\infty}^k}\rightarrow 0$ as $n\rightarrow\infty$. 

Applying \eqref{equ6.3} to $h=g_n$  we then obtain
\[
\lim_{\alpha}\bigl([\tilde L(f_\alpha)](g_n-g)+[\tilde L(f_\alpha)](g)-[\tilde L(f)](g)\bigr) =[\tilde L(f)](g_n-g).
\]

Further, since by our assumptions $\tilde L(f), \tilde L(f_\alpha)\in \tilde L\bigl(B(\dot V_{p;S}^k)\bigr)\subseteq
B(V_{p\infty}^k)$ for all $\alpha\in\Lambda$, 
\[
\limsup_\alpha \left|[\tilde L(f_\alpha)] (g_n-g)\right|+
\left|[\tilde L(f)] (g_{n}-g)\right|\le 2\|g_n-g\|_{U_{p\infty}^k}\rightarrow 0
\quad {\rm as}\quad n\rightarrow \infty.
\]

This and the previous  imply that
\[
\limsup_{\alpha}\bigl|[\tilde L(f_\alpha)](g)-[\tilde L(f)](g)\bigr|\le \lim_{n\rightarrow\infty}\left(\limsup_\alpha \left|[\tilde L(f_\alpha)] (g_n-g)\right|+
\left|[\tilde L(f)] (g_{n}-g)\right|\right)=0,
\]
i.e., $
\lim_\alpha [\tilde L(f_\alpha)](g)=[\tilde L(f)](g)$ for all $g\in U_{p\infty}^k$.

But $\tilde L(f_\alpha)\in U\cap B(V_{p\infty}^k)$ for all $\alpha\in\Lambda$; hence, $\tilde L(f)$ also belongs to $U\cap B(V_{p\infty}^k)$ since this set is closed in the weak$^*$ topology of $V_{p\infty}^k$. In particular,
\[
f\in \tilde L^{-1}(\tilde L(f))\cap B(\dot V_{p;S}^k) \subset \tilde L^{-1}(U)\cap B(\dot V_{p;S}^k).
\]

This shows that the set $\tilde L^{-1}(U)\subset \dot V_{p;S}^k$ is weak$^*$ closed, as required. 

The proof of the proposition is complete.
\end{proof}
\subsection{Proof of Theorem \ref{te2.2.5}}
(a) We should prove that $U_p^k$ is a Banach space.
Since $U_p^k$ is the completion of $((U_p^k)^0,\|\cdot\|_{(U_p^k)^0})$, it suffices to prove that $\|\cdot\|_{(U_p^k)^0}$ is a norm, i.e., that if $\|g\|_{(U_p^k)^0}=0$ for some $g\in (U_{p}^k)^0$, then $g=0$. 

To prove this, we first show that for every $f\in\dot V_p^k\, (\subset\ell_\infty)$ and $g\in (U_p^k)^0\, (\subset\ell_1)$ 
\begin{equation}\label{eq8.4}
\left|\sum_{x\in Q^d}f(x)g(x)\right|\le |f|_{V_p^k}\|g\|_{U_p^k}.
\end{equation}

Let $m_Q\in\mathcal P_{k-1}^d$, $Q\subset Q^d$, be such that
\[
\|f-m_Q;Q\|_{\infty}=E_{k}(f;Q).
\]
Then for $b_\pi:=\sum_{Q\in\pi}c_Q\,a_Q$ we obtain by the definition of $(k,p)$-atoms
\[
\sum_{x\in Q^d}f(x) b_\pi (x)=\sum_{Q\in\pi} c_Q\sum_{x\in Q} (f(x)-m_Q(x))\,a_Q(x).
\]
Applying the H\"{o}lder inequality we derive from here
\[
\begin{array}{l}
\displaystyle \left|\,\sum_{x\in Q^d}f(x) b_\pi (x)\,\right|\le\left(\sum_{Q\in\pi} |c_Q|^{p'}\right)^{\frac{1}{p'}}\left(\sum_{Q\in\pi}\bigl(\|f-m_Q;Q\|_{\infty}\,\|a_Q\|_{1}\bigr)^p\right)^{\frac{1}{p}}\\
\\
\displaystyle \qquad\quad\qquad\qquad\ \le [b_\pi]_{p'} \left(\sum_{Q\in\pi}E_{k}(f;Q)^p\right)^{\frac 1p}\le [b_\pi]_{p'} |f|_{V_p^k}.
\end{array}
\]
Since $g\in (U_p^k)^0$  can be presented as a finite sum of $(k,p)$-chains $b_\pi$, this gives
\[
\left|\,\sum_{x\in Q^d}f(x) g(x) \,\right|\le\left(\inf\sum_{\pi}\,[b_\pi]_{p'}\right) |f|_{V_p^k}=\|g\|_{U_p^k}\, |f|_{V_p^k};
\]
here infimum is taken over all such presentations of $g$.

From this inequality and the equality $\|g\|_{U_p^k}=0$ we deduce that
\begin{equation}\label{eq8.5}
\sum_{x\in Q^d}f(x)g(x)=0\quad {\rm for\ all}\quad f\in \dot V_p^k .
\end{equation}
In particular, since $C^\infty\subset\dot{\textsc{v}}_p^k\subset\dot V_p^k$, see Section~5, equation \eqref{eq8.5} is valid for all $f\in C^\infty$.

Now assume, on the contrary, that $g$ satisfies \eqref{eq8.5} but $g\ne 0$. Let $X\subset Q^d$ be a finite subset of  ${\rm supp}\, g$ such that 
\[
\sum_{x\in X}|g(x)|\ge \frac{1}{2}\|g\|_{1}\, (>0).
\]
Clearly, there exists a function $\varphi_X\subset C^\infty$ such that 
\[
\varphi_X(x)={\rm sgn}(g(x)),\quad  x\in X,\quad {\rm and}\quad \|\varphi_X\|_{\infty}\le 1.
\] 
From here and \eqref{eq8.5} we obtain the contradiction
\[
\frac 1 2\|g\|_{1}\le \sum_{x\in X} |g(x)|=
\left| \sum_{x\in X} \varphi_X(x)g(x) \right|\le
\sum_{x\in Q^d\setminus X} |\varphi_X(x)| |g(x)|<\frac 1 2\|g\|_{1}.
\]
Hence, $g=0$ and  the proof of part (a) of the theorem is complete.\medskip

\noindent (b) We should prove that $B(U_p^k)$ is the closure of the symmetric convex hull of the set $\{b_\pi\in (U_p^k)^0\, :\, [b_\pi]_{p'}\le 1\}$. This, in fact, is proved in Theorem~2.5 of \cite{BB-18} for $B(U_{p\infty}^k)$. The proof can be easily adapted to the required case.\smallskip

\noindent (c) We should prove that $U_p^k$ is nonseparable and contains a separable subspace. $\hat U_p^k\equiv U_{p\infty}^k$ such that $(U_p^k/\hat U_p^k)^*\cong\ell_p$. 

Let us  show that $U_p^k$ is nonseparable. To this end, we as in Lemma \ref{lemma6.1} take $S\subset Q^d$ to be an interpolating set for the space $\mathcal P_{k-1}^d$ and a function $c_x: S\rightarrow\RR$, $x\in Q^d$, such that  
\[
m(x)=\sum_{s\in S}c_x(s) m(s)\quad {\rm for\ every}\quad m\in\mathcal P_{k-1}^d.
\]

Further, we set
\[
\delta_x':=\delta_x-\sum_{s\in S}c_x(s)\delta_S,
\]
where $\delta_y$ is the delta-function at $y\in\RR^d$.

If $x\ne y\in Q^d\setminus S$, then by definition
\[
\|\delta_{x}'-\delta_{y}'\|_{U_p^k}:=\sup_{f\in \dot V_{p;S}^k\,; |f|_{V_p^k}\le 1} |f(x)-f(y)|\ge
2^{-\frac 1 p}|f_{x,y}(x)-f_{x,y}(y)|=2^{\frac{1}{p'}};
\]
here $f_{x,y}(x)=1$, $f_{x,y}(y)=-1$ and $f_{x,y}(z)=0$ otherwise. 

Hence, $U_p^k$ contains a discrete uncountable subset, i.e., it is nonseparable.

Next, let us show that $U_p^k$ contains a separable subspace $\hat U_p^k$ such that
\[
\hat U_p^k\equiv U_{p\infty}^k\quad {\rm and}\quad (U_p^k/\hat U_p^k)^*\cong\ell_p .
\]

In fact, since $\tilde L: \dot V_{p;S}^k\rightarrow V_{p\infty}^k$ is weak$^*$ continuous and surjective, see Proposition \ref{prop8.2}, there is a bounded linear injective map $T: U_{p\infty}^k\rightarrow U_p^k$ such that $T^*=\tilde L$, see, e.g., \cite[VI.9.13]{DSch-58}. 

We set
\begin{equation}\label{e8.6}
\hat U_p^k:= T(U_{p\infty}^k)
\end{equation}
and prove that $T: U_{p\infty}^k\rightarrow (\hat U_p^k, \|\cdot\|_{U_p^k})$ is isometry.

Due to Theorem \ref{te2.6}\,(c) $\tilde L$ maps the Banach space $N\dot V_p^k\cap \dot V_{p;S}^k$ isometrically onto $V_{p\infty}^k$. 

Further, let $u\in U_{p\infty}^k\setminus \{0\}$ and $f_u\in V_{p\infty}^k$ be such that $f_u(u)=\|u\|_{U_{p\infty}^k}$ and $\|f\|_{V_{p\infty}^k}=1$. Then there exists $\tilde f_u\in N\dot V_p^k\cap \dot V_{p;S}^k$ such that $\tilde L(\tilde f_u)=f_u$ and $|\tilde f_u|_{V_p^k}=1$. In particular, we have
\[
\|T(u)\|_{U_p^k}\ge |\tilde f_u(Tu)|=|\tilde L(\tilde f_u)(u)|=|f_u(u)|=\|u\|_{U_{p\infty}^k}.
\]
On the other hand,
\[
\|T(u)\|_{U_p^k}\le \|T\|\cdot\|u\|_{U_{p\infty}^k}=\|\tilde L\|\cdot\|u\|_{U_{p\infty}^k}=\|u\|_{U_{p\infty}^k}
\]
and therefore
\[
\|T(u)\|_{U_p^k}=\|u\|_{U_{p\infty}^k}\quad {\rm for\ all}\quad u\in U_{p\infty}^k,
\]
i.e., $T: U_{p\infty}^k\rightarrow (\hat U_p^k, \|\cdot\|_{U_p^k})$ is an isometric isomorphism and $\hat U_p^k$ is a closed subspace of $U_p^k$. It is separable because $U_{p\infty}^k$ is  by \cite[Th.\,2.5(b)]{BB-18}.

Further, by the definition of dual to a factor-space, $(U_p^k/\hat U_p^k)^*$ is isomorphic to the annihilator $(\hat U_p^k)^{\perp}\subset V_p^k$ of $\hat U_p^k$.  Moreover, since $T: U_{p\infty}^k\rightarrow \hat U_p^k$ is an isomorphism, an element $f\in (\hat U_p^k)^{\perp}$ iff for every $u\in U_{p\infty}^k $
\[
f(T(u))=(\tilde L(f))(u)=0,
\]
i.e., iff $f\in {\rm ker}(\tilde L)\cong\ell_p$.

This shows that
$(U_p^k/\hat U_p^k)^*\cong\ell_p$ and completes the proof of the theorem.
\sect{Proof of Theorem \ref{te2.2.6} }
In the proof we use the following auxiliary result.
\begin{Lm}\label{lem9.1}
Let $f\in \ell_\infty(Q)$. There is a sequence of functions $\{g_n\}_{\in\N}\in \ell_1(Q)$ such that
\begin{equation}\label{eq9.1}
E_{k}(f;Q)=\lim_{n\rightarrow\infty}\sum_{x\in Q} f(x)g_n(x);
\end{equation}
\begin{equation}\label{eq9.2}
\|g_n\|_{\ell_{1}(Q)}=1\quad {\rm and}\quad
\sum_{x\in Q} x^\alpha g_n(x)=0,\quad \quad |\alpha|\le k-1,\quad {\rm for\ all}\quad n\in\N.
\end{equation}
\end{Lm}
\begin{proof}
Let $\hat\ell_1(Q):=\{g\in\ell_1(Q)\, :\, g\perp\mathcal P_{k-1}^d\}$. By the Hahn-Banach theorem 
$\hat\ell_1(Q)^*\equiv \ell_\infty(Q)/\mathcal P_{k-1}^d$. 
 Since the best approximation in \eqref{eq9.1} is the norm of the image of $f$ in 
 $\ell_\infty(Q)/\mathcal P_{k-1}^d|_Q$, the definition of the norm of a bounded linear functional on $\hat\ell_1(Q)$ implies existence of a sequence $\{g_n\}_{n\in\N}\subset\hat\ell_1(Q)$ of elements of norm $1$ satisfying \eqref{eq9.1}
 \end{proof}
 \begin{proof}[Proof of Theorem \ref{te2.2.6}]
 By definition, $(U_{p}^k)^0\subset\hat \ell_{1}$.
Further, every function $f\in \hat \ell_{1}$ is a $(k,p)$-chain subordinate to the packing $\pi=\{Q^d\}$. 

In fact, a function $f=c_{Q^d}\,a_{Q^d}$, where $c_{Q^d}:=\|f\|_{1}$ and $a_{Q^d}:=f/\|f\|_{1}$, vanishes on $\mathcal P_{k-1}^d$ and, moreover, $\|a_{Q^d}\|_{1}=1$. Hence, by the definition of the seminorm of   $U_p^k$, 
\begin{equation}\label{eq9.3}
\|f\|_{U_p^k}\le |c_{Q^d}|=\|f\|_{1}.
\end{equation}
In other words, the linear embedding
\begin{equation}\label{eq9.4}
E:\hat \ell_{1}\hookrightarrow U_p^k
\end{equation} 
holds with the embedding constant $1$ and has dense image. Passing to the conjugate map we obtain that 
\[
E^*: (U_p^k)^*\hookrightarrow \hat \ell_{1}^*\equiv \ell_\infty/\mathcal P_{k-1}^d
\] 
is a linear injection of norm $\le 1$.
On the other hand, $V_p^k$ is contained in $\ell_\infty/\mathcal P_{k-1}^d$. 

Further, we show that ${\rm range}(E^*)$ is in $V_p^k$ and  that the linear map $E^*:(U_p^k)^*\rightarrow V_p^k$ is of norm $\le 1$.

To this end, for $\ell\in (U_p^k)^*$ we denote by $f_\ell\in \ell_\infty$ an element whose image in $\ell_\infty/\mathcal P_{k-1}^d$ coincides with $E^*(\ell)$. Then we take for every $Q\subset Q^d$ and $\varepsilon\in (0,\frac 12)$ a $(k,p)$-atom denoted by $\tilde a_{Q}$ such that
\begin{equation}\label{eq9.5}
(1-\varepsilon)E_{k}(f_\ell;Q)\le \sum_{x\in Q} f_\ell(x) \tilde a_{Q}(x)\le E_{k}(f_\ell;Q);
\end{equation}
its existence directly follows from Lemma \ref{lem9.1} and the definition of $(k,p)$-atoms. 

Then for a $(k,p)$-chain $ \tilde b_{\pi}$ given by $ \tilde b_{\pi}:=\sum_{Q\in\pi}c_Q \tilde a_{Q}$ 
we get from \eqref{eq9.5}
\[
[E^*(\ell)]( \tilde b_{\pi})=\sum_{x\in Q^d} \tilde b_{\pi}(x) f_\ell(x)\ge (1-\varepsilon)\sum_{Q\in\pi}c_Q E_{k}(f_\ell;Q).
\]
This, in turn, implies
\[
\begin{array}{l}
\displaystyle
\sum_{Q\in\pi} c_Q  E_{k}(f_\ell;Q)\le (1-\varepsilon)^{-1}\|E^*(\ell)\|_{1}\| \tilde b_{\pi}\|_{U_p^k}\le (1-\varepsilon)^{-1}\|\ell\|_{(U_p^k)^*}\| \tilde b_{\pi}\|_{U_p^k}\\
\displaystyle
\qquad\qquad\qquad\quad\ \le (1-\varepsilon)^{-1}\left(\sum_{Q\in\pi}|c_Q|^{p'}\right)^{\frac{1}{p'}}\|\ell\|_{(U_p^k)^*}.
\end{array}
\]
Sending here $\varepsilon\rightarrow 0$ and then taking supremum over all sequences $(c_Q)_{Q\in\pi}$ of the $\ell_{p'}(\pi)$ norm $1$ and supremum over all $\pi$ we conclude that
\[
|f_\ell|_{V_p^k}:=\sup_{\pi}\left(\sum_{Q\in\pi}E_{k}(f_\ell;Q)^p\right)^{\frac 1p}\le \|\ell\|_{(U_p^k)^*}.
\]
Hence, $E^*(\ell)\in V_p^k$ for every $\ell\in (U_p^k)^*$ and $E^*:(U_p^k)^*\hookrightarrow V_p^k$ is a linear injection of norm $\le 1$.

Next, let us show that there is a linear injection $F$ of norm $\le 1$
\begin{equation}\label{eq9.6}
F:V_p^k\hookrightarrow (U_p^k)^*
\end{equation}
such that
\begin{equation}\label{eq9.7}
FE^*={\rm id}|_{(U_p^k)^*}.
\end{equation}

Actually, let $f\in \dot V_p^k$ and $\ell_f: (U_p^k)^0\rightarrow\RR$ be a linear functional given for $g\in (U_p^k)^0$ by
\begin{equation}\label{eq9.8}
\ell_f(g):=\sum_{x\in Q^d}f(x)g(x).
\end{equation}
Due to \eqref{eq8.4}
\[
|\ell_f(g)|\le |f|_{V_p^k}\|g\|_{U_p^k}.
\]
Thus, $\ell_f$ continuously extends to a linear functional  from $(U_p^k)^*$ (denoted by the same symbol) and the linear map $F:V_p^k\rightarrow (U_p^k)^*$,  $\{f\}+\mathcal P_{k-1}^d\mapsto \ell_f$, is of norm $\le 1$. 

Further, we show that $F$  is an injection. 

Indeed, let $\ell_f=0$ for some $f\in \dot V_p^k$. Since $(U_p^k)^0=\hat \ell_{1}$ and $\dot V_p^k\subset \ell_\infty$, equality \eqref{eq9.8} implies that $\ell_f|_{(U_p^k)^0}$ determines the trivial functional on $(\hat \ell_{1})^*\equiv \ell_\infty/\mathcal P_{k-1}^d$. Hence, $f\in  \mathcal P_{k-1}^d$, i.e., $f$ determines the zero element of the factor-space $V_p^k$, as required.

Further, by the definitions of $E^*$ and $F$ we have for each $h\in (U_p^k)^*$ and $g\in (U_p^k)^0$ 
\[
[FE^*(h)](g)=\ell_{E^*(h)}(g)=\sum_{x\in Q^d}E^*(h)(x)g(x)=h(E(g))=h(g).
\]
The proof of \eqref{eq9.6} and \eqref{eq9.7} is complete. 

In turn, the established results mean that ${\rm range}(E^*)=V_p^k$, ${\rm range}(F)=(U_p^k)^*$ and $F$ and $E^*$ are isometries.

The theorem is proved.
 \end{proof}
 \sect{Proof of Theorem \ref{te2.3.2}}
 \noindent (a) We should prove that the map sending a chain $b\in (U_p^k)^0\, (\subset\ell_1)$ to a discrete measure $\mu_b\in (\widetilde U_p^k)^0\, (\subset M)$, where $\mu_b(\{x\}):=b(x)$, $x\in Q^d$, extends to an isometric embedding $\mathcal I$ of $U_p^k$ into  $\widetilde U_p^k$.
  
In fact,  by the definitions of seminorms of $(U_p^k)^0$ and $(\widetilde U_p^k)^0$, see Sections~2.3, 2.4, we obtain
\begin{equation}\label{eq10.3}
\|\mu_b\|_{\widetilde U_p^k}\le \|b\|_{U_p^k}\quad {\rm for\ all}\quad b\in (U_p^k)^0,
\end{equation}
i.e., the map $(U_p^k)^0\ni b\mapsto\mu_b\in (\widetilde U_p^k)^0$ extends by continuity to a bounded linear map $\mathcal I: U_p^k\rightarrow \widetilde U_p^k$ of norm $\le 1$. 

The fact that $\mathcal I$ is an isometry is a  consequence of the following properties of the introduced below operator $\mathcal E: \widetilde U_p^k\rightarrow U_p^k$  that will be established in part (b) of the proof:
\[
\mathcal E\circ\mathcal I={\rm id}\quad {\rm and}\quad \|\mathcal E\|\le 1.
\]

Actually, these imply
\[
\|u\|_{U_p^k}=\|\mathcal E(\mathcal I(u))\|_{U_p^k}\le \|\mathcal I(u)\|_{\widetilde U_p^k}\le \|u\|_{U_p^k}\quad {\rm for\ all}\quad u\in U_p^k.
\]
Hence, $\mathcal I: U_p^k\rightarrow \widetilde U_p^k$ is an isometric embedding.\smallskip

\noindent (b) We should prove that there exists a linear continuous surjection $\mathcal E: \widetilde U_p^k\rightarrow U_p^k$ such that
 \[
 {\rm ker}(\mathcal E)\cap (\widetilde U_p^k)^0=(0)\quad {\rm and}\quad \mathcal E\circ \mathcal I={\rm id}.
 \]
 
To define such $\mathcal E$ we use a bilinear form given for $v\in\dot V_p^k$ and $\mu\in (\widetilde U_p^k)^0$ by 
\begin{equation}\label{2.3.2}
\langle v,\mu\rangle :=\int_{Q^d}v\,d\mu.
\end{equation}
Due to Theorem \ref{te2.1.3} this form is correctly defined.
As we show now it can be continuously extended to the seminormed Banach space $\dot V_p^k\times\widetilde U_p^k$. (We will retain notation $\langle\cdot,\cdot\rangle$ for the extension.) 

To this end it suffices to establish the inequality
\begin{equation}\label{2.3.3}
|\langle v,\mu\rangle |\le |v|_{V_p^k}\|\mu\|_{\widetilde U_p^k},\quad (v,\mu)\in \dot V_p^k\times (\widetilde U_p^k)^0,
\end{equation}
whose proof repeats line by line the proof of inequality \eqref{eq8.4} with sum replaced by integral, $a_Q\in\ell_1$ by $\mu_Q\in M$ and $\|\cdot\|$ by $\|\cdot\|_M$. We leave the details to the readers.

 Thus, inequality \eqref{2.3.3} is valid for all $\mu\in \widetilde U_p^k$ and, hence,
for every such $\mu$  the map
\[
\ell_\mu: v\mapsto \langle v,\mu\rangle,\quad v\in \dot V_{p;S}^k\, (\equiv V_p^k),
\]
determines a linear continuous functional from $(V_p^k)^*$.  

Now we define the required map $\mathcal E$ by the formula
\begin{equation}\label{2.3.4}
\mathcal E(\mu):=\ell_\mu,\quad \mu\in \widetilde U_p^k.
\end{equation}
Then $\mathcal E$ acts linearly from $\widetilde U_p^k$ to $(V_p^k)^*$ and is of norm $\le 1$. We should prove that the range of $\mathcal E$ coincides with $U_p^k\, (\subset (U_p^k)^{**}\equiv (V_p^k)^{*})$. 

To this end we establish the following: 

{\em Each bounded linear functional $\ell_\mu$ with $\mu\in (\widetilde U_p^k)^0$ is weak$^*$ continuous on $V_p^k$ with respect to the weak$^*$ topology induced by the duality} $(U_p^k)^*\equiv V_p^k$. 

It suffices to check that $\ell_\mu^{-1}([-1,1])\subset V_p^k$, $\ell_\mu:=\mathcal E(\mu)$, is weak$^*$ closed. In fact, since $\ell_\mu^{-1}([-1,1])$ is a convex absorbing subset of $ V_{p}^k$,
it suffices by  the Krein-Smulian theorem, see, e.g., \cite[Thm.\,V.5.7]{DSch-58}, to check that $\ell_\mu^{-1}([-1,1])\cap B(V_{p}^k)$ is closed in the weak$^*$ topology. In the proof of this, 
without loss of generality we identify $V_p^k$ with $\dot V_{p;S}^k\subset \dot V_p^k$, see Section~4.1.1, so that $B(V_{p}^k)=B(\dot V_{p;S}^k)$.

Let $\{f_\alpha\}_{\alpha\in\Lambda}$ be a net of functions in  $\ell_\mu^{-1}([-1,1])\cap B(\dot V_{p;S}^k)$ converging in the weak$^*$ topology to a function $f\in B(\dot V_{p;S}^k)$. Since $\dot V_p^k$ consists of functions of the  first Baire class, see Theorem \ref{te2.1.3}, and the closed ball $B(\dot V_{p;S}^k)$ is compact in the topology of pointwise convergence on $Q^d$ (see Lemma \ref{lemma6.1}), the net
$\{f_\alpha\}_{\alpha\in\Lambda}$ pointwise converges to $f$; in addition, $\sup_\alpha \|f_\alpha\|_{\infty}\le 1$ and $\|\ell_\mu\|\le 1$. These and the Rosenthal Main Theorem \cite{Ro-77} imply that 
\begin{equation}\label{eq10.1}
\ell_\mu(f):=\int_{Q^d} f\, d\mu=\int_{Q^d}\lim_{\alpha}f_\alpha \,d\mu=\lim_{\alpha}\int_{Q^d}f_\alpha\,d\mu=
\lim_{\alpha}\ell_\mu(f_\alpha)\, (\in [-1,1]);
\end{equation}
hence, $f\in \ell_\mu^{-1}([-1,1])$ and $\ell_\mu\in (V_p^k)^*$ is weak$^*$ continuous. 

Since the subspace of weak$^*$ continuous functionals in $(V_p^k)^*$ coincides with $U_p^k$, we conclude that $\mathcal E(\mu)\in U_p^k$ for all $\mu\in (\widetilde U_p^k)^0$. Moreover, $U_p^k$ is a closed subspace of $(V_p^k)^*$ and $\mathcal E$ is continuous on $(\widetilde U_p^k)^0$ that is dense in $\widetilde U_p^k$. These facts imply that
${\rm range}(\mathcal E)\subset U_p^k$.  

Next, we show that 
\begin{equation}\label{eq-n12.6}
{\rm ker}(\mathcal E)\cap (\widetilde U_p^k)^0=(0).
\end{equation}

In fact, let $\mu\in {\rm ker}(\mathcal E)\cap (\widetilde U_p^k)^0$. Then for each $f\in C^\infty\subset\dot V_p^k$
\[
\bigl(\mathcal E(\mu)\bigr)(f):=\int_{Q^d}f\, d\mu=0.
\]
Since $C^\infty $ is dense in $C$ in the topology of uniform convergence, the latter implies that
\begin{equation}\label{eq10.5}
\int_{Q^d}f\, d\mu=0\quad {\rm for\ all}\quad f\in C.
\end{equation}
According to the Riesz representation theorem $C^*\equiv (M,\|\cdot\|_M)$ with the duality defined by integration with respect to the corresponding measure; hence, \eqref{eq10.5} implies that $\mu=0$ proving \eqref{eq-n12.6}.

Finally, we prove that
\begin{equation}\label{eq10.4}
\mathcal E\circ \mathcal I={\rm id}.
\end{equation}

In fact, we have for $b\in (U_p^k)^0\, (\subset\ell_1)$ and all $v\in \dot V_{p;S}^k\, (\subset\ell_\infty)$
\[
\bigl(\mathcal E(\mathcal I(b))\bigr)(v)=\int_Q v\, d\mu_b=\sum_{x\in Q^d}v(x) b(x)=b(v).
\]
Since by Theorem \ref{te2.2.6} $(U_p^k)^*\equiv V_p^k$, the latter implies that
\[
\mathcal E(\mathcal I(b))=b\quad {\rm for\ all}\quad b\in (U_p^k)^0.
\]
From here using  that $\mathcal E$ and $\mathcal I$ are continuous maps and that  $(U_p^k)^0$ is dense in $U_p^k$ we obtain the required identity \eqref{eq10.4}.

In particular, this implies that ${\rm range}(\mathcal E)=U_p^k$, i.e., surjectivity of $\mathcal E$, and completes the proof of part (b) and hence of the theorem.
\sect{Proofs of Theorem \ref{te2.3.3} and Corollary \ref{cor2.3.4}}
\begin{proof}[Proof of Theorem \ref{te2.3.3}] We have to prove that if
\begin{equation}\label{e11.1}
s:=\frac{d}{p}< k,\quad 1<p<\infty,
\end{equation}
then
\begin{equation}\label{e11.2}
(\textsc{v}_p^k)^*\cong U_p^k.
\end{equation}
Since the relation
\begin{equation}\label{eq-n13.3}
(\textsc{v}_{p\infty}^k)^*\cong U_{p\infty}^k
\end{equation}
under condition \eqref{e11.1} has been just proved in \cite[Th.\,2.7]{BB-18} and the spaces of our consideration are defined similarly to those in \eqref{eq-n13.3}, several of the basic arguments of the cited paper can and will after some trivial modification of notations be used in the forthcoming proof.

In the following text we as before identify $U_p^k$ with its image under the natural embedding $U_p^k\hookrightarrow (U_p^k)^{**}$. 
and  $V_p^k$ with $(U_p^k)^*$, see Theorem \ref{te2.2.6}; hence, we regard $U_p^k$ as a linear subspace of $(V_p^k)^*\, (=((U_p^k)^*)^*)$.

Further, $\mathfrak i:\textsc{v}_p^k\hookrightarrow V_p^k$ is the natural embedding and $\mathfrak i^*:(V_p^k)^*\rightarrow(\textsc{v}_p^k)^*$ is its adjoint. 
\begin{Prop}\label{prop11.1}
(a) $\mathfrak i^*$ is a surjective linear map of norm one which maps $U_p^k$ isomorphically onto a closed subspace of
$(\textsc{v}_p^k)^*$.

\noindent (b) The image $\mathfrak i^*(B(U_p^k))$ is a dense subset of  $B((\textsc{v}_p^k)^*)$ in the weak$^*$ topology of $(\textsc{v}_p^k)^*$. 
\end{Prop}
\begin{proof}
The proof that $\mathfrak i^*: (V_p^k)^*\rightarrow (\textsc{v}_p^k)^*$ is a surjective linear map of norm one repeats line by line the proof of the similar assertion of Proposition 6.1 in \cite{BB-18}. The same is true for part (b) of Proposition \ref{prop11.1}. 

To prove the remaining statement of part (a) of the proposition we use
the following analog of Lemma 6.2 of \cite{BB-18}.
\begin{Lm}\label{lem11.2}
The subspace $\textsc{v}_p^k$ is weak$^*$ dense in the space $V_p^k\, (=(U_p^k)^*)$.
\end{Lm}

\begin{proof}
The result is a straightforward corollary of Theorem \ref{te2.1.4}.
\end{proof}
Now, to prove that $\mathfrak i^*$ maps $U_p^k$ isomorphically onto a closed subspace of 
$(\textsc{v}_p^k)^*$, we have to show that  there is a constant $c\in (0,1)$ such that $\|\mathfrak i^*(v)\|_{(\textsc{v}_p^k)^*}\ge c\|v\|_{U_p^k}$ for all $v\in U_p^k$.

In fact, let $u\in U_p^k\setminus\{0\}$. By the Hahn-Banach theorem there exists $f\in V_p^k$ such that $\|f\|_{V_p^k}=1$ and $f(u)=\|u\|_{U_p^k}$.  By Theorem \ref{te2.1.4} there exists a sequence $\{f_n\}_{n\in\N}\subset \textsc{v}_p^k$ weak$^*$ converging to $f$ such that
\[
\sup_n \|f_n\|_{V_p^k}\le C\|f\|_{V_p^k}=C
\]
for some $C\ge 1$ independent of $f$.

These imply that
\[
\|v\|_{U_p^k}=|f(v)|=\lim_{n\to\infty}|f_n(v)|\le C\sup_{g\in B(\textsc{v}_p^k)}|g(v)|=C\sup_{g\in B(\textsc{v}_p^k)}|(\mathfrak i^*(v))(g)|:=
C\|\mathfrak i^*(v)\|_{(\textsc{v}_p^k)^*},
\]
as required.

This completes the proof of Proposition \ref{prop11.1}.
\end{proof}

Our next result is the analog of Proposition 6.3 of \cite{BB-18}. In its formulation, we use Theorem \ref{te2.3.2} to identify the vector space of measures $\hat M:=(\widetilde U_p^k)^0$ vanishing on $\mathcal P_{k-1}^d$ with a subspace of $U_p^k$. Let $\bar{\mathscr B}_p^k$  stand for the closure in $U_p^k$ of the set of $(k,p)$-chains
\begin{equation}\label{equ11.3}
\mathscr B_p^k:=\{b_\pi\in (\widetilde U_p^k)^0\, :\, [b_\pi]_{p'}\le 1\}\subset\hat M,
\end{equation}
see \eqref{2.2.1} for the definition.
\begin{Prop}\label{prop11.3}
 $\mathfrak i^*(\bar{\mathscr B}_p^k)$ is a subset of $B((\textsc{v}_p^k)^*)$ compact
in the weak$^*$ topology of $(\textsc{v}_p^k)^*$.
\end{Prop}
\begin{proof}
As in  the proof of \cite[Prop.\,6.3]{BB-18}, the required result follows from the next one.
\begin{Stat}\label{stat11.4}
If $\{b^n\}_{i\in\N}\subset\mathscr B_p^k$ is such that the sequence $\{\mathfrak i^*(b^n)\}_{n\in\N}$ weak$^*$ converges in $B((\textsc{v}_p^k)^*)$, then its limit belongs to $\mathfrak i^*(\bar{\mathscr B}_p^k)$.
\end{Stat}
\begin{proof}
Let $b^n$ has the form
\[
b^n:=\sum_{i=1}^{N(n)} c_i^n \mu_{Q_i^n},\quad n\in\N,
\]
where $\pi_n:=\{Q_i^n\, :\, 1\le i\le N(n)\}\in\Pi$ is a packing.
\end{proof}
As in the cited proposition we assume without loss of generality that
\[
|Q_{i+1}^n|\le |Q_i^n|,\quad 1\le i < N(n).
\]
Further, we extend sequences $\pi_n$, $\{c_i^n\}$ and $\{\mu_{Q_i^n}\}$  by setting
\[
Q_i^n:=\{0\},\quad c_i^n:=0,\quad \mu_{Q_i^n}:=0\quad {\rm for}\quad i>N(n).
\]
Hence, we write
\[
b^n:=\sum_{i=1}^\infty c_i^n \mu_{Q_i^n},\quad n\in\N.
\]

\begin{Lm}\label{lem11.5}
There is an infinite subsequence $\{b^n\}_{n\in J}$, $J\subset\N$, such that for every $i\in\N$ the following is true.
\begin{itemize}
\item[(a)]
$\{Q_{i}^n\}_{n\in J}$  converges in the Hausdorff metric to a closed subcube of $Q^d$ denoted by $Q_i$; 
\item[(b)] $\{\mathfrak i^*(\mu_{Q_{i}^n})\}_{n\in\N}\subset B((\textsc{v}_p^k)^*)$  converges in the weak$^*$ topology of $B((\textsc{v}_p^k)^*)$;
\item[(c)] if the limiting cube $Q_{i}$ has a nonempty interior, then the sequence of measures $\{\mu_{Q_{i}^n}\}_{n\in J}$ converges in the weak$^*$ topology of $M$ (regarded as the dual space of $C:=C(Q^d)$);
\item[(d)] the sequence $\{c^n:=(c_{i}^n)_{i\in\N}\}_{n\in J}$ of vectors from $B(\ell_{p'}(\N))$ converges in the weak$^*$ topology of $\ell_{p'}(\N)\, (=\ell_p(\N)^*)$ to a vector denoted  by $c\, (\in B(\ell_{p'}(\N)))$.
\end{itemize}
\end{Lm}
\begin{proof}
The result, in fact, is proved in \cite[Lm.\,6.5]{BB-18}.
\end{proof}
Hence, without loss of generality we can and will assume that the sequence $\{b^n\}\subset\mathscr B_p^k$ of $(k,p)$-chains satisfies the assertions of Lemma \ref{lem11.5}. 
In particular, there are closed cubes $Q_i\subset Q^d$, $i\in\N$, such that in the Hausdorff metric
\begin{equation}\label{eq11.3}
Q_i=\lim_{n\to\infty} Q_i^n.
\end{equation}
Since for each $n\in\N$ the cubes $Q_i^n$, $i\in\N$, are nonoverlapping and their volumes form a nonincreasing  sequence, the same is true for the family of cubes $\{Q_i\}_{i\in\N}$. Thus, for every $i\in\N$
\begin{equation}\label{eq11.4}
\mathring{Q}_i\cap\mathring{Q}_{i+1}=\emptyset\quad {\rm and}\quad |Q_i|\ge |Q_{i+1}|.
\end{equation}

Now we let $N=\infty$ if $|Q_i|\ne 0$ for all $i\in\N$, otherwise,
 $N$ be the minimal element of the set of integers $n\in\Z_+$ such that
\begin{equation}\label{eq11.5}
 |Q_i|=0\quad {\rm for}\quad i> n.
\end{equation}
Then due to Lemma \ref{lem11.5}\,(c) for  $N\ne 0$ there are measures $\mu_i\in M$, $1\le i< N+1$, such that in the weak$^*$ topology of $M$
\begin{equation}\label{eq11.6}
\mu_i=\lim_{n\to\infty}\mu_{Q_i^n}.
\end{equation}
\begin{Lm}\label{lem11.6}
(1) If $N\ne \infty$ and $i> N$, then in the weak$^*$ topology of $B((\textsc{v}_p^k)^*)$
\begin{equation}\label{eq11.7}
\lim_{n\to\infty} i^*(\mu_{Q_i^n})=0.
\end{equation}
(2) If $N\ne 0$ and $1\le i<N+1$, then the measure $\mu_i$ is a $(k,p)$-atom subordinate to $Q_i$.
\end{Lm}
\begin{proof}
The proof of part (1) repeats with the corresponding change of notations and definitions the analogous assertion in Lemma 6.6 of \cite{BB-18}.

\noindent (2) First, due to \eqref{eq11.6}
\[
\|\mu_i\|_{M}\le 1
\]
 as atoms $\mu_{Q_i^n}$ belong to the closed unit ball of $M=C^*$ which is weak$^*$ compact.

Next, we prove that
\begin{equation}\label{eq11.8}
{\rm supp}\, \mu_i\subset Q_i\quad {\rm and}\quad \mu_i\perp\mathcal P_{k-1}^d.
\end{equation}
If, on the contrary, ${\rm supp}\, \mu_i\setminus Q_i\ne\emptyset$, then there is a continuous function $f\in C$ such that $({\rm supp}\, f)\cap Q_i=\emptyset$ and
\begin{equation}\label{eq11.9}
\int_{Q^d}fd\mu_i\ne 0.
\end{equation}
However, $Q_i^n\to Q_i$ in the Hausdorff metric
as $n\to\infty$ and therefore $Q_i^n\cap {\rm supp}\, f=\emptyset$ for all sufficiently large $n$.
This and condition (c) of Lemma \ref{lem11.5} imply that
\[
0=\lim_{n\rightarrow\infty}\int_{Q^d} f d\mu_{Q_i^n}=\int_Q f d\mu_i
\]
in contradiction with \eqref{eq11.9}.

To prove the second assertion of \eqref{eq11.8} we note that $\mu_{Q_i^n}\perp\mathcal P_{k-1}^d$ for all $i$ and $n$. Then according to condition (c) of Lemma \ref{lem11.5} we have for each $m\in\mathcal P_{k-1}^d$ 
\[
\int_{Q^d}m d\mu_i=\int_{Q^d}m d\mu_{Q_i^n}=0,
\]
as required.

Hence, $\mu_i$ (in the sequel denoted by $\mu_{Q_i}$) is a $(k,p)$-atom subordinate to $Q_i$.

The proof of the lemma is complete.
\end{proof}

Now we show that for $1\le N<\infty$
\begin{equation}\label{eq11.10}
v_N:=\sum_{i=1}^{N} c_i \mu_{Q_i}\in\mathscr B_p^k.
\end{equation}

In fact, by Lemma \ref{lem11.6}\,(2) and \eqref{eq11.4} $\mu_{Q_i}$ are $(k,p)$-atoms and $\{Q_i\}_{1\le i\le N}$ is a packing. Moreover, by Lemma \ref{lem11.5}\,(d) the $(k,p)$-atom $v_N$ satisfies 
\[
[v_N]_{p'}=\|c\|_{p'}\le 1,
\]
as required.

Further, for $N=\infty$ 
\begin{equation}\label{eq11.11}
v_\infty:=\sum_{i=1}^\infty c_i \mu_{Q_i}\in\bar{\mathscr B}_p^k
\end{equation}
In fact, by Lemmas \ref{lem11.5}\,(d) and \ref{lem11.6}\,(2), and inequality \eqref{2.3.3},
\[
\left\|\sum_{i=\ell}^m c_i\mu_{Q_i}\right\|_{U_p^k}\le \left(\sum_{i=\ell}^m |c_i|^{p'}\right)^{\frac{1}{p'}}\to 0
\]
as $\ell, m\to\infty$, i.e., the series  in \eqref{eq11.11} converges in $U_p^k$.
Moreover, its partial sums belong to $\mathscr B_p^k$, cf. \eqref{eq11.10}, hence, $v_\infty$ belongs to the closure of $\mathscr B_p^k$.

In the remaining case of $N=0$ we set
\begin{equation}\label{eq11.12}
v_0:=0.
\end{equation}
Now the proof of Statement \ref{stat11.4} will be completed if we show that
\begin{equation}\label{eq11.13}
\lim_{n\to\infty} \mathfrak i^*(b^n)=\mathfrak i^*(v_N)
\end{equation}
in the weak$^*$ topology of $B((\textsc{v}_p^k)^*)$.

However, the derivation of \eqref{eq11.13} repeats line by line the proof of the similar statement in \cite{BB-18}, see equation (6.23) there.

The proof of Statement \ref{stat11.4} is complete. 
\end{proof}

 Now, we complete the proof of Theorem \ref{te2.3.3} beginning with the following:
\begin{Lm}\label{lem11.7}
The space $\textsc{v}_p^k$ isometrically embeds in the space  $C(\mathfrak i^*(\bar{\mathscr B}_p^k))$ of continuous functions on the (metrizable) compact space $\mathfrak i^*(\bar{\mathscr B}_p^k)$,
\end{Lm}
\begin{proof}
By Theorem \ref{te2.2.5}\,(b) $B(U_p^k)$ is the closure of the symmetric convex hull of the set $\{b_\pi\in (U_p^k)^0\, :\, [b_\pi]_{p'}\le 1\}$.  Since the latter set is a subset of $\mathscr B_p^k\subset B(U_p^k)$, see \eqref{equ11.3}, $B(U_p^k)$ is the closure of the symmetric convex hull of $\mathscr B_p^k$, denoted by ${\rm sc}(\mathscr B_p^k)$.
As in \cite[Lm.\,6.7]{BB-18} this implies that the set
\begin{equation}\label{equ11.15}
\mathfrak i^*({\rm sc}(\bar{\mathscr B}_p^k))={\rm sc}(\mathfrak i^*(\bar{\mathscr B}_p^k))
\end{equation}
is weak$^*$ dense in $B((\textsc{v}_p^k)^*)$ and for every element $v\in \textsc{v}_p^k$ regarded as a bounded linear functional on $(\textsc{v}_p^k)^*$ 
\begin{equation}\label{equ11.16}
\|v\|_{\textsc{v}_p^k}=\sup_{v^*\in \mathfrak i^*(\bar{\mathscr B}_p^k)}|v(v^*)|.
\end{equation} 

Since $v|_{\mathfrak i^*(\bar{\mathscr B}_p^k)}$ is a continuous function on $\mathfrak i^*(\bar{\mathscr B}_p^k)$ in the weak$^*$ topology induced from $B((\textsc{v}_p^k)^*)$ and its supremum norm equals $\|v\|_{\textsc{v}_p^k}$, the map
\[
\textsc{v}_p^k\ni v\mapsto v(v^*),\quad v^*\in \mathfrak i^*(\bar{\mathscr B}_p^k),
\]
is a linear isometric embedding of $\textsc{v}_p^k$ in $C(\mathfrak i^*(\bar{\mathscr B}_p^k))$.
\end{proof}

Now let $v^*$ be a linear continuous functional on the space $\textsc{v}_p^k$ regarded as the closed subspace of the space $C(\mathfrak i^*(\bar{\mathscr B}_p^k))$. By the Hahn-Banach theorem $v^*$ can be extended to a linear continuous functional, say, $\hat v^*$ on this  space with the same norm. In turn, by the Riesz representation theorem there is a {\em regular finite (signed) Borel measure} on the compact space $\mathfrak i^*(\bar{\mathscr B}_p^k)$ denoted by $\mu_{v^*}$ that represents $\hat v^*$.

This implies that
\begin{equation}\label{equ11.17}
v(v^*)=\int_{\mathfrak i^*(\bar{\mathscr B}_p^k)}v\, d\mu_{v^*},\quad v\in\textsc{v}_p^k.
\end{equation}

Now we use this measure to find a similar representation for elements of $V_p^k$.

Since $\textsc{v}_p^k$ is a weak$^*$ dense subspace of the space $V_p^k$, see Lemma \ref{lem11.2} and Theorem \ref{te2.1.4}, for every $v\in V_p^k$ there is a {\em bounded} in the $\textsc{v}_p^k$ norm sequence $\{v_j\}_{j\in\N}\subset\textsc{v}_p^k$ such that
\begin{equation}\label{equ11.18}
\lim_{j\to\infty}v_j(u)=v(u),\quad u\in U_p^k.
\end{equation}

Now let $\tau: \mathfrak i^*(U_p^k)\rightarrow U_p^k$ be the inverse to the {\em injection} $\mathfrak i^*|_{U_p^k}: U_p^k\rightarrow (\textsc{v}_p^k)^*$, see Proposition \ref{prop11.1}\,(a). Making the change of variable $u\to \tau(v^*)$ we derive from \eqref{equ11.18}
\begin{equation}\label{equ11.19}
\lim_{j\to\infty} v_j(v^*)=(v\circ\tau)(v^*),\quad v^*\in \mathfrak i^*(\bar{\mathscr B}_p^k).
\end{equation}
Since linear functionals $v_j:\textsc{v}_p^k\to\RR$ are continuous in the weak$^*$ topology defined by $(\textsc{v}_p^k)^*$ their traces to $\mathfrak i^*(\bar{\mathscr B}_p^k)$ are continuous functions in the weak$^*$ topology induced from $B((\textsc{v}_p^k)^*)$. This implies the following:
\begin{Lm}\label{lem11.8}
The function $(v\circ\tau)|_{\mathfrak i^*(\bar{\mathscr B}_p^k)}$ is  $\mu_{v^*}$-integrable and bounded.

Moreover, a function $\phi_{v^*}:V_p^k\to\RR$ given by
\begin{equation}\label{equ11.20}
\phi_{v^*}(v):=\int_{\mathfrak i^*(\bar{\mathscr B}_p^k)}v\circ\tau\,d\mu_{v^*}
\end{equation}
belongs to $(V_p^k)^*$.
\end{Lm}
\begin{proof}
The result is, in fact, proved in Lemma~6.8  of \cite{BB-18}.
\end{proof}

At the next stage we establish  weak$^*$ continuity of  $\phi_{v^*}$ on $V_p^k$ regarded as the dual space of $U_p^k$, see Theorem \ref{te2.2.6}. 

To this end it suffices to show that  $\phi_{v^*}^{-1}(R)\subset V_p^k$ is weak$^*$ closed for every closed interval  $R\subset\RR$. Since this preimage is convex, we can use the Krein-Smulian weak$^*$ closedness criterion, see, e.g., \cite[Thm.\,V.5.7]{DSch-58}. In our case, it asserts that 
$\phi_{v^*}^{-1}(R)$ is weak$^*$ closed iff
$B_r(0)\cap \phi_{v^*}^{-1}(R)$ is for every $r>0$; here $B_r(0):=\{v\in V_p^k\, :\, \|v\|_{V_p^k}\le r\}$.

Without loss of generality we identify $(V_p^k,\|\cdot\|_{V_p^k})$ with $(\dot V_{p;S}^k, |\cdot|_{V_p^k})$, where $\dot V_{p;S}^k\subset \dot V_p^k$ is the subspace of functions vanishing on an interpolating set $S\subset Q^d$ for $\mathcal P_{k-1}^d$, see Section~4.1.1. 

Further, let $\mathcal B_1(Q^d)$ stand for the space of functions on $Q^d$ of the first Baire class equipped with the topology of pointwise convergence. By Lemma \ref{lemma6.1}
$B_r(0)\cap \phi_{v^*}^{-1}(R)\subset \dot V_{p;S}^k\subset\mathcal B_1(Q^d)$ is relatively compact in the latter space and by Theorem F3 of \cite{BFT-78} is  {\em sequentially dense}  in its closure. Hence,
if $v$ belongs to this closure, then there is a
sequence $\{v_j\}_{j\in\N}\subset B_r(0)\cap\phi_{v^*}^{-1}(R)$ pointwise convergent to $v$ on $Q^d$. Therefore 
$\{v_j\}_{j\in\N}$ satisfies the assumptions of the Rosenthal Main theorem \cite{Ro-77} implying that
\begin{equation}\label{equ11.21}
\lim_{j\rightarrow\infty}\int_{Q^d}v_j\,d\mu=\int_{Q^d}v\,d\mu
\end{equation}
for all finite signed Borel measures on $Q^d$. 

In particular, this is true for discrete measures giving rise  by functions in $\hat\ell_1=(U_p^k)^0$. But $(U_p^k)^0$ is dense in $U_p^k$ and by our definition 
the sequence $\{v_j\}_{j\in\N}$ is bounded in $\dot V_{p;S}^k$. Hence, \eqref{equ11.21} implies that
\begin{equation}\label{equ11.22}
\lim_{j\rightarrow\infty}v_j(u)=v(u)\quad {\rm for\ all}\quad u\in U_p^k;
\end{equation}
here we regard $v_j$ and $v$ as bounded linear functionals on $U_p^k$, see Theorem \ref{te2.2.6}. 

This means that $\{v_j\}_{j\in\N}\subset \dot V_{p;S}^k$ weak$^*$ converges to $v$.

To show weak$^*$ closedness 
of $B_r(0)\cap\phi_{v^*}^{-1}(R)$ it remains to prove the following:
\begin{Lm}\label{lem11.9}
If a sequence $\{v_j\}_{j\in\N}\subset B_r(0)\cap\phi_{v^*}^{-1}(R)$ weak$^*$ converges to some $v\in V_p^k$, then $v\in B_r(0)\cap\phi_{v^*}^{-1}(R)$.
\end{Lm}
\begin{proof}
Weak$^*$ convergence of $\{v_j\}_{j\in\N}$ to $v$ implies pointwise convergence of the sequence of functions $\{v_j\circ\tau |_{\mathfrak i^*(\bar{\mathscr B}_p^k)}\}_{j\in\N}$ to the function $v\circ\tau |_{\mathfrak i^*(\bar{\mathscr B}_p^k)}$, see \eqref{equ11.18}. Further, the functions of this sequence are $\mu_{v^*}$-measurable and bounded by $\sup_j\|v_j\|_{V_\kappa}$, see Lemma \ref{lem11.8}. Moreover, by the assumption of Lemma \ref{lem11.9}
\begin{equation}\label{eq11.23}
\sup_j\|v_j\|_{V_p^k}\le r\quad {\rm and}\quad \phi_{v^*}(v_j)\in R,\ \  j\in\N.
\end{equation}
Therefore, the Lebesgue pointwise convergence theorem implies
\[
\lim_{j\to\infty}\phi_{v^*}(v_j)=\int_{\mathfrak i^*(\bar{\mathscr B}_p^k)}\left(\lim_{j\to\infty} v_j\circ\tau\right)d\mu_{v^*}=\int_{\mathfrak i^*(\bar{\mathscr B}_p^k)}v\circ\tau\,d\mu_{v^*}=\phi_{v^*}(v).
\]
Since $R\subset\RR$ is closed, the limit on the left-hand side belongs to $R$, hence,
the limit point $v\in B_r(0)\cap\phi_{v^*}^{-1}(R)$ as required.
\end{proof}

Thus, $\phi_{v^*}$ is a weak$^*$ continuous linear functional from $(V_p^k)^*$.  By the definition of the weak$^*$ topology on $V_p^k=(U_p^k)^*$ every weak$^*$ continuous functional is uniquely determined by an element of $U_p^k$, i.e., for some $u_{v^*}\in U_p^k$
\[
\phi_{v^*}(v)=v(\mathfrak i^*(u_{v^*})),\quad v\in V_p^k.
\]
On the other hand, see \eqref{equ11.20}, for all $v\in\textsc{v}_p^k$,
\[
\phi_{v^*}(v)=v(v^*).
\]
Moreover, since $\mathfrak i^*|_{U_p^k}:U_p^k\rightarrow (\textsc{v}_p^k)^*$ is an isomorphic embedding, see Proposition \ref{prop11.1}\,(a), $\textsc{v}_p^k$ separates points of $U_p^k$. 

These two equalities imply that
\[
v^*=\mathfrak i^*(u_{v^*}).
\]
Thus, every point $v^*\in (\textsc{v}_p^k)^*$ is the image under $\mathfrak i^*$ of some point of $U_p^k$, i.e., $\mathfrak i^*: U_p^k\to (\textsc{v}_p^k)^*$ is a surjection. Moreover, $\mathfrak i^*|_{U_p^k}$ is also an embedding. Hence, $\mathfrak i^*$ is an isomorphism of the Banach spaces $U_p^k$ and $(\textsc{v}_p^k)^*$. 

This completes the proof of Theorem \ref{te2.3.3}.
\end{proof}
\begin{proof}[Proof of Corollary \ref{cor2.3.4}]
The result is the combination of Theorems \ref{te2.3.3} and \ref{te2.2.6}.
\end{proof}


\end{document}